\documentclass[final,onefignum,onetabnum, onethmnum]{siamart220329}

\usepackage{amsmath}
\usepackage{amssymb}
\usepackage{amsfonts}
\usepackage{algorithmic}
\usepackage{graphicx}
\usepackage{textcomp}
\usepackage{xcolor}
\usepackage{bm}
\usepackage{tikz}
\usepackage{tikzscale}
\usetikzlibrary{angles, quotes, intersections}
\usetikzlibrary{arrows,shapes,er,automata,backgrounds,topaths,trees}
\usepackage{physics}
\usepackage{centernot}
\usepackage{enumerate}

\newcommand*{\CopyCounter}[2]{%
  \expandafter\def\csname c@#2\endcsname{\csname c@#1\endcsname}%
  \expandafter\def\csname p@#2\endcsname{\csname p@#1\endcsname}%
  \expandafter\def\csname the#2\endcsname{\csname the#1\endcsname}}
\newcommand{\R}{\mathbb{R}}
\renewcommand{\P}{\mathbb{P}}
\newcommand{\E}{\mathbb{E}}

\newcommand{\J}{\mathcal{J}}
\newcommand{\F}{\mathcal{F}}
\newcommand{\Scal}{\mathcal{S}}
\newcommand{\M}{\mathcal{S}}
\newcommand{\T}{\mathcal{T}}
\newcommand{\U}{\bm{U}}
\newcommand{\V}{\bm{V}}
\newcommand{\Vf}{\mathcal{V}}
\newcommand{\W}{\bm{W}}
\newcommand{\I}{\mathcal{I}}
\newcommand{\N}{\mathcal{N}}
\newcommand{\ba}{\bm{a}}
\newcommand{\e}{\bm{e}}
\newcommand{\bt}{\bm{t}}
\newcommand{\bv}{\bm{v}}
\newcommand{\bw}{\bm{w}}
\newcommand{\bn}{\bm{n}}
\newcommand{\nhat}{\hat{\bn}}
\newcommand{\x}{\bm{x}}
\newcommand{\xtilde}{\tilde{\x}}
\newcommand{\y}{\bm{y}}
\newcommand{\z}{\bm{z}}
\newcommand{\bq}{\bm{q}}
\newcommand{\bal}{\bm{a}_{\lambda}^{}}
\newcommand{\bxi}{\bm{\xi}}
\newcommand{\bphi}{\bm{\phi}}

\newcommand{\bgamma}{\bm{\gamma}}
\newcommand{\boldeta}{\bm{\eta}}
\newcommand{\s}{\bm{s}}

\newcommand{\m}{s}
\newcommand{\domain}{\Omega}
\newcommand{\cdomain}{\bar{\Omega}}
\newcommand{\boundary}{\partial \domain}

\DeclareMathOperator{\interior}{int}

\newif\ifnever\neverfalse

\makeatletter
\def\munderbar#1{\underline{\sbox\tw@{$#1$}\dp\tw@\z@\box\tw@}}
\makeatother

\usepackage{natbib}
 \bibpunct[, ]{[}{]}{,}{n}{}{,}%

\usepackage{blindtext}

\DeclareMathOperator*{\argmin}{argmin}

\newcommand{\Halmos}{\hfill \ensuremath{\Box}}

\renewenvironment{equation*}{\[}{\]\ignorespacesafterend}

\title{Monotone Causality in Opportunistically Stochastic Shortest Path Problems}
\author{Mallory E. Gaspard\thanks{meg375@cornell.edu, Center for Applied Mathematics, Cornell University, Ithaca, NY 14853}
\and Alexander Vladimirsky\thanks{vladimirsky@cornell.edu, Department of Mathematics, Cornell University, Ithaca, NY 14853}}

\headers{Monotone Causality in OSSP Problems}{M.E.~Gaspard \& A.~Vladimirsky}

\begin{document}

\maketitle

\begin{abstract}
When traveling through a graph with an accessible deterministic path to a target,
is it ever preferable to resort to stochastic node-to-node transitions instead?
And if so, what are the conditions guaranteeing that such a stochastic optimal routing policy 
can be computed efficiently?
We aim to answer these questions here by defining a class of Opportunistically
Stochastic Shortest Path (OSSP) problems and deriving sufficient conditions for applicability of
non-iterative label-setting methods.
The usefulness of this framework is demonstrated in two very different contexts:
numerical analysis and autonomous vehicle routing.
We use OSSPs to derive causality conditions for semi-Lagrangian discretizations of anisotropic Hamilton-Jacobi equations.
We also use a Dijkstra-like method to solve OSSPs optimizing the timing and urgency of lane change maneuvers
for an autonomous vehicle navigating road networks with a heterogeneous traffic load. 
\end{abstract}%



%


\section{INTRODUCTION}
Sequential decision making under uncertainty is ubiquitous in science and engineering.
If the duration of the process is not known in advance and the termination results from entering a specific part of state space, 
this is often modeled as a Stochastic Shortest Path (SSP) problem \cite{bertsekas1991analysis}.
SSPs arise naturally in diverse applications including aircraft routing \cite{hong2021anytime} and optimal scheduling in wireless sensor networks \cite{chen2007transmission}.   
Unfortunately, solving most SSPs is computationally expensive and is usually accomplished iteratively.  
In this work, we introduce a subclass of SSPs called \emph{Opportunistically Stochastic Shortest Path} (OSSP) problems, in which 
one can always use deterministic transitions at each stage, but opting for stochastic transitions can often yield a lower cumulative cost.
We derive a number of sufficient conditions on the OSSP's transition cost function to guarantee the applicability of non-iterative label-setting methods. 
Our conditions cover problems with finite and infinite action spaces, make no assumptions about the smoothness of transition cost functions,
and are naturally satisfied by many OSSPs arising in very different applications (e.g., autonomous vehicle routing models and discretizations of
partial differential equations in optimal control theory).

The state space of an SSP is usually identified with nodes of a finite graph, with a separate set of actions available at each of them.  Each chosen action incurs an instantaneous cost and yields a probability distribution over possible successor nodes.  The goal is to choose actions that minimize the expected total cost accumulated by the time we reach the pre-specified target.  The {\em value function} is defined by minimizing the expected cumulative cost starting from each of the $n$
nodes, and it can be recovered by solving a system of $n$ coupled non-linear optimality equations.  But unlike in the case of fixed-horizon Markov Decision Processes, for general SSPs the direction of information flow is not a priori obvious (i.e., the dependence of the nodes/equations usually does not have an obvious tree-like structure).  As a result, iterative methods are generally unavoidable in computing the SSP value function.
This is expensive even in ``causal'' SSPs, for which the convergence is
obtained in finitely many iterations: 
even if we assume that the optimal action at any single node can be found in $O(1)$ operations once the value function is already known at all possible successor nodes,
the iterative process yields the overall computational cost of $O(n^2).$
In deterministic shortest path problems with non-negative transition costs and
bounded node outdegrees, such iterations are avoided by using more efficient {\em label-setting} algorithms: most notably, Dijkstra's \cite{dijkstra1959note} and Dial's \cite{dial1969algorithm}, achieving the computational cost of $O(n \log n)$ and $O(n)$ respectively.  It is thus natural to consider whether these non-iterative methods are applicable to any broad class of SSPs as well.  The previous attempts to answer this question yielded either implicit conditions that could be only verified a posteriori \cite{Bertsekas_DPbook} or a priori verifiable  sufficient conditions that were fairly restrictive (e.g., relying on a very rich set of available actions and smoothness of the cost function) \cite{vladimirsky2008label}.   The OSSPs studied in this paper significantly broaden this class, with simple explicit criteria for ``monotone causality'' (the property that ensures the applicability of label-setting)  and far fewer assumptions on the problem data.

We begin \S \ref{section:ssps_and_ossps} with a quick introduction to general SSPs (and the ``pruning process'', which allows reducing their action sets without altering the value function)  followed by a formal definition of  OSSPs. For the latter, the sufficient conditions for applicability of Dijkstra's and Dial's methods are then developed in \S \ref{section:label_setting_and_mc}.
We then proceed to demonstrate the usefulness of OSSPs in analyzing the computational (causal) properties of semi-Lagrangian discretizations of stationary Hamilton-Jacobi-Bellman (HJB) partial differential equations (PDEs) with different types of anisotropy. The fact that many semi-Lagrangian PDE discretizations can be re-interpreted as SSPs is well-known \cite{KushnerDupuis}.  But our analysis in \S \ref{section:ossp_and_ahj} yields
(a) a simplification and geometric interpretation of prior disparate criteria for applicability of Dijkstra's method and 
(b) new sufficient conditions for the applicability of Dial's method to HJB discretizations.

OSSPs  also provide a particularly attractive modeling framework for discrete problems in which the key choice is between maintaining the status quo or pushing (with some degree of urgency) to enact a specific change.  This corresponds to OSSPs in which every action yields at most two possible successor nodes.
One very natural example of this is a
decision on whether (and when) to change lanes while driving.
In \S \ref{section:lane_change_formulation}, we show how OSSPs can be used in autonomous vehicle routing problems\footnote{
Since this paper is meant to be of interest to multiple research communities, we note that readers primarily interested in AV applications 
can safely skip the second half of \S \ref{ss:OSSP_MC} and the entire \S \ref{section:ossp_and_ahj}.}. 
The standard driving directions provide turn-by-turn instructions for reaching the target.
But our OSSP model also answers finer-scale tactical questions based on heterogeneous traffic conditions: where, when, and with what urgency should a car try to change lanes along the way? 
And what should be done if an attempted lane change fails?   
The success of lane change maneuvers is inherently uncertain, and failure to take this into account during the planning process may result in sub-optimal performance or even an unsafe driving experience.
Our work in that section provides an alternative to a model 
recently introduced by  Jones, Haas-Heger, and van den Berg \cite{jones2022lane}. 
That earlier paper has showed the applicability of Dijkstra's method to a specific SSP-based routing model that allowed ``tentative'' and ``forced'' lane change maneuvers.  
The modeling framework in \S \ref{section:lane_change_formulation} goes much farther, allowing for  
a continuous spectrum of lane change urgency and a variety of cost functions treatable by label-setting methods.

We conclude by discussing the current limitations of the OSSP framework and possible directions for future work in \S \ref{section:conclusions}. 
\section{SSPs and OSSPs}\label{section:ssps_and_ossps}
\subsection{General Stochastic Shortest Path Problems}
Let $X = \{\x_1, \dots, \x_n, \x_{n+1} = \bt\}$ be the set of all states (nodes), where the last one encodes the target (destination). 
A sequence $(\y_k)_{k=0,1,...}$ reflects a possible stochastic path,
where
each $\y_k \in X$ denotes the state of the process at the $k$-th stage (i.e., after $k$ transitions).
We use $A(\x_i)$ to denote the set of actions available at a state $\x_i \in X$ and $A = \bigcup_{\x \in X}A(\x)$ to denote the set of all actions.
Choosing any action $\ba \in A(\x_i)$ incurs a known cost $C(\x_i, \ba)$ and yields a known probability distribution $p(\x_i, \ba, \x_j) = p_{ij}(\ba)$
over the possible successor nodes for the next transition; i.e., if an action $\ba \in A(\y_k)$ is chosen at the $k$-th stage of the process, then
$\P(\y_{k+1} = \x_j) = p(\y_k, \ba, \x_j)$ for all $\x_j \in X.$  

The target node is assumed to be {\em absorbing} and 
no additional cost is incurred after
we reach it; i.e., $p(\bt, \ba, \bt) = 1$ and $C(\bt, \ba) = 0$ for all $\ba \in A(\bt)$. 
The decision maker's goal is to choose actions to minimize the expected cumulative cost up to the target.
This is done by optimizing over the set of control policies.  A function $\mu: X \to  A$ is a {\em control mapping} if $\mu(\x) \in A(\x)$ for all $\x \in X.$
A {\em control policy} is an infinite sequence of control mappings $\pi = (\mu_k)_{k=0,1,...}$ to be used at respective stages of the process.
The expected cost of using a policy $\pi$ starting from any node $\y_0 = \x$ is defined as
\begin{equation}\label{eqn:value_of_policy_ssp}
\J(\x, \pi) = \E \left[ \sum\limits^{\infty}_{k = 0}C \left(\y_k, \mu_k(\y_k) \right) \, \mid \, \y_0 = \x \right],
\end{equation}
and the value function $U(\x)$ denotes the result of minimizing 
the expected cost starting from $\x$:
\begin{equation}\label{eqn:value_fn_ssp}
U(\x) = \inf\limits_\pi \J(\x, \pi).
\end{equation}
With a slight abuse of notation, we also use the symbol $\mu$ to refer to a {\em stationary} policy $(\mu, \mu, \dots )$
that uses the same control mapping $\mu$ at each stage of the process.

 The theory of SSPs was developed to describe a broad class of models, including those with infinite and even continuous action spaces.  
As shown in \cite{bertsekas1991analysis}, the following
five assumptions\footnote{
If the set of available actions is finite, the existence of an optimal  stationary policy can be demonstrated without these additional assumptions \cite{feinberg1992markov}.
}
 guarantee that the infimum in \eqref{eqn:value_fn_ssp} is actually attained and that 
there exists an optimal stationary policy $\mu^*$ which attains this minimum for every starting state $\x \in X$:
\begin{enumerate}[\bfseries ({A}1)]
	\item $A(\x_i)$ is compact for all $\x_i \in X.$\label{item:axi_compact}
	\item $C(\x_i, \ba)$ is lower semicontinuous over $A(\x_i)$  for all $\x_i \in X.$\label{item:c_lsc}
	\item $p_{ij}(\ba)$ is continuous over $A(\x_i)$  for all $\x_i$ and $\x_j$ in $X$.
	\item There exists at least one policy which results in arriving at the target with probability $1$ from every starting state. 
	We refer to policies of this type as \emph{proper} and all others as {\em improper}. 
	\item Every improper policy yields an infinite expected cost starting from at least one node.\label{item:improper_policy_guarantee}
\end{enumerate}

\vspace*{2mm}
\noindent
Throughout the paper, we will also impose a stronger cost-positivity assumption 
\begin{enumerate}[(\bfseries {A}5')]
	\item \label{item:cost_positivity}
$C(\x, \ba) > 0, \qquad\qquad \forall \x \in X \setminus \{\bt\}, \, \ba \in A(\x).$ 
\end{enumerate}
This is a necessary condition for the applicability of label-setting methods, 
which will be discussed in detail in \S \ref{section:label_setting_and_mc}.
We note that (\textbf{A\ref{item:axi_compact}}) and (\textbf{A\ref{item:c_lsc}}) make (\textbf{A\ref{item:improper_policy_guarantee}}) a trivial
consequence of
(\textbf{A5'}).
If $C$ is also bounded from above, the existence of a proper stationary policy implies that $U(\x)$ is finite $\forall \x \in X.$ 

For simplicity of exposition, 
we will further assume the lack of self-transitions on $X \setminus \{\bt\}$; i.e., 
\begin{enumerate}[\bfseries ({A}6)]
\item
$p(\x, \ba, \x) = 0, \qquad \qquad \forall \x \in X \setminus \{\bt\}, \ba \in A(\x).$
\end{enumerate}
This is not really restrictive since self-transitions can be removed without affecting the value function.
Indeed, whenever $p_{ii}(\ba) > 0$ for $\x_i \neq \bt$, one can obtain an SSP with the same value function by switching to
$$
\tilde{C}(\x_i, \ba) \, = \, \frac{C(\x_i, \ba)}{1 - p_{ii}(\ba)};
\qquad \text{and} \qquad
\tilde{p}_{ij}(\ba) \, = \,
\begin{cases}
0 & \text{if } j = i,\\
\frac{p_{ij}(\ba)}{1 - p_{ii}(\ba)}
& \text{if } j \neq i;
\end{cases}
\qquad
\text{for } 
j=1,\cdots,n+1.
$$
Replacing $p(\ba)$ by $\tilde{p}(\ba)$ is an example of ``oblique  projections,'' in which one of the possible successor nodes is excluded
and its probability is redistributed among other successor nodes proportionally.  We will later employ a similar approach in proving causal 
properties of SSP problems.

For notational ease, denote $U(\x_i)$ as $U_i$ and $A(\x_i)$ as $A_i$. 
The value function can be computed by solving a coupled system of dynamic programming equations
\begin{align}\label{eqn:value_fn_vi}
\nonumber
U_i & = \min_{\ba \in A_i} \left\{C(\x_i, \ba) + \sum\limits^{n+1}_{j = 1}p_{ij}(\ba)U_j \right\}, & \text{ for } i = 1, \ldots n;\\
U_{n+1} & =
0,
\end{align}
and the optimal control mapping $\mu^*$ 
can be obtained by using any action from the $\argmin$ in  \eqref{eqn:value_fn_vi}.
In vector notation, 
$\U = \begin{bmatrix}
U_1, & \dots ,& U_n 
\end{bmatrix}^T$ 
is a fixed point of the 
operator $\T:\R^n \to \R^n$ defined componentwise as
\begin{equation} \label{eq:T_and_F}
(\T\W)_i = \min_{\ba \in A_i} \F_i(\ba, \W),
\quad \text{ where }
\F_i(\ba, \W) = C(\x_i, \ba) + \sum\limits^{n}_{j = 1}p_{ij}(\ba)W_j.
\end{equation}
This $\F_i(\ba, \W)$
represents the expected cost-to-go from $\x_i$ when using $\ba \in A_i$ for the first transition and under the assumption that the remaining expected cost-to-go
is encoded in $\W \in \R^n.$ 

It is often useful to consider ``stochastic connectivity'' of various nodes in an SSP.
Focusing on any specific
$\x_i$,  we define the set of potential successor nodes $\I(\ba) = \{\x_j \in X \hspace{0.2cm} | \hspace{0.2cm} p_{ij}(\ba) > 0\}$ 
for each action $\ba \in A_i$. The set of \emph{all} possible successor nodes from $\x_i$ is then defined as $\N(\x_i) =  \cup_{\ba \in A_i} \I(\ba) =  \{\x_j \in X \hspace{0.2cm} | \hspace{0.2cm} \exists \ba \in A_i \text{ s.t. } p_{ij}(\ba) > 0\}$. 
The SSP's \emph{transition digraph} $G$ can then be defined as a network representation of all possible transitions to successor nodes from each $\x_i \in X$ under all available actions in $A_i$. 
We provide two simple examples in Figure \ref{fig:ssp_examples}, with each possible  $\x_i \rightarrow \x_j$ transition indicated by a dashed arrow.

Since this perspective aggregates transitions possible under {\em all} available actions, a transition digraph captures only the global (topological) structure of the SSP, which is typically insufficient to infer the optimal policy.
But the practicality of a computational method based on fixed-point iterations to determine $\U$ depends heavily on whether $G$ is acyclic (as in Figure \ref{fig:ssp_examples}(a)).  Even for large acyclic SSPs, it is easy to compute $\U$ efficiently via an iterative process.  But if $G$ contains cycles (as in 
Figure \ref{fig:ssp_examples}(b)), the number of iterations can be prohibitive, particularly when $n$ is large.
In such cases, faster non-iterative algorithms described in \S \ref{section:label_setting_and_mc} provide an attractive alternative, whenever they are applicable.
Unfortunately, their applicability to each specific SSP problem can be hard to verify a priori.
In this paper, we show how this can be done by checking simple, explicit conditions on $C(\x, \ba)$ for a broad subclass of SSPs defined in \S \ref{ss:ossp}.


\begin{figure}[hhhh]
\center{
$
\begin{array}{cc}
\begin{tikzpicture}[->,>=stealth',shorten >=1pt,auto,node distance=3cm, scale=2,
                    semithick]
  \tikzstyle{every state}=[draw, shape=circle, inner sep=0mm, minimum size = 6mm]

  \node[state]         (T)	at (1.65,0)					{$\bt$};
    \node[state]         (x1)	at (-1.55,0.6)			{$\x_1$};
   \node[state]         (x2)	at (-1.55, -0.6)			{$\x_2$};
  \node[state]         (x3)	at (-0.5, 1)				{$\x_3$};
    \node[state]         (x4)	at (-0.5, 0.4)			{$\x_4$};
  \node[state]         (x5)	at (-0.5, -0.4)			{$\x_5$};
  \node[state]         (x6)	at (-0.5, -1)			{$\x_6$};
  \node[state]         (x7)	at (0.625, 0.5)			{$\x_7$};
  \node[state]         (x8)	at (0.625, 0)			{$\x_8$};
\node[state]         (x9)	at (0.625, -0.5)			{$\x_9$};

  \path (x1)	edge	[dashed]									(x3);
    \path (x1)	edge	[dashed]									(x4);
        \path (x1)	edge	[dashed]									(x5);
  \path (x2)	edge	[dashed]					node {}				(x5);
    \path (x2)	edge	[dashed]					node {}				(x6);
  \path (x2)	edge	[dashed]					node {}				(x4);
  \path (x3)	edge	[dashed]									(x7);
    \path (x3)	edge	[dashed]									(x8);
   \path (x4)	edge	[dashed]									(x9);
    \path (x4)	edge	[dashed]									(x7);
    \path (x7)	edge	[dashed]									(T);
      \path (x6)	edge	[dashed]									(x9);
      \path (x6)	edge	[dashed]									(x8);
     \path (x9)	edge	[dashed]									(T);
      \path (x5)	edge	[dashed]									(x7);
       \path (x5)	edge	[dashed]									(x9);
             \path (x8)	edge	[dashed]									(T);
  
\end{tikzpicture}
\hspace*{2cm}
&
\begin{tikzpicture}[->,>=stealth',shorten >=1pt,auto,node distance=3cm, scale=2,
                    semithick]
  \tikzstyle{every state}=[draw, shape=circle, inner sep=0mm, minimum size = 6mm]

  \node[state]         (x6)	at (0,0.13)					{$\x_6$};
  \node[state]         (x3)	at (1,0.13)					{$\x_3$};
  \node[state]         (x2)	at (0.5, 1.1)				{$\x_2$};
  \node[state]         (x1)	at (-0.5, 1.1)			{$\x_1$};
  \node[state]         (x5)	at (-1,0.13)					{$\x_5$};
  \node[state]         (x4)	at (0, -0.866)			{$\x_4$};
  
    \node[state]         (x7)	at (1.75, 1.1)			{$\x_7$};
        \node[state]         (x8)	at (1.75, 0.13)			{$\x_8$};
            \node[state]         (x9)	at (1.75, -0.866)			{$\x_9$};
  
  \node[state]         (T)	at (2.5, 0.13)			{$\bt$};

  \path (x1)	edge	[bend right=10, dashed]		node {}				(x2)
				edge	[bend right=10, dashed]		node {}				(x5);
				
  \path (x2)	edge	[bend right=10, dashed]		node {}				(x1)
				edge	[bend right=10, dashed]		node {}				(x3);
				
  \path (x3)	edge	[bend right=10, dashed]		node {}				(x4)
				edge	[bend right=10, dashed]		node {}				(x2);
  \path (x4)	edge	[bend right=10, dashed]		node {}				(x5)
				edge	[bend right=10, dashed]		node {}				(x3);
  \path (x5)	edge	[bend right=10, dashed]		node {}				(x1)
				edge	[bend right=10, dashed]		node {}				(x4);
       \path (x6)	edge	[dashed]									(x1);
              \path (x6)	edge	[dashed]									(x2);
                     \path (x6)	edge	[dashed]									(x3);
                            \path (x6)	edge	[dashed]									(x4);
                                   \path (x6)	edge	[dashed]									(x5);
 
  \path (x3)	edge	[dashed]									(x8);
    \path (x8)	edge	[dashed]									(T);
        \path (x7)	edge	[dashed]									(T);
            \path (x9)	edge	[dashed]									(T);
    
     \path (x7)	edge	[bend right=10, dashed]		node {}				(x8);
          \path (x8)	edge	[bend right=10, dashed]		node {}				(x7);
               \path (x9)	edge	[bend right=10, dashed]		node {}				(x8);
          \path (x8)	edge	[bend right=10, dashed]		node {}				(x9);
            \path (x4)	edge	[bend right=10, dashed]		node {}				(x9);
            \path (x9)	edge	[bend right=10, dashed]		node {}				(x4);
             \path (x2)	edge	[bend right=10, dashed]		node {}				(x7);
            \path (x7)	edge	[bend right=10, dashed]		node {}				(x2);

\end{tikzpicture}  \\
\mbox{\footnotesize (a)} & \mbox{\footnotesize(b)} 
\end{array}
$
}
\caption{
{\footnotesize
Transition digraphs for two simple SSP examples. The target node is marked by $\bt$, and the possibility of a transition from node $\x_i$ to a successor node $\x_j$ is indicated by a dashed arrow. 
}
}
\label{fig:ssp_examples}
\end{figure}

\subsection{Randomized Policies, Dominated Actions, and Convexified Cost}
The above definition of control policies selects state-dependent actions deterministically.
It is well known that extending SSPs to {\em randomized control policies} (which select actions probabilistically) does not lower the value function.  
Nevertheless, 
this generalization is very useful in 
proving the properties of OSSPs in \S \ref{section:label_setting_and_mc}. 
In addition, it can be also used to justify a reduction of the action space in general SSP problems.

In this section, we will refer to all actions available in the original SSP as ``pure'' to distinguish them from the ``relaxed''/mixed actions 
allowed here.
Suppose the set of pure actions $A_i$ is finite and $\lambda = (\lambda_1, \ldots, \lambda_{|A_i|})$ is some probability distribution over it.
Preparing for a transition from $\y_k = \x_i,$
one can implement a mixed (or {\em relaxed}) action $\bal$ by selecting each $\ba_r \in A_i$ with probability $\lambda_r.$
This requires extending the definition in \eqref{eq:T_and_F} to use
$$
\F_i(\bal, \W) = \sum\limits^{|A_i|}_{r = 1} \lambda_r \F_i(\ba_r, \W)
$$
and minimizing  over all such $\lambda$'s when computing $(\T\W)_i$.  
Of course, this also includes all pure actions since $\lambda$ might be prescribing a single $\ba_r \in A_i$ with probability one.
We will refer to the problem based on such relaxed actions as rSSP.
But it's easy to see that the value function of this rSSP is actually the same as in the original SSP since 
\begin{equation} 
\label{eq:best_pure_action}
\ba_* \in \argmin_{\ba \in A_i} \F_i(\ba, \W)
\qquad \implies \qquad
\F_i(\ba_*, \W) \leq \F_i(\bal, \W) \qquad
\end{equation}
for all probability distributions $\lambda$.
Thus, there exists an optimal stationary policy for rSSP that uses only pure actions.
 (The same conclusions also hold for a general compact $A_i$, except that $\lambda$ would need to be a probability measure.) 

We will say that two pure actions $\ba, \tilde{\ba} \in A_i$ are {\em transition-equivalent} if  $p_{ij}(\ba) = p_{ij}(\tilde{\ba}).$
Clearly, if they are transition-equivalent and $C(\x_i, \ba) \geq C(\x_i, \tilde{\ba}),$ then $\ba$ can be removed from $A_i$ without affecting the value function. 
This pruning idea can also be extended to use relaxed actions.

\begin{definition}[Transition-equivalent Relaxed Actions]\label{defn:trans_equiv}
Given a set of pure actions $\{\ba_1, ... , \ba_l\} \subset A_i$
and a probability distribution over them $\lambda = (\lambda_1, \ldots, \lambda_l),$ 
we say that the corresponding relaxed action $\bal$ is {\em transition-equivalent} to 
a pure action $\ba \in A_i$ if
\begin{equation} \label{eq:dominated_match_distrib}
p_{ij}(\ba) \; = \; p_{ij}(\bal) \; = \; \sum\limits^{l}_{r = 1} \lambda_r p_{ij}(\ba_r), \qquad \forall \x_j \in X.
\end{equation}
We will also use $\Lambda(\ba)$ to denote the set of all relaxed actions that are transition-equivalent to $\ba$.
\end{definition}

\begin{definition}[Dominated, Replaceable, and Useful Actions]\label{defn:dom_point}
An action $\ba \in A_i$ is {\em probabilistically dominated} if there exists 
a transition-equivalent relaxed action $\bal \in \Lambda(\ba)$ such that
the resulting expected cost of the first transition is lower with $\bal$; i.e.,
\begin{equation} \label{eq:dominated_improve_cost}
C(\x_i, \ba) \; >  \; C(\x_i, \bal) \, = \, \sum\limits^{l}_{r = 1} \lambda_r C(\x_i, \ba_r).
\end{equation}

We will say that an action $\ba \in A_i$ is {\em replaceable} if there exists 
a set of actions $\{\ba_1, ... , \ba_l\} \subset A_i$
and a probability distribution over them $\lambda = (\lambda_1, \ldots, \lambda_l)$ such that
\begin{itemize}
\item none of the $\ba_r$'s are individually transition-equivalent to $\ba$ but
$\bal \in \Lambda(\ba);$ 
\item the expected transition cost is the same; i.e., $C(\x_i, \ba) \; =  \; C(\x_i, \bal) \, = \, \sum\limits^{l}_{r = 1} \lambda_r C(\x_i, \ba_r).$
\end{itemize}

We will say that a pure action $\ba \in A_i$ is (potentially) {\em useful} in an SSP if it is not probabilistically dominated or replaceable,
with $A_i^u$ denoting the set of all such actions.
In view of \eqref{eq:best_pure_action}, whenever $\ba \in A_i$ is probabilistically dominated or replaceable, there will always exist another 
 $\tilde{\ba} \in A_i$ which is at least as good of a choice at $\x_i$.  
\end{definition}
\begin{observation}\label{obs:ssp_useless_non_optimum}  $\quad$
	Suppose $\ba \in A_i$ is not useful.  Then, for every $\W \in \R^n$, there exists another action $\tilde{\ba} \in A_i \backslash \{\ba\}$ such that
	$\F_i(\ba, \W) \geq \F_i(\tilde{\ba}, \W).$ So, removing this $\ba$ from $A_i$ will not affect $U_i$ regardless of the $U_j$ values 
        at any potential successor nodes. The SSP based on $A_i^u$'s instead of $A_i$'s will have exactly the same value function.
\end{observation}
\proof{Proof: }
Suppose $\ba$ is probabilistically dominated (or replaceable) by some $\bal$
and $\tilde{\ba} \in \argmin_{\ba_r \in \mathcal{A}(\bal)}  \F_i(\ba_r, \W),$ where $\mathcal{A}(\bal) \subset A_i$ is the set of pure actions selected by $\bal$ with positive probability.
Then
$$
\F_i(\ba, \W) \, = \, C(\x_i, \ba) + \sum\limits^{n}_{j = 1}p_{ij}(\ba)W_j
\,  \geq \,  
 \sum\limits^{l}_{r = 1} \lambda_r C(\x_i, \ba_r) + \sum\limits^{n}_{j = 1}p_{ij}(\ba)W_j
 \, = \,
\sum\limits^{l}_{r = 1} \lambda_r \F_i(\ba_r, \W) 
 \geq \F_i(\tilde{\ba}, \W).
$$
\Halmos
\endproof

\vspace*{5mm}

Since the above definitions are fully based on the cost $C(\x_i, \cdot),$ the set of ``potentially useful'' pure actions is defined locally for each $\x_i$ and independently of 
the global properties of the SSP.
The corresponding pruning of the action space can be performed as a pre-processing step to speed up the computation of the value function, which might be particularly 
worthwhile if one needs to solve a family of related problems; e.g., differing only by the identity of the target $\bt \in X.$ 

{\em Geometric Interpretation:}
Once $A_i$ no longer includes any transition-equivalent elements, it is often convenient to identify the remaining pure actions with their corresponding probability distributions over successor nodes.
Suppose $\N(\x_i) = \{\z_1, \dots, \z_m\} \subset X$ is a set of all possible immediate successors of $\x_i$ and let 
$$\Xi_m = \left\{ \bxi = (\xi_1,  ... , \xi_m) \mid \ \xi_1 + \ldots + \xi_m = 1 \text{ and } \xi_j \geq 0 \hspace{0.2cm} \text{for }j = 1, \dots, m\right\}$$ 
denote the $(m-1)$-dimensional probability simplex. 
We can then take $A_i$ to be a subset of $\Xi_m$ and  $\bxi \in A_i$ implies $p(\x_i, \bxi, \z_j) = \xi_j$ for $j = 1, \ldots, m.$
The cost of the next transition is then $C(\x_i, \bxi)$ though we will also use the notation  $C(\bxi)$ 
or $C(\xi_1, ..., \xi_m)$ 
whenever $\x_i \in X$ is clear from the context.
Allowing the use of relaxed actions in rSSP is equivalent to switching from $A_i$ to its convex hull $A_i^{co} \subset \Xi_m.$
Using the best relaxed action in each transition equivalence class essentially replaces $C$ with its lower convex envelope
\begin{equation}
\label{eq:convexified_cost}
\check{C}(\bxi) = \min_{\bal \in \Lambda(\bxi)} C(\bal) \qquad \forall \bxi \in A_i^{co},
\end{equation}
where $\bal$ is specified by some choice of $\lambda \in \Xi_m$ and 
$\{\bxi_1, \ldots, \bxi_m\} \subset A_i$
that satisfy $\xi_j = \sum^{m}_{r = 1} \lambda_r \xi_{r,j}$ for $j = 1, \ldots, m$ while
$C(\bal) = \sum^{m}_{r = 1} \lambda_r C(\bxi_r).$  
The resulting $\check{C}$ is convex\footnote{
The same ``cost convexification'' approach has also been used in \cite[Lemma 6]{fainberg1976}.  
} and thus continuous on the interior of $A_i^{co}$ (and lower semi-continuous on $\partial A_i^{co}$).
The pruning outlined above is possible because the minimum in \eqref{eqn:value_fn_vi} can only be attained at $\bxi \in A_i$ if $C(\bxi) = \check{C}(\bxi).$ 
If an optimal $\bxi$ is replaceable, that same minimum will be also achieved by another (non-replaceable) action.
An action $\bxi \in A_i$ is ``useful'' if and only if $\left(\bxi, C(\bxi) \right)$ is an extreme point of the epigraph of $\check{C}.$

For a concrete example, consider an SSP whose transition digraph is represented in Figure \ref{fig:ssp_examples}(b). 
Focusing on $\x_1$, there are two possible successor nodes, $\x_2$ and $\x_5$.
In this case, we can identify $\bxi = \left( 1-p,  p \right) \in \Xi_2$ with $p \in [0,1],$ 
which makes it convenient to visualize $K(p) = C\left( 1-p,  p \right)$ and its convexified version  $\check{K}(p) = \check{C}\left( 1-p,  p \right).$ 
In many SSPs, it is natural to consider a continuous spectrum of actions; in other applications, the set of actions is discrete and finite.
To illustrate the most general case, we will suppose here that $\x_1$ has a ``hybrid'' action set; e.g.,  
$A_1 =  [\underline{\ba}, \bar{\ba}] \cup \{\ba_1, \dots, \ba_4\}.$  
We will further 
suppose that some of these actions are transition-equivalent (yielding the same probabilities of transition to $\x_2$ and $\x_5$),
which makes the graph of $K(p)$ multivalued and possibly discontinuous. 
Figure \ref{fig:cost_pruning} illustrates the deterministic and probabilistic pruning process at $\x_1$. 
Once the transition-equivalent pure actions are removed (crossed in red in Figure \ref{fig:cost_pruning}(a)), 
we can perform additional pruning based on relaxed actions (Figure \ref{fig:cost_pruning}(b)). 
Although the value function remains unchanged during the convexification procedure, it is clear now that the optimal action at $\x_1$ will come from the set $A^u_1$.

\begin{figure*}[t]
\centering
$
\arraycolsep=1pt\def\arraystretch{0.1}
\begin{array}{cc}
\includegraphics[width = 0.5\textwidth]{det_prune_cost.tikz} &
\includegraphics[width = 0.5\textwidth]{prob_prune_cost.tikz}\\[-5pt]
\mbox{\footnotesize(a)} & \mbox{\footnotesize(b)}
\end{array}
$
\caption{
Example of deterministic (panel (a)) and probabilistic (panel (b)) pruning of
the action set
when $m = 2$. 
Depicted action choices at the node $\x_1$ in Figure \ref{fig:ssp_examples}(b), with transition probabilities $\P(\x_1 \rightarrow \x_5 \mid \ba) = p(\ba)$ and 
$\P(\x_1 \rightarrow \x_2 \mid \ba) = 1-p(\ba).$
The specific cost function $K(p(\ba))$ is chosen for the sake of illustration only.  Many application-motivated examples of cost functions are considered in \S\ref{section:ossp_and_ahj} and \S\ref{section:lane_change_formulation}.
Panel (a): The transition probabilities associated with each action in $A_1 = [\underline{\ba}, \bar{\ba}] \cup \{\ba_1, \ldots, \ba_4\}$ are indicated in green on the $p$-axis. During the deterministic pruning process, all actions $\ba \in A_1$ that are transition-equivalent to another action $\tilde{\ba} \in A_1$ with $K(p(\ba)) \ge K(p(\tilde{\ba}))$ are removed along with the corresponding portion of the transition cost curve (crossed in red).  Panel (b): Following the deterministic pruning, we obtain $\check{K}(p)$ by taking the lower convex envelope of $K$ (solid and dashed purple curve), and the resulting $A^{co}_1$ is indicated in orange above the $p$-axis. Replaceable actions $\ba'$ and $\ba_1$ are also removed at this stage. The transition probabilities associated with the remaining useful pure actions $A_1^u$ are indicated on the $p$-axis in solid purple, and their corresponding transition cost values are also marked in solid purple along $\check{K}$. \label{fig:cost_pruning}}
\end{figure*}

\subsection{Opportunistically Stochastic Shortest Path Problems}
\label{ss:ossp}
We now focus on a specific subclass of SSPs, in which every stochastically realizable path is also realizable using only deterministic actions (i.e., any pure action $\ba \in A_i$ for which $\I(\ba)$ is a singleton).  
Later, we will demonstrate their usefulness in approximating solutions of continuous optimal control problems (\S \ref{section:ossp_and_ahj}) and 
in routing of autonomous vehicles (\S \ref{section:lane_change_formulation}).

\begin{definition}[OSSP]\label{defn:ossp}
We will refer to an SSP
 as {\em Opportunistically Stochastic} (OSSP)
if
\begin{equation}
\label{eq:ossp_def}
\exists \ba \in A_i 
\text{ s.t. } 
p_{ij}(\ba) > 0 
\qquad \implies \qquad
\exists \tilde{\ba}  \in A_i 
\text{ s.t. } 
p_{ij}(\tilde{\ba}) = 1
\qquad 
\text{holds for all $i$ and $j$.}  
\end{equation}
\end{definition}

If the set of potential successor nodes  $\N(\x_i)$ has $m$ elements,
then the post-pruning $A_i^u$ can be identified with a subset of $\Xi_m,$ which in an OSSP will include all vertices of this probability simplex,
encoding deterministic transitions.  Correspondingly, the randomized-policy version of this problem (rOSSP), would have the action set $A_i^{co} = \Xi_m.$

We note that the OSSPs are also closely related to Multimode Stochastic Shortest Path (MSSP) problems introduced in \cite{vladimirsky2008label}.
MSSPs are based on a more stringent requirement that, whenever some nodes
are possible successors of $\x_i$ under any specific action,  
then {\em every} probability distribution 
over those same successor nodes must be realizable. 
To be more precise,
the {\em modes} in MSSPs are 
defined as
subsets of $X = \{\x_1, \dots, \x_n, \bt\},$ describing possible successor nodes under each particular class of actions. 
Each non-terminal node $\x \in X  \backslash \{\bt\}$ has its own set of such modes, denoted $\M(\x),$  and each mode $\m \in \M(\x)$ can be specified by enumerating the nodes in it; i.e., 
$\m = \{\z^{\m}_1, \dots, \z^{\m}_{|\m|}\} \subset \N(\x) \subset \, 
X^{} \setminus \{\x\}.$   MSSPs are built on two critical assumptions:
\begin{enumerate}[\bfseries ({M}1)]
\item
any available action has all of its successor nodes in one of these modes; i.e.,
$$
\forall \x_i \in X, \ba \in A(\x_i)   \quad \exists \m \in \M(\x_i) \text{ s.t. }  \I(\ba) \subset \m;
$$
\item\label{assumption:mssp_Actions}
every probability distribution over the set of successor nodes in any mode is achievable via some action; i.e.,
$$
\forall \x_i \in X \backslash \{\bt\}, \, \m \in \M(\x_i), \text{ and }  \bxi \in \Xi_{|\m|} \quad \exists \ba \in A(\x_i) \text{ s.t. }  p(\x_i, \ba, \z_r) = \xi_r, \quad \forall \z_r \in \m,
$$
 where $\Xi_{|\m|}$ is the $(|\m|-1)$-dimensional probability simplex.  \label{item:every_pdf_in_MSSP}
 \end{enumerate}
 
Thus, a decision at each stage of an MSSP is twofold: a deterministic choice of a mode and a choice of a probability distribution over the possible successor nodes in that mode.  This makes it natural to represent the actions as $\ba = (\m, \bxi) \in A(\x)$, where $\bxi \in \Xi_{|\m|},$
 and then re-write the optimality equation \eqref{eqn:value_fn_vi} as
\begin{align}
	\nonumber
	V^{\m}(\x) &= \min\limits_{\bxi \in \Xi_{|\m|}} \left\{C^{\m}(\x, \bxi) + \sum\limits^{|\m|}_{r = 1} \xi_r U(\z^{\m}_{r}) \right\},\\
	\nonumber
	U(\x) &= \min\limits_{\m \in \M(\x)} V^{\m}(\x),  \hspace*{45mm} \forall \x \in X \backslash \{\bt\},\\
	\label{eq:MSSP}
	U(\bt) &= 0.
\end{align}

OSSP is a generalization of MSSP since the assumption \textbf{(M\ref{item:every_pdf_in_MSSP})} implies 
the availability of deterministic transitions from $\x_i$ to each 
$\z_r \in \m \in \M(\x_i).$  
Unless an MSSP is fully deterministic, it must also include a continuum of actions spanning entire probability simplexes corresponding to each mode,
making all modes ``perfect'' in this sense.  
This clearly does not have to be the case in general OSSPs, for which each $A^u_i$ might be a proper (or even finite) subset of the probability simplex.

On the other hand, for any given OSSP, it is easy to construct an MSSP with the same value function; e.g., 
by using the convexified cost $\check{C}$ instead of the original $C$.  
To have a single mode at every node, one can define $\M(\x) = \{ \m \}$ with $\m = \N(\x) = \bigcup\limits_{\ba \in A(\x)} \I(\ba),$ and then use 
the convexified cost $\check{C}$ instead of the original $C$.
The action set $A(\x)$ can be then extended by adding all relaxed actions to satisfy \textbf{(M\ref{item:every_pdf_in_MSSP})}.

It is worth noting that in the definition of $\check{C}$, the result will not change if the relaxed actions are formed by only using pure actions $\tilde{\ba}$
such that  $\I(\tilde{\ba}) \subset \I(\ba).$ (In all other cases, $\bal \not \in \Lambda(\ba).$)
This suggests a natural approach to introduce 
a minimal set of
``imperfect modes'' in OSSPs.  We will say that $\m \subset \N(\x)$ is an imperfect mode if $\exists \ba \in A(\x)$ s.t.
$\m = \I(\ba)$ but $\not \exists \tilde{\ba} \in A(\x)$ s.t. $\m$ is a proper subset of $\I(\tilde{\ba}).$ (This ensures that a mode is never a subset of another mode.)
An action $\ba \in A(\x)$ is associated with a mode $\m \in \M(\x)$ 
if $\I(\ba) \subset \m.$ 
Thus, every action is associated with at least one mode and might well be associated with multiple modes simultaneously.

To provide a concrete example, we focus on the node $\x_6$ in Figure \ref{fig:ssp_examples}(b) and 
consider $A(\x_6) = \{\ba_1, ..., \ba_{10}\},$ where the first five actions are deterministic
($\I(\ba_j) = \{\x_j\}, \, j=1, ...,5$)  while the rest of them are not:
$$
\I(\ba_6) = \{\x_1, \x_2\}, \quad
\I(\ba_7) = \{\x_2, \x_3\}, \quad 
\I(\ba_8) = \{\x_3, \x_4\}, \quad 
\I(\ba_9) = \{\x_1, \x_2, \x_3\}, \quad
\I(\ba_{10}) = \{\x_1, \x_2, \x_4\}.
$$
Based on the above definition, $\N(\x_6) = \{\x_1, ..., \x_5\}$ and $\M(\x) = \{\m_1, ..., \m_4\}$ with
$\m_1 = \{\x_1, \x_2, \x_3\}$ (associated actions: $\ba_1, \ba_2, \ba_3, \ba_6, \ba_7, \ba_9$),
$\m_2 = \{\x_1, \x_2, \x_4\}$ (associated actions: $\ba_1, \ba_2, \ba_4, \ba_6, \ba_{10}$),
$\m_3 = \{\x_3, \x_4\}$ (associated actions: $\ba_3, \ba_4, \ba_8$),
and $\m_4 = \{\x_5\}$ (associated action: $\ba_5$).
Compared to MSSPs, these modes are ``imperfect'' since \textbf{(M\ref{item:every_pdf_in_MSSP})} is violated.  
But each mode can be ``perfected" without changing the value function by including relaxed actions 
and using $\check{C}$ to define their cost.

\section{Label-Setting and Monotone Causality in OSSPs}\label{section:label_setting_and_mc}

Since the value function $\U = [U_1, \ldots, U_n]^T$ is a fixed point of the operator $\T$ defined in \eqref{eq:T_and_F},
the simplest approach for computing $\U$ is through {\em value iterations} (VI), in which one starts
with an initial guess $\W^0 \in \R^n$ and updates it iteratively by taking $\W^{k+1} = \T \W^k.$
The operator $\T$ is generally not a contraction unless all stationary policies are known to be proper \cite{bertsekas1991analysis}.
But Tsitsiklis and Bertsekas have shown that
assumptions (\textbf{A1})-(\textbf{A5}) guarantee that this fixed point is unique and $\W^k \rightarrow \U$ as $k \rightarrow \infty$ regardless of $\W^0$ \cite{bertsekas1991analysis}.
Unfortunately, for a general SSP, this does not necessarily occur after any finite number of iterations.
The convergence can be slow and the VI algorithm is often impractical for large problems.
Finding computationally efficient alternatives to these basic value iterations has been an active research area in the last several decades.
One obvious direction is to use a Gauss-Seidel relaxation (GSVI), where the components of  $\W^{k+1}$ are computed sequentially and the previously 
computed components are immediately used in computing the remaining ones; i.e., 
$$
\W^{k+1}_i \; = \; 
\min\limits_{\ba \in A_i} \F_i \left( \ba, \, [\W^{k+1}_1, \ldots, \W^{k+1}_{i-1}, \W^k_i, \ldots, \W^k_{n} ]^T \right).
$$
But the efficiency of this approach is heavily dependent on the ordering imposed on the nodes/states in $X$
even if the convergence is achieved in finitely many steps.  E.g., for the very simple example in Figure \ref{fig:ssp_examples}(a), GSVI will require 3 iterations to converge with the default node ordering
though only 1 iteration would be needed if the ordering were reversed. More generally, this difference in the number of needed iterations can be as high as $O(n).$ 
To address this,
some implementations of GSVI alternate through several orderings that are likely to be efficient \cite{BoueDupuis}.
In others, the so called {\em label-correcting} methods, the next node/state to be updated is determined dynamically, based on the current vector of tentative labels (or values) $\W$
and the history of updates up till that point \cite{Bellman_DP_book, pape1974implementation, Bertsekas_SLF, Bertsekas_LLL, glover1986threshold}.  Some of these methods were originally developed for the deterministic shortest path setting, but have since been adapted to SSPs as well, particularly when used to discretize Hamilton-Jacobi equations; e.g., \cite{PolyBerTsi, BorRasch, Renzi}.

Other notable approaches include {\em topological value iterations} (in which the topological structure imposed on $X$ by $A$ is taken into account to attempt splitting an SSP into a sequence of causally ordered subproblems) \cite{DaiGoldsmith2007, DaiGoldsmith2009}, policy iterations (in which the goal is to produce an improving sequence of stationary policies $(\mu_k)$, with $\U$ recovered as a limit of their respective policy value vectors) \cite{howard1960dynamic}, and hybrid policy-value iterations \cite{van1976set}. 

Here we are primarily interested in a subclass of SSPs for which VI actually does converge in at most $n$ iterations regardless of $\W^0.$
(Typically, this happens when at least one element of $\W^k$ becomes newly converged in each iteration.) 
But even assuming the minimization in \eqref{eq:T_and_F} can be performed in $O(1)$ operations,
the computational cost of a single value iteration is $O(n),$ which yields the overall cost of $O(n^2)$ up to convergence.
In such SSPs, the same worst-case $O(n^2)$ computational cost also holds for all VI variants mentioned above,
while the policy iterations would still require an infinite number of steps for full convergence
if the action sets $A_i$ are infinite.
In contrast, our goal is to obtain $\U$ in $O(n \log n)$ or $O(n)$ operations, bounding the number of ``approximate $U_i$'' updates based on that node's stochastic outdegree 
(i.e., $|\N(\x_i)|$) rather than on the overall number of non-terminal nodes $n \gg \max_{\x_i}  |\N(\x_i)|.$
For the deterministic shortest path (SP) problems, this cost reduction is accomplished by the classical {\em label-setting} methods reviewed in \S \ref{ss:SP_label_setting}.
We then describe the prior (implicit) conditions for their general SSP-applicability in \S \ref{ss:SSP_causality},
and derive new (explicit) conditions for their OSSP-applicability in \S \ref{ss:OSSP_MC}.

\subsection{Label-Setting Methods for Deterministic SP Problems}\label{ss:SP_label_setting}
Classical shortest/cheapest path problems on directed graphs can be interpreted as a subclass of SSPs in which all actions yield deterministic transitions.
In that setting,  $C_{ij} = C(\x_i, \x_j)$ encodes the cost of a direct $(\x_i \rightarrow \x_j)$ transition, with the set of potential successor nodes denoted
$N_i = N(\x_i) = \{ \x_j \in X \mid C_{ij} < \infty \}.$  We will further assume that $C_{ij} \geq \delta \geq 0.$ 
The value function satisfies the following dynamic programming equations:
\begin{equation}
\label{eq:determ_DP}
	U_{i} = \min_{\x_j \in N(\x_i)} \left\{C_{ij} + U_j \right\}, \hspace{0.3cm} \forall \x_i \in X \setminus \{\bt\},
\end{equation}
with $U(\bt) = 0$.
Efficient algorithms for solving \eqref{eq:determ_DP} are well-known and covered in standard references (e.g., \cite{ahuja1993, Bertsekas_DPbook}), but we provide a quick overview for the sake of completeness.

The key idea of label-setting methods is to re-order iterations so that the tentative value of each node $\x_i$ is updated at most $|N_i|$ times.
This yields a significant performance advantage over value iterations, particularly when the outdegree of nodes is bounded and relatively small, $\kappa = \max_i |N_i| \ll n.$
Since the number of updates per node is bounded by $\kappa,$ such methods are also often considered {\em noniterative}.
The non-negativity of transition costs implies a {\em monotone causality} property: 
\begin{equation}
\label{cond:causality_determ}
U_i \text{ may depend on $U_j$ only if } \; U_i \geq U_j.
\end{equation} 
If all $C_{ij} > 0$ and nodes could be ordered monotonically based on $U_i$'s, GSVI
would converge in a single iteration. 
But since such an ordering is not known in advance, it has to be obtained at run-time.  
Dijkstra's algorithm \cite{dijkstra1959note} accomplishes this by recomputing tentative labels $V_i$ while maintaining two lists of nodes: the list of known/converged nodes $K = \{\x_i \mid V_i = U_i \text{ is confirmed} \}$ and the list of  tentative nodes $L = X \setminus K$ for which $V_i \geq U_i$.  The basic version of this algorithm starts with $V(\bt) = U(\bt) =0, 
\; K = \{\bt\},$ and $V_i = + \infty$ for all $\x_i \in L = X \setminus \{\bt\}.$  At each stage of Dijkstra's method, the node $\bar{\x}$ with the current smallest finite value in $L$ is moved to $K$,
and other tentative nodes adjacent to it (i.e., all $\x_i \in L$ s.t. $\bar{\x} \in N_i$) are updated by setting
$V_i := \min\limits_{\x_j \in N_i \cap K} \left\{C_{ij} + V_j \right\},$ or, more efficiently, by using $V_i := \min \left( V_i, \, C(\x_i, \bar{\x}) + V(\bar{\x}) \right).$
The process terminates once $L = \emptyset$ or once all values remaining in $L$ are infinite.  (In the latter case, the nodes remaining in $L$ are not path-connected to $\bt$.)
The method terminates after at most $\kappa n$ updates and yields the correct $U_i$ values even if $\delta = 0.$  To identify $\bar{\x}$ efficiently, $L$ is typically implemented using heap-sort data structures, resulting in an overall computational cost of $O(n \log n).$

If $\delta >0,$ 
the nodes whose tentative values are less than $\delta$ apart cannot depend on each other.
This makes it possible to use another label-setting method due to Dial \cite{dial1969algorithm},
in which the tentative values are not sorted but instead placed into ``buckets'' of width $\delta.$
At each stage of the algorithm, all tentative nodes in the current smallest bucket are moved to $K$ simultaneously,
and all tentative labels of nodes adjacent to them are updated, switching these adjacent nodes to new buckets if necessary. 
Since inserting to and deleting from a bucket can be performed in $O(1)$ time, the overall computational complexity of Dial's method
is $O(n)$.   Unlike Dijkstra's algorithm, Dial's method is also naturally parallelizable.  Nevertheless, which of them is more efficient in practice depends on the properties of the SP problem.  

\subsection{Applicability of Noniterative Methods to general SSPs}
\label{ss:SSP_causality}

Avoiding iterative methods in general SSPs is usually
harder.
Given any stationary policy $\mu,$ one can define the corresponding dependency digraph $G_{\mu}$ 
by starting with the nodes $X$ and adding arcs wherever direct transitions are possible under $\mu$.
(I.e., adding an arc from $\x_i$ to $\x_j$ whenever $p_{ij} \left( \mu(\x_i) \right) >0.$)
We will call an SSP {\em causal} if there exists an optimal stationary policy $\mu^*$
whose dependency digraph $G_{\mu^*}^{}$ is acyclic.
As noted in \cite[Volume 2, \S2.3.3]{Bertsekas_DPbook},
value iterations on any such causal SSP will converge to $\U$ after at most $n$ iterations.
If one were to use a reverse topological ordering of $G_{\mu^*}^{}$ to sort $X$,
the 
GSVI algorithm
would converge in a single iteration.
This property is trivially satisfied by special {\em explicitly causal} SSPs, in which the dependency digraph of {\em every} stationary policy is acyclic
(and thus the transition digraph $G$ is as well),
making it easy to obtain the full $\U$ with only $n$ node value updates regardless of the properties of $C$.  (This mirrors the fact that Dijkstra's and Dial's methods
are not really needed to find deterministic shortest paths on acyclic graphs.)

Turning to a broader class of SSPs that are not explicitly causal, 
a natural question to consider is whether the label-setting methods might be applicable.
Following Bertsekas \cite{Bertsekas_DPbook}, 
we will say that an optimal stationary policy $\mu^*$ is {\em consistently improving} if, for all $\x_i \neq \bt,$
\begin{equation}
\label{eq:improving_policy}
p_{ij}(\mu^*(\x_i)) > 0 \implies 
U_i > U_j.
\end{equation}
This property is a stochastic analogue  of \eqref{cond:causality_determ}. 
If such $\mu^*$ is known to exist, \eqref{eq:improving_policy} guarantees that an SSP-version of Dijkstra's method 
will correctly produce $\U$ as its output\footnote{It might seem natural to pose a ``$U_i \geq U_j$'' condition in \eqref{eq:improving_policy}, to fully mirror \eqref{cond:causality_determ}.  But for non-deterministic actions, this condition turns out to be too weak and may result in Dijkstra's failing to converge to $\U.$  See \cite[Figure 1]{vladimirsky2008label} for a simple example.}.  In terms of implementation, the only difference from the deterministic SP case described in \S\ref{ss:SP_label_setting},
is that, once $\bar{\x}$ is moved to $K$, we would need to update all $\x_i$'s such that $\bar{\x} \in \N(\x_i)$ by using
\begin{equation}
\label{eq:Dijkstra_update_SSP}
V_i \; := \; \min \left( V_i, 
\;
\min_{\substack{\ba \in A_i \text{ s.t. }\\ \bar{\x} \in \I(\ba) \subset K}} \F_i(\ba, \V) \right).
\end{equation}

Continuing this approach, we will say that an optimal stationary policy $\mu^*$ is {\em consistently $\delta$-improving} if, for some $\delta>0$ and all $\x_i \neq \bt,$
\begin{equation}
\label{eq:delta_improving_policy}
p_{ij}(\mu^*(\x_i)) > 0 \implies U_i \geq U_j + \delta.
\end{equation}
The existence of such a $\mu^*$ similarly guarantees the applicability of Dial's method\footnote{
Consistently $\delta$-improving policies were first defined in \cite{vladimirsky2008label} with a strict inequality in \eqref{eq:delta_improving_policy} and $\delta \geq 0.$
Since the buckets are non-overlapping and have a positive width $\delta,$ we know that $(U_j + \delta)$ is never in the same bucket as $U_j.$ 
Thus, the above definition is more suitable to guarantee the applicability of Dial's method.}  with buckets of width $\delta.$ 
If we assume that the minimization in  \eqref{eq:Dijkstra_update_SSP} can be performed in $O(1)$ operations 
and  stochastic outdegrees $|\N(\x_i)|$ are bounded by a constant $\kappa$, then
the overall computational cost will scale the same way as in the deterministic case: $O(n \log n)$ for Dijkstra's and $O(n)$ for Dial's method.
 
\begin{definition}
We will say that an
SSP is {\em monotone causal} (MC) if at least one of its optimal stationary policies is consistently improving.
We will refer to an SSP as {\em monotone $\delta$-causal} ($\delta$-MC) if that optimal stationary policy is consistently $\delta$-improving.
\end{definition}

Unfortunately, the above criteria based on \eqref{eq:improving_policy} and \eqref{eq:delta_improving_policy}
are implicit and hard to apply in practice since none of the optimal policies are known before $\U$ is computed.
Thus, it is natural to search for sufficient
($\delta$-)MC
criteria that can be verified a priori and locally (e.g., based on the cost function properties at each node / state) without considering the global structure of the SSP.
The first two such SSP examples were found by Tsitsiklis \cite{tsitsiklis1994efficient, tsitsiklis1995efficient}, who used them to apply Dijkstra's and Dial's methods to regular-grid semi-Lagrangian 
discretizations of isotropic optimal control problems.  Related and more general methods were also developed by others (see the detailed discussion in \S \ref{section:ossp_and_ahj}),
and the Multimode SSPs were then introduced in \cite{vladimirsky2008label} -- both to provide an overall structure for discussing the label-setting in PDE 
discretizations and to show that the monotone ($\delta$-)causality can be guaranteed even for some SSPs unrelated to optimal control problems.
The idea was to pose the sufficient conditions based on the cost $C$ only, and have them verified on a mode-per-mode, node-per-node basis.
Given the close relationship between the OSSPs and MSSPs, it is worth highlighting the main difference between the criteria in \cite{vladimirsky2008label}
and those developed below.  
In addition to a restrictive condition {\bf (M2)}
described in \S\ref{ss:ossp}, the criteria  
developed in \cite{vladimirsky2008label}
have either required $C$ to be concave or posed conditions on its derivatives, while we allow $C$ to be merely lower semi-continuous. 

\subsection{Guaranteeing Monotone Causality in OSSPs} 
\label{ss:OSSP_MC}

Consider an OSSP satisfying assumptions {\bf (A1)}-{\bf (A4)}, {\bf (A5')}, {\bf (A6)}, and \eqref{eq:ossp_def}.
Suppose an action $\ba \in A_i$ results in a list of possible successor nodes $\I(\ba) = \{ \z_1, ..., \z_m \} \subset X \setminus \{\x_i\},$
with the respective transition probabilities\footnote{Of course, these $m$, $\z_j$'s, and $\xi_j$'s are always action-dependent.
But for the sake of readability, we do not indicate this in the notation whenever a specific action is clear from the context.}
$p(\x_i, \ba, \z_j) = \xi_j > 0$ for $j = 1, \ldots, m$
and $\bxi = (\xi_1, ... , \xi_m) \in \interior(\Xi_m).$  
Assuming this $\ba \in A_i$ is not deterministic (i.e., $m>1$) and choosing any specific $r \in \{1, ..., m\},$ we define 
$\bgamma_r = (\gamma_{r,1}, ..., \gamma_{r,m})$ to be an oblique (proportional) projection of  $\bxi$ as follows 
			$$
			\gamma_{r,j} \; = \; 
			\begin{cases}
			0, & \text{ if } j =r;\\
			\xi_j / (1-\xi_r), & \text{ otherwise.} 
			\end{cases}
			$$
We note that once we omit the $r$-th zero component of $\bgamma_r,$ it can be thought of as a point in $\interior \left( \Xi_{(m-1)} \right).$ 
The set $\Lambda(\bgamma_r)$ of relaxed actions 
transition-equivalent to $\bgamma_r$ is nonempty, 
and the convexified cost of $\bgamma_r$ can be defined as $\check{C}(\bgamma_r) = \min_{\bal \in \Lambda(\bgamma_r)} C(\bal).$		

\begin{theorem}[Sufficient Monotone $\delta$-Causality Condition in OSSPs]\label{theorem:gen_caus_condition_improved} \mbox{ }\\
Suppose there exists a $\delta \ge 0$ such that, for all  $\x_i \in X \backslash \{\bt\}$ and $\ba \in A_i^u,$
\begin{itemize}
\item 
if $\ba$ is deterministic, then $C(\x_i, \ba) \geq \delta$;
\item
if $\ba$ is not deterministic, then
			\begin{equation}
			\label{eq:improved_MC_cond}
			C(\x_i, \ba) \geq (1-\xi_r) \check{C}(\bgamma_r) + \xi_r \delta, \qquad \forall r \in \{1, ..., |\I(\ba)| \}.
			\end{equation}
\end{itemize}
If these conditions are satisfied, this OSSP is  monotone causal and Dijkstra's method is applicable.
If $\delta >0,$ the OSSP is monotone $\delta$-causal and Dial's method with buckets of width $\delta$ is also applicable.
\end{theorem}
\begin{proof}{Proof: }
Define the set of optimal actions $A^*_i = \argmin\limits_{\ba \in A_i} \left\{C(\x_i, \ba) + \sum\limits^{n+1}_{j = 1}p_{ij}(\ba)U_j \right\},$ and suppose that $\ba^* \in A^*_i$ yields the minimal number of successor nodes among them.  (I.e., $|\I(\ba^{\sharp})| \geq |\I(\ba^*)|$ for all $\ba^{\sharp} \in A^*_i.$)
If $\ba^*$ is deterministic, then $\I(\ba^*) = \{\z_1\}$ and $U_i$ depends only on $U(\z_1)$ while $U_i = C(\x_i, \ba^*) + U(\z_1) \geq \delta + U(\z_1).$
For $\delta=0$ this inequality is strict due to ({\bf A5'}).

If $\ba^*$ is not deterministic, its use leads to one of the successor nodes $\{\z_1, ..., \z_m\}$ with respective probabilities $\xi_r >0,$ 
and $U_i$ depends on all of the $U(\z_r)$ with $r \in \{1, ..., m\}.$  Choosing any specific $r$ and combining the optimality of $\ba^*$ 
with \eqref{eq:improved_MC_cond}, we see that
\begin{equation}
\label{eq:step1}
U_i \; = \; C(\x_i, \ba^*) + \sum\limits^{m}_{j = 1} \xi_j U(\z_j)
\; \geq \; (1-\xi_r) \check{C}(\bgamma_r) + \xi_r \delta + \sum\limits^m_{j = 1} \xi_j U(\z_j).
\end{equation}
On the other hand, using $\ba^*$ must be strictly better that using any relaxed action $\bal \in \Lambda(\bgamma_r)$. Otherwise, 
the best pure action $\tilde{\ba}$ among those potentially selected by $\bal$ would have to be at least 
as good as $\ba^*$, leading to $\tilde{\ba} \in A^*_i$ with $|\I(\tilde{\ba})| \leq m-1,$ which is a contradiction.
Thus, we have $U_i < \F_i(\bal, \U)$ for all $\bal \in \Lambda(\bgamma_r),$ which implies
\begin{equation}
\label{eq:step2}
U_i \; < \; \check{C}(\bgamma_r) + \sum_{j \neq r} \gamma_{r,j} U(\z_j)
 \; = \; \check{C}(\bgamma_r) + \frac{1}{(1-\xi_r)} \sum_{j \neq r}\xi_j U(\z_j).
\end{equation}
Multiplying both sides by $(1-\xi_r)$ and rearranging,
\begin{equation}\label{eq:step3}
U_i \; < \; (1-\xi_r) \check{C}(\bgamma_r) + \sum_{j \neq r}\xi_j U(\z_j) + \xi_r U_i.
\end{equation}
Combining \eqref{eq:step1} and \eqref{eq:step3}, canceling the terms present on both sides of the inequality,
and dividing by $\xi_r > 0,$ we obtain
$U_i \;  >  \; \delta + U(\z_r),$
which proves the monotone ($\delta$-)causality of this OSSP.
\Halmos
\end{proof}

The above condition is explicit and sharp in the sense outlined in Theorem \ref{thm:mc_condition_is_sharp},
but it requires checking the convexified cost of 
multiple oblique projections for each non-deterministic action.
We now provide another sufficient condition, which is more restrictive but easier to verify.

\begin{theorem}[Simplified Monotone $\delta$-Causality Condition in OSSPs]\label{theorem:gen_caus_condition} \mbox{}\\
Suppose there exists a $\delta \geq 0$ such that, for all $\x_i \in X \backslash \{\bt\}$, $\ba \in A_i^u,$ and every $r \in \{1, ..., |\I(\ba)| \},$
			\begin{equation}
				\label{eq:original_MC_cond}
				C(\x_i, \ba) \; \geq \; \sum^{m}_{j = 1, j \neq r} \xi_j C_j  + \xi_r \delta,
			\end{equation} 
where $\I(\ba) = \{ \z_1, ..., \z_m \}, \, p(\x_i, \ba, \z_j) = \xi_j > 0,$ and each $C_j$ encodes the cost of a deterministic action $\e_j \in A_i$ transitioning to $\z_j.$
If these conditions are satisfied, this OSSP is monotone \mbox{($\delta$-)}causal.
\end{theorem}
\begin{proof}{Proof: }
If $\ba$ is deterministic and leads to $\z_r$, then \eqref{eq:original_MC_cond} yields $C_r = C(\x_i, \ba) \geq \delta.$
If $\ba$ is not deterministic, then all of these $\e_j$'s  exist by the definition of OSSPs, and
a relaxed action $\bal$ selecting each $\e_j$ with probability $\xi_j / (1 - \xi_r)$ is transition-equivalent to $\bgamma_r.$
Thus, $C(\bal) = 
\frac{1}{(1 - \xi_r)} \sum^{m}_{j = 1, j \neq r} \xi_j C_j  \geq \check{C}(\bgamma_r)$ and 
\eqref{eq:original_MC_cond} implies \eqref{eq:improved_MC_cond}.
\Halmos

Since it is simple and instructive, we also include an alternative/direct proof that does not use
oblique projections.  Suppose $\ba \in A_i$ is optimal.  If it is deterministic and leads to $\z_r$, then $\xi_r = 1$ and \eqref{eq:original_MC_cond} implies
$U(\x_i) = C_r + U(\z_r) \geq \delta + U(\z_r).$ 

If that $\ba$ is not deterministic and \eqref{eq:original_MC_cond} holds with $\xi_r \in (0,1),$ 
\begin{align}
\label{eq:originalMC_step1}
U(\x_i) \; = \; C(\x_i, \ba) + \sum_{j=1}^m \xi_j U(\z_j)
\; \geq \; 
\sum^{m}_{j = 1, j \neq r} \xi_j \left(C_j  + U(\z_j) \right) + \xi_r \left(\delta + U(\z_r) \right).
\end{align}
From the definition of the value function, $\left(C_j + U(\z_j)\right) \geq U(\x_i)$ for all $j$.
Combining with \eqref{eq:originalMC_step1},
$$
U(\x_i) 
\geq  
\hspace*{-3mm}
\sum^{m}_{j = 1, j \neq r} \xi_j U(\x_i) + \xi_r \left(\delta + U(\z_r) \right)
=   
(1- \xi_r) U(\x_i) + \xi_r \left(\delta + U(\z_r) \right)
\quad \Longrightarrow \quad
U(\x_i) \geq  \delta + U(\z_r).
$$
\Halmos
\endproof
\end{proof}

\begin{figure}[t]
\centering
$
\arraycolsep=1pt\def\arraystretch{0.1}
\begin{array}{cc}
\includegraphics[width = 0.5\textwidth]{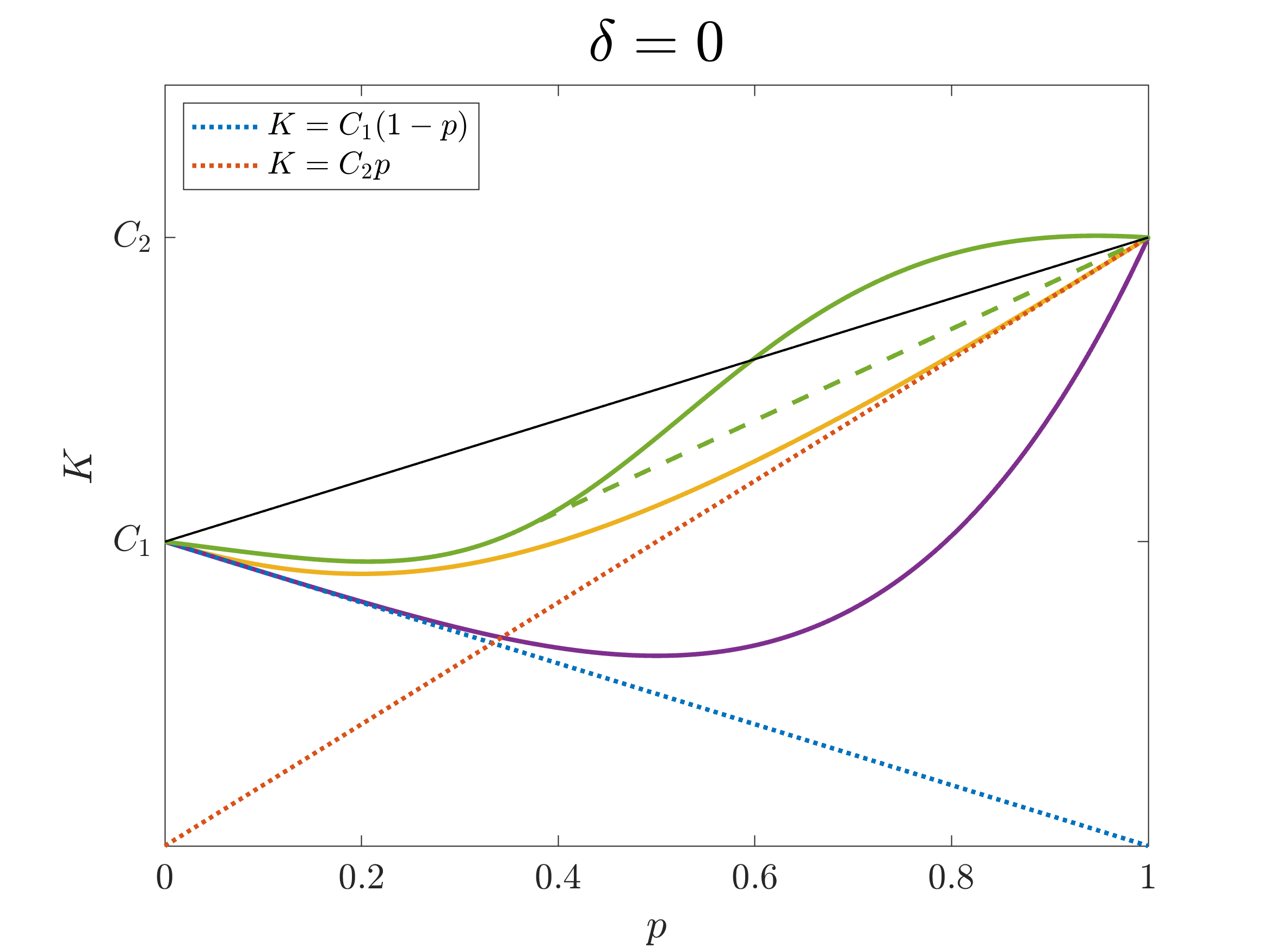} &
\includegraphics[width = 0.5\textwidth]{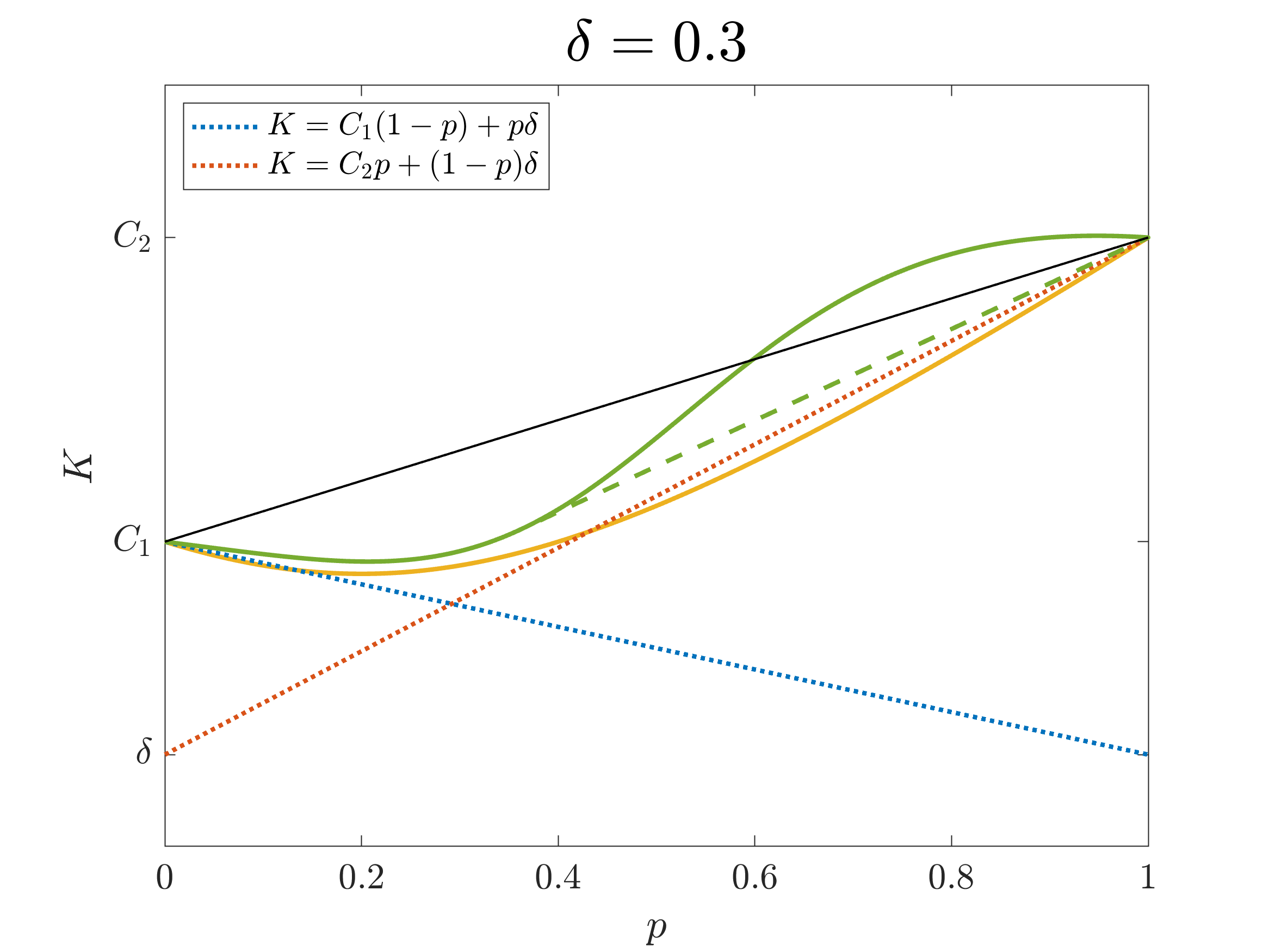} \\[3pt]
\mbox{\footnotesize(a)} & \mbox{\footnotesize(b)}
\end{array}
$
\caption{The geometric interpretation of condition \eqref{eq:original_MC_cond} for $m = 2$ with $\delta = 0$ (panel (a)) and $\delta = 0.3$ (panel (b)) for several examples of transition cost functions. 
The solid purple and solid gold curves are smooth and convex. The solid green curve is smooth and nonconvex, and the green-dashed curve is its convexified version. In all cases, $K(p)$ is monotone ($\delta$-)causal provided that the curve stays entirely above the two restriction lines on the interval $[0, 1]$. 
In (a), the purple curve violates the condition \eqref{eq:original_MC_cond} for $r = 1$ even with $\delta = 0$; so, it cannot be monotone causal. 
While the other two curves satisfy \eqref{eq:original_MC_cond} with $\delta = 0$, only the nonconvex curve (along with its convexified version) is monotone $\delta$-causal for $\delta \leq 0.3$. The smooth gold convex curve can only be monotone causal (with $\delta = 0$), as the restriction lines imposed by \eqref{eq:original_MC_cond}  coincide with its tangent lines at $p = 0$ and $p = 1$. \label{fig:mc_geometric_illustration}}
\end{figure}

\begin{remark}
\label{rem:difference_in_MC_conditions}
Note that the two criteria \eqref{eq:improved_MC_cond} and  \eqref{eq:original_MC_cond} are equivalent when $m=2.$ 
The latter also has a simple geometric interpretation illustrated in Figure \ref{fig:mc_geometric_illustration}  with $\bxi = (1-p, p) \in \Xi_2$
and the assumption that all actions $\ba \in A_i$ lead to the same two successor states $\z_1$ and $\z_2.$
The oblique projections of any $\bxi \in \interior (\Xi_2)$ will correspond to deterministic transitions to $\z_1$ and $\z_2.$
This OSSP is monotone ($\delta$-)causal if
the graph of $K(p) = C(1-p,p)$ does not venture below the two straight lines
$K = (1-p) C_1 + \delta p$ and $K = p C_2 + \delta (1-p).$

For $m > 2,$ the two criteria are only equivalent if $\sum^{m}_{j = 1, j \neq r} \xi_j C_j  = (1 - \xi_r) \check{C}(\bgamma_r)$ for each $r$,
which is only true if $\check{C}$ is linear on $\partial \Xi_m$ (or, equivalently, if $A_i^u \cap \partial \Xi_m$ contains only deterministic actions).
Unfortunately, this is often not the case for OSSPs arising in practical applications.  
Consider, for example, the case where $m = 3, \, A_i = \Xi_3,$ and the cost function is 
$
C(\bxi) = \sqrt{\xi_1^2 + \xi_2^2 + \xi_3^2}.
$
Theorem  \ref{theorem:gen_caus_condition_improved} guarantees its monotone causality:
if $\xi_r > 0$ then  
$\check{C}(\bgamma_r) = C(\bgamma_r)  = (1-\xi_r)^{-1} \sqrt{ \sum_{j \neq r} \xi_j^2}.$
Thus, $\left( C(\bxi) \right)^2  \, -  \, \left( (1-\xi_r)C(\bgamma_r ) \right)^2 = \xi_r^2 > 0,$  which verifies that \eqref{eq:improved_MC_cond} holds for $\delta = 0$.
On the other hand, the simpler criterion from Theorem \ref{theorem:gen_caus_condition} is not satisfied by this $C$
since \eqref{eq:original_MC_cond} requires
$\sqrt{\xi_1^2 + \xi_2^2 + \xi_3^2} \ge \sum_{j \neq r} \xi_j,$
which is violated on a large part of $\Xi_3,$ including at $\bxi = (\frac{1}{3}, \frac{1}{3}, \frac{1}{3}).$
\end{remark}

\vspace*{3mm}

For an OSSP already known to be monotone causal, another natural question to ask is about the largest $\delta \geq 0$ for which it is monotone $\delta$-causal.
(If that maximal $\delta_{\star}$ is zero, the problem can be only treated by Dijkstra's method.  Otherwise, $\delta_{\star} > 0 $ specifies the the largest bucket-width usable in Dial's method.)
With $m=2,$ the answer has a simple geometric interpretation.

\begin{proposition}[Maximum Allowable $\delta$ when $m = 2$]\label{cor:max_delta_m2}
Suppose that all non-deterministic actions in $A_i$ lead to $\{\z_1, \z_2\}$ and satisfy \eqref{eq:original_MC_cond} for $\delta = 0$.
Letting $\e_1, \e_2 \in A_i$ be the deterministic actions such that $p(\x_i, \e_j, \z_j) = 1,$ we will use the notation $C_j = C(\x_i, \e_j).$\\ 
The largest $\delta$ for which the MC conditions are satisfied is then $\delta_{\star} = \min( \delta_1, \delta_2 ),$ where
\begin{equation}
\label{eq:delta_restrict}
\delta_1 =  \min\limits_{\ba \in A_i \setminus \{e_2\}} \frac{C(\x_i, \ba) - p(\x_i, \ba, \z_2)C_2}{p(\x_i, \ba, \z_1)}
\qquad \text{and} \qquad
\delta_2  = \min\limits_{\ba \in A_i \setminus \{e_1\}} \frac{C(\x_i, \ba) - p(\x_i, \ba, \z_1)C_1}{p(\x_i, \ba, \z_2)} 
\end{equation}
Furthermore, if $A_i = \Xi_2$ while $C(\bxi) = C(\x_i, \bxi)$ is convex on $\Xi_2$
and differentiable at $\e_1=(1,0)$ and $\e_2=(0,1),$
then 
\begin{equation}
\label{eq:delta_restrict_convex}
\delta_1 = C(\e_2) + \frac{\partial C}{\partial \xi_1}(\e_2) - \frac{\partial C}{\partial \xi_2}(\e_2) 
\qquad \text{and} \qquad
\delta_2  = C(\e_1) - \frac{\partial C}{\partial \xi_1}(\e_1) + \frac{\partial C}{\partial \xi_2}(\e_1)
\end{equation}
\end{proposition}
\begin{proof}{Proof:}
Since \eqref{eq:original_MC_cond} already holds with $\delta = 0$ for all non-deterministic $\ba \in A_i,$
the ratios minimized in \ref{eq:delta_restrict} are always non-negative.
Based on assumptions {\bf (A1)}-{\bf (A3)}, {\bf (A5')}, the minima are attained  and guarantee
that \eqref{eq:original_MC_cond} is satisfied for each $r=1,2$ with $\delta = \delta_r.$ Figure \ref{fig:max_deltas}(a) illustrates this idea and the restriction lines corresponding to $\delta_{\star}$ for an example of a monotone causal OSSP with $|A_i| = 6$.

If $A_i = \Xi_2,$ each action can be again identified with its probability distribution $\bxi=
(1-p,p)$ for $p\in[0,1].$
If $C(\bxi)$ is convex then so is $K(p) = C(1-p,p).$ 
So, the inequality \eqref{eq:original_MC_cond} for $r=1$ can be rewritten as
$K(p) \geq \delta + p (K(1) - \delta),$ which is equivalent to $K'(1) \leq K(1) - \delta$ due to convexity.
(I.e., this is the requirement for the tangent at $p=1$ to lie no lower than the $\delta$-dependent restriction line; see the solid green curve in Figure \ref{fig:max_deltas}(b).)
Thus, the largest $\delta$ that can guarantee this is $\delta_1 =  K(1) - K'(1).$  The same argument for $r=2$ results in 
$\delta_2 =  K(0) + K'(0),$ and the chain rule yields \eqref{eq:delta_restrict_convex}. \Halmos
\end{proof}

\begin{figure}[hhh]
\centering
$
\arraycolsep=1pt\def\arraystretch{0.1}
\hspace*{-1mm}
\begin{array}{cc}
\includegraphics[width = 0.5\textwidth]{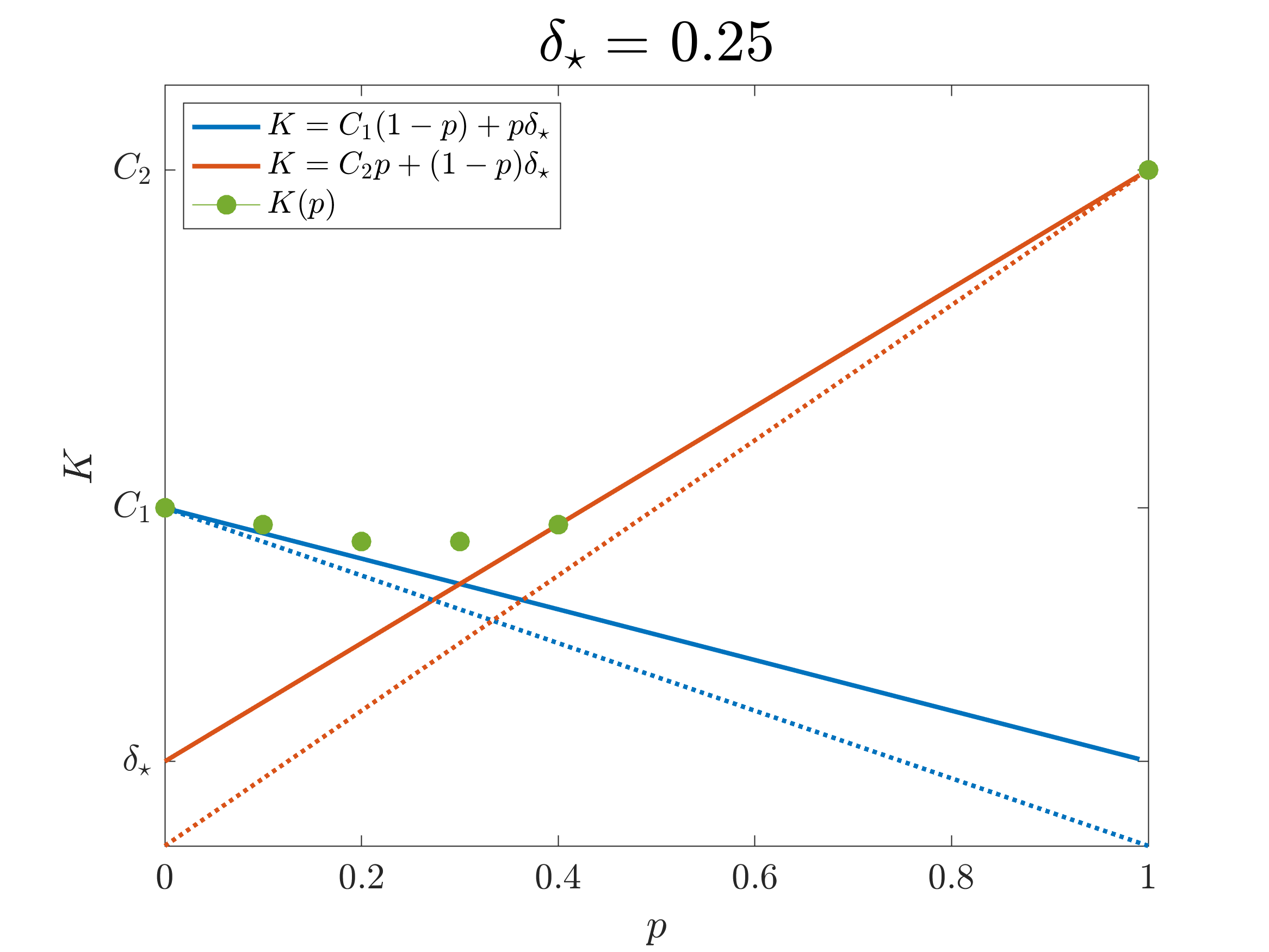} &
\includegraphics[width = 0.5\textwidth]{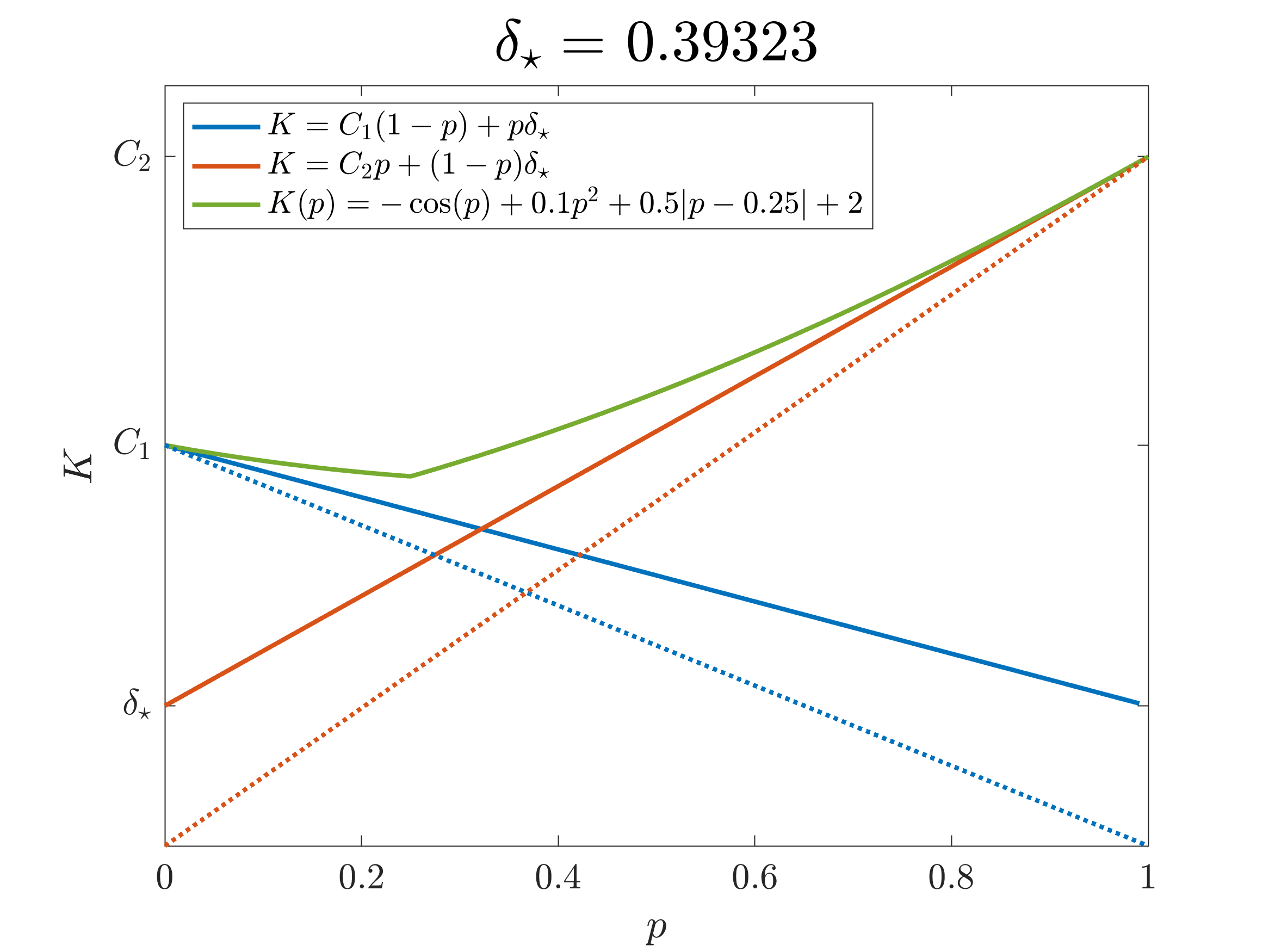} \\[3pt]
\mbox{\footnotesize(a)} & \mbox{\footnotesize(b)} 
\end{array}
$
\caption{Illustration of the geometric interpretation of $\delta_{\star}$ for monotone $\delta$-causality when $m = 2$ for panel (a): a monotone-causal OSSP where $|A_i|$ is finite and the transition cost function is piecewise-continuous and convex, and panel (b): a monotone-causal OSSP where $A_i = \Xi_2$ and the transition cost function is convex and continuous, yet not everywhere differentiable. In panel (a), $|A_i| = 6$ and $\delta_{\star}$ is easily computed via an application of \eqref{eq:delta_restrict}.  In panel (b), while $K(p)$ is not differentiable at $p = 0.25$, it is differentiable at $p = 0$ and $p = 1$, so the conditions of Proposition \ref{cor:max_delta_m2} are still satisfied. Thus, $\delta_{\star}$ follows immediately from \eqref{eq:delta_restrict_convex}. \label{fig:max_deltas}}
\end{figure}

The ($\delta$-)MC criterion in Theorem \ref{theorem:gen_caus_condition_improved} 
is local (i.e., stated using the cost function and actions at each node separately) and does not require any information about the global topological structure of OSSP. 
As a result, it is sufficient but not necessary: for some OSSPs, a label-setting method might compute the value function correctly even if the conditions of Theorem \ref{theorem:gen_caus_condition_improved} are violated.  But these conditions are ``sharp'' in the sense that, whenever they are violated, there will exist a (possibly different) OSSP with the same local structure, for which label-setting methods would compute the value function incorrectly.

\begin{theorem}[Sharpness of MC Condition]\label{thm:mc_condition_is_sharp}
Suppose that the condition \eqref{eq:improved_MC_cond} of Theorem \ref{theorem:gen_caus_condition_improved} is violated for some 
specific $\delta \geq 0, \, \x_i \neq \bt,$   non-deterministic action $\ba,$ and $r \in \{1, ...., m=|\I(\ba)|\}.$  

Then there exists a non-monotone $\delta$-causal OSSP in which 
\begin{enumerate}
\item
$\ba \in A_i,$ 
with $\I(\ba) = \{\z_1, ..., \z_m\} \in X \setminus \{\x_i\}$ and $p(\x_i, \ba, \z_j) = \xi_j > 0$ for $j=1,...,m;$
\item
$\exists \tilde{\ba} \in A_i,$ with $\I(\tilde{\ba}) = \I(\ba) \setminus \{\z_r\}, \, C(\x_i, \tilde{\ba})  = \check{C}(\bgamma_r),$
and\\ 
$
p(\x_i, \tilde{\ba}, \z_j) = \gamma_{r,j} = 
\begin{cases}
	0, & \text{ if } j =r\\
	\xi_j / (1-\xi_r), & \text{ otherwise} 
\end{cases}
$
\qquad \qquad
for $j=1,...,m;$
\item
\begin{equation}
\label{eq:cond_violated}
C(\x_i, \ba) \; < \; (1-\xi_r) C(\x_i, \tilde{\ba}) + \xi_r \delta;
\end{equation}
\item
$\ba$ is the unique optimal action at $\x_i$;
\item
$U(\x_i) \; < \; U(\z_r) + \delta.$
\end{enumerate}
For $\delta = 0,$ this implies that any label-setting method will produce a wrong answer for this OSSP.
($U_i$ will be computed incorrectly since $\z_r$ will be still in $L$ by the time $V_i$ is last updated.) 
For $\delta > 0,$ this means that a wrong answer would be produced by Dial's method with the bucket-width $\delta.$
\end{theorem}
\begin{proof}{Proof:}
To construct the simplest such OSSP example, we will assume that
there is only one action $\hat{\ba}$ available at every
$\x \not \in  
\{\x_i, \bt\},$  and that $\hat{\ba}$
leads to an immediate termination deterministically;  i.e., $p(\x, \hat{\ba}, \bt) = 1.$ 
This essentially allows us to prescribe any values to $\x \in X \backslash \{\x_i, \bt\}$ since $U(\x) = C(\x, \hat{\ba}).$ 
(The resulting OSSP will be thus explicitly causal, but this does not imply the monotone causality.)
At $\x_i,$ we will consider the minimal relevant set of actions $A_i =  \{\ba, \tilde{\ba}, \e_1, ..., \e_m\},$
where $p(\x_i, \e_j, \z_j) = 1,$  with costs $C(\x_i, \ba)$ and  $C(\x_i, \tilde{\ba})$ already prescribed above
and satisfying  \eqref{eq:cond_violated}, while all $C(\x_i, \e_j)$ costs will be prescribed later.

The key quantity of interest in this proof is $B = U(\z_r)  - \sum\limits^{m}_{j = 1, j \neq r} \gamma_{r,j} U(\z_j).$
To ensure that $\ba$ is optimal and $\tilde{\ba}$ is not, we must have 
$$
U(\x_i) = C(\x_i, \ba ) + \sum\limits^{m}_{j = 1} \xi_j U(\z_j) \; < \; C(\x_i, \tilde{\ba} ) + \sum\limits^{m}_{j = 1, j \neq r} \gamma_{r,j} U(\z_j),
$$
which is equivalent to $B < \overline{B} = \frac{C(\x_i, \tilde{\ba}) - C(\x_i, \ba)}{\xi_r}.$
The condition $U(\x_i) \; < \; U(\z_r) + \delta$ is equivalent to $B > \underline{B} = \frac{C(\x_i, \ba) - \delta}{1-\xi_r}.$
These bounds on $B$ can be 
satisfied simultaneously   
since \eqref{eq:cond_violated} is equivalent to $\underline{B} < \overline{B}.$
E.g., we can choose 
$U(\z_r) = C(\z_r, \hat{\ba}) = 2 \delta + (\underline{B} + \overline{B})/2$ and 
$U(\z_j) = C(\z_j, \hat{\ba}) = 2 \delta$ for all $j \neq r.$ 
Finally, the non-optimality of all $\e_j$'s (i.e., $U(\x_i) \; < \; C(\x_j,\e_j) + U(\z_j), \, \forall j \neq r$) is easy to ensure by selecting 
$C(\x_j,\e_j)$ to be sufficiently large once all $U$ values are already selected.
\Halmos
\end{proof}

 \vspace*{2mm}

To illustrate the connection to causality criteria previously derived for MSSPs, we now turn 
to homogeneous transition cost functions, which often arise naturally in applications.
Recall that $H:\R^m \mapsto \R$ is \emph{homogeneous of degree d} if  $H(\beta \bq) = |\beta|^d H(\bq)$ holds for all $\beta \in \R, \bq = (q_1, ...., q_m) \in \R^m$.
Further, when $H(\bq)$ is smooth, Euler's Homogeneous Function Theorem applies and $H(\bq) = \frac{1}{d}\sum^{m}_{j = 1}\ q_j \frac{\partial H}{\partial q_j}(\bq), \, \forall \bq \in \R^m.$

Focusing on a single node $\x$ and one of its modes $\m = \{\z^{\m}_1, \dots, \z^{\m}_{|\m|}\} \subset X \backslash \{\x\},$ in MSSPs one can choose any $\bxi \in \Xi_{|\m|}$ 
for the desired probability distribution over $\m.$ The corresponding cost $C^{\m}(\x, \bxi)$ can be separately defined for each $\x$ and $\m \in \M(\x);$ when they are clear from context,
it is convenient to abbreviate the notation to $C(\bxi)$ and $m = |\m|.$
As shown in \cite[Theorem 3.2]{vladimirsky2008label}, an MSSP is guaranteed to be monotone $\delta$-causal if, in its every mode, $C(\bxi)$ is continuously differentiable, homogeneous of degree $d$ and satisfies
\begin{equation}
\label{eq:MSSP_MC}
\xi_j > 0 
\quad \implies \quad
\frac{\partial C}{\partial \xi_j} (\bxi) - (d-1)C(\bxi)  \; > \; \delta \geq 0, \qquad
\forall \bxi \in \Xi_{m}, \, j \in \{1,...,m\}.
\end{equation}
Note that this covers the example 
$
C(\bxi) = \sqrt{\xi_1^2 + \xi_2^2 + \xi_3^2}
$
over $\Xi_3$ already considered in Remark \ref{rem:difference_in_MC_conditions}.
This $C(\bxi)$ is homogeneous of degree one with $\frac{\partial C}{\partial \xi_i} (\bxi) = \xi_i / \sqrt{\xi_1^2 + \xi_2^2 + \xi_3^2} > 0$ whenever $\xi_i > 0.$
Thus, by \cite[Theorem 3.2]{vladimirsky2008label}, this $C$ is monotone causal (with $\delta = 0$),
which can be also concluded from our Theorem \ref{theorem:gen_caus_condition_improved}.  

Below we show that, for $m=2$, the condition \eqref{eq:MSSP_MC} implies our \eqref{eq:original_MC_cond}.

\begin{proposition}[Monotone $\delta$-Causality for Homogeneous $C(\bxi)$ when $m = 2$]\label{theorem:3pt2_m2_connection}
\mbox{}\\
Suppose $m = 2$, and consider a transition cost function $C(\bxi)$ which is convex, continuously differentiable, homogeneous of degree $d$, and satisfying
\eqref{eq:MSSP_MC}.
Then it also satisfies 
\begin{equation} \label{eq:simplified_mc_in_2D}
C(\bxi) \; \geq \; \max( C_1 \xi_1 + \delta \xi_2, \, \delta \xi_1 + C_2 \xi_2) \quad \forall \bxi \in \Xi_2, 
\end{equation}
where $C_1$ and $C_2$ are the shorthands for the cost of deterministic actions $\e_1 = (1,0)$ and $\e_2 = (0,1).$ 
\end{proposition}

\begin{proof}{Proof: }
Define $K: [0,1] \mapsto \R$ by setting $K(p) = C(1-p, p),$ as in Figure \ref{fig:mc_geometric_illustration}.
Arguing by contradiction, assume that, for some  $\bar{\bxi} = (1- \bar{p},\, \bar{p}),$
we have 
$  
C(\bar{\bxi}) \, < \, \delta \bar{\xi}_1 + C_2 \bar{\xi}_2
$
or, equivalently,
$
K(\bar{p}) \, < \, \delta + \bar{p} (K(1) - \delta).
$
By the convexity of $K,$ this inequality must also hold for all $p \in [\bar{p}, 1)$
and implies $K'(1) > K(1) - \delta.$ (Illustrated by the solid purple curve in Figure \ref{fig:mc_geometric_illustration}(a).)
From the chain rule, this is equivalent to 
$$
- \frac{\partial C}{\partial \xi_1}(\e_2) +  \frac{\partial C}{\partial \xi_2}(\e_2) \; > \;
C(\e_2) - \delta.
$$
By Euler's Homogeneous Function Theorem, $d C(\e_2) = \frac{\partial C}{\partial \xi_2}(\e_2);$
thus,
$$
\frac{\partial C}{\partial \xi_1}(\e_2) \; < \; (d-1) C(\e_2) + \delta.
$$
Since $C$ is continuously differentiable, this inequality also holds for $\bxi \in \Xi_2$ sufficiently close to $\e_2$,
which contradicts \eqref{eq:MSSP_MC}.

If we instead assume
$ 
C(\bar{\bxi}) \, < \, C_1 \bar{\xi}_1 + \delta \bar{\xi}_2$
(or, equivalently, 
$K(\bar{p}) \, < \, K(0) + \bar{p} (\delta - K(0))$),
a similar argument leads to $K'(0)< \delta - K(0),$ yielding the same contradiction for $\bxi \in \Xi_2$ near $\e_1.$  
\Halmos
\end{proof}
\section{OSSPs and Continuous Optimal Control Problems}\label{section:ossp_and_ahj}

\subsection{Background, prior work, and notation}
SSPs are also quite useful in approximating solutions of stationary Hamilton-Jacobi-Bellman (HJB) partial differential equations (PDEs), which arise naturally in areas as diverse as robotics, anomaly detection, optimal control, computational geometry, image registration, photolithography, and exploration geophysics \cite{SethBook2}. We will focus on a simple example of a time-optimal control
problem in an open domain $\domain \subset \R^m.$  With $\y(t)$ encoding the position at the time $t$, the dynamics is prescribed by
$\y'(t) = f(\y(t), \ba(t))\ba(t),$ where $\ba(t) \in S_{m-1} = \{\ba \in \R^m \mid |\ba| = 1\}$ is a unit vector describing our chosen direction of motion, while the speed 
$f: \domain \times \ S_{m-1} \mapsto \R_+$ generally depends both on the position and the direction.
We will generally assume that $f$ is Lipschitz continuous and will leverage the properties of the corresponding {\em speed profile} 
$\Vf(\x) = \left\{ f(\x, \ba) \ba \, \mid \,  \ba \in S_{m-1} \right\},$ which describes all velocities achievable at $\x.$
It is contained in a spherical shell with the inner and outer radii 
$F_1(\x) = \min_{\ba} f(\x, \ba)$
and $F_2(\x) = \max_{\ba} f(\x, \ba).$
We will further assume that $\exists F_1, F_2 \in \R$ such that $0 < F_1 \leq F_1(\x) \leq F_2(\x) < F_2$ holds for all $\x \in \domain.$
In addition to the $\ba(\cdot)$-dependent time-to-$\boundary,$ we also incur a Lipschitz continuous exit-penalty $q : \boundary \mapsto \R_{+,0}.$
The value function $u(\x)$ encodes the optimal overall time to exit starting from $\x,$  formally defined by taking the infimum over all measurable controls $\ba(\cdot).$
Our assumptions guarantee that $u$ will be Lipschitz-continuous on  $\domain$; moreover, if it is also smooth, 
a standard argument shows that $u$ must also satisfy
the HJB equation 
\begin{equation}\label{eqn:stationary_hj}
\begin{split}
H(\nabla u(\x), \x) &= \min\limits_{\ba \in S_{m-1}} \{(\nabla u(\x) \cdot \ba)f(\x, \ba)\} + 1  = 0, \hspace{0.3cm} \x \in \domain \\
 u(\x) & = q(\x), \hspace{0.3cm} \x \in \boundary.
\end{split}
\end{equation}
More generally, even if $u$ is not smooth, it can be recovered as a unique {\em viscosity solution} of this PDE \cite{bardi1997optimal}.
If we further assume that $\Vf(\x)$ is convex for all $\x \in \Omega$, this guarantees the existence of time-optimal trajectories, which 
are encoded by the characteristic curves of \eqref{eqn:stationary_hj}.

A frequently encountered subclass of problems has isotropic dynamics (i.e.,  with $f(\x, \ba) = f(\x)$) and allows for further simplification.
In this case, $\Vf(\x)$ is always just a ball, the optimal direction is known to be $\ba_* = -\nabla u /  |\nabla u|,$ the optimal trajectories 
coincide with the gradient lines of $u,$ and \eqref{eqn:stationary_hj} reduces to an {\em Eikonal equation}
$ | \nabla u(\x) | f(\x) = 1$.  

Efficient numerical methods for both the Eikonal equation and the general (anisotropic) HJB \eqref{eqn:stationary_hj} have been an active area of research in the last 30 years.
Many of such methods start with a semi-Lagrangian discretization \cite{GonzalezRofman, Falcone_book} on a grid or simplicial mesh $X$ over $\cdomain.$
Suppose $\Scal(\x)$ is the set of simplexes used to build a computational stencil at $\x,$ with each $s \in \Scal(\x)$ having vertices $\x, \z^s_1, ..., \z^s_m$.
(For $m=2$ and a Cartesian grid, two examples of such stencils are shown in Figure \ref{fig:4pt_and_8pt_stencils}.)
\begin{figure*}[t]
\centering
\includegraphics[width = 0.7\textwidth]{figs/f_and_causality/4pt_8pt.tikz} 
\caption{Two computational stencils based on a uniform Cartesian grid in $\R^2$.
The value function
is computed at $\x$, and $\tilde{\x}_{\bxi}$ is the
new position after traveling in a straight line along the chosen direction of motion $\ba_{\bxi}$ with speed $f(\x, \ba_{\bxi})$. 
In panel (a), the modes are essentially quadrants: $\Scal(\x) = \{ (\x_1, \x_3), (\x_3, \x_5), (\x_5, \x_7), (\x_7, \x_1) \}.$  In panel (b), each quadrant is split into two modes: 
$\Scal(\x) = \{ (\x_1, \x_2), (\x_2, \x_3), (\x_3, \x_4), (\x_4, \x_5), (\x_5, \x_6), (\x_6, \x_7), (\x_7, \x_8), (\x_8, \x_1) \}.$
\label{fig:4pt_and_8pt_stencils}}
\end{figure*}
For a chosen $s,$ any $\bxi \in \Xi_m$ defines a ``waypoint'' $\xtilde^{}_{\bxi} = \sum_{j} \xi_j \z^s_j$ on the opposite face of the simplex  and 
the corresponding direction of motion $\ba_{\bxi} = (\xtilde_{\bxi} - \x) / |\xtilde_{\bxi} - \x|.$
Assuming we choose that direction and follow it from $\x$ to $\xtilde_{\bxi},$  the time it takes will be approximately $|\xtilde_{\bxi} - \x| / f(\x, \ba_{\bxi})$ and
the remaining minimal time from there on can be approximated by linear interpolation as $u(\xtilde_{\bxi}) \approx \sum_{j} \xi_j u(\z^s_j).$ 
Using $U(\x)$ for the function approximating the true solution $u$ at meshpoints/gridpoints, this yields the following first-order accurate semi-Lagrangian discretization:
\begin{align}\label{eqn:sl_scheme_hj}
U(\x) & = \; \min\limits_{s \in \Scal} 
\min\limits_{\bxi \in \Xi_m}
\left\{ \frac{| \xtilde_{\bxi} - \x |}{f(\x, \ba_{\bxi})} \, + \, \sum\limits_{j=1}^m \xi_j U(\z^s_j)  \right\},  & \qquad \forall \x \in X \cap \domain;\\
\nonumber
U(\x) & = \; q(\x), & \qquad \forall \x \in X \cap \boundary.
\end{align}
Assuming that $\boundary$ is well resolved by a family of computational grids $\{X^h\}$, it is well-known that $U^h \rightarrow u$ uniformly under grid refinement
[i.e., as $|X^h| \rightarrow +\infty$ with $\left( \max_{\x \in X^h} \max_{s \in \Scal(\x)} \max_j |\z^s_j - \x| \right) \rightarrow 0$].  
In the approach pioneered by Kushner and Dupuis \cite{KushnerDupuis}, the same system of equations can be obtained by choosing a suitable controlled Markov process on $X$.
Indeed, taking $C^s(\x, \bxi) = |\xtilde_{\bxi} - \x| / f(\x, \ba_{\bxi}),$ it is easy to see that the above is actually an MSSP defined in \eqref{eq:MSSP}
with modes $s = \{\z^s_1, ..., \z^s_m\}.$  The nodes $\x \in \boundary$ are assumed to have a single deterministic action only, leading to $\bt$ at the cost $q(\x).$
In a yet another approach \cite{SethVlad3}, the same system can be also derived through finite differences, approximating directional derivatives of $u$ for each $(\z^s_j - \x)$.
Since the discretized system \eqref{eqn:sl_scheme_hj} is nonlinear and coupled, it has been traditionally treated with value iterations, but there was always a significant interest in speeding them up, often mirroring the label-correcting and label-setting methods on graphs.  Here we focus on the latter, referring to \cite{ChacVlad1, ChacVlad2} for a review of the former, the hybrid versions, and parallelization.  

The key label-setting ideas were first developed in the isotropic case.  
If $C^s(\x, \ba_{\bxi})$ is monotone \mbox{($\delta$-)causal} for each $s,$ 
the system \eqref{eqn:sl_scheme_hj} may be solved using noniterative label-setting methods. 
This was first done by Tsitsiklis \cite{tsitsiklis1994efficient, tsitsiklis1995efficient}, who showed that any Eikonal discretized in a semi-Lagrangian way on the stencil in Figure  \ref{fig:4pt_and_8pt_stencils}(a) (or its $m$-dimensional version) can be solved by Dijkstra's method.
Sethian independently derived the same result in an Eulerian (finite-difference) framework. Working in the context of isotropic front propagation problems, he introduced a Dijkstra-like Fast Marching Method \cite{sethian1996fast},
which was soon generalized to handle higher-order accurate discretizations \cite{sethian1999fast} and arbitrary simplicial meshes both in $\R^m$ and on manifolds \cite{kimmel1998fast, sethian2000fast}.
Tsitsiklis also showed that Dial's method with bucket width $\delta = \frac{h}{F_2 \sqrt{2}}$ is applicable to an Eikonal discretized on the stencil in  Figure  \ref{fig:4pt_and_8pt_stencils}(b)
while its 3D version can be treated with the bucket width $\delta = \frac{h}{F_2 \sqrt{3}}.$  Another Dial-like approach was defined in \cite{KimGMM}.  In all these cases, the proofs have heavily relied on  specific discretization approaches and stencil choices.
The MSSPs introduced in \cite{vladimirsky2008label} have later provided a unifying framework, in which it is easy to prove all related $\delta$-MC conditions 
based on the properties of $C^s(\x, \bxi).$

Building label-setting methods to correctly solve the system \eqref{eqn:sl_scheme_hj} in the anisotropic setting is significantly more challenging since characteristics of the HJB PDE no longer coincide with the gradient lines of $u.$  Aside from the small set of problems with ``grid-aligned anisotropy''
\cite{osher2001level, alton2009fast},
the causality is generally not guaranteed on usual stencils of the type shown in Figure \ref{fig:4pt_and_8pt_stencils}.
Sethian and Vladimirsky have handled this problem by introducing Dijkstra-like Ordered Upwind Methods (OUMs), in which the stencil is extended dynamically (at run-time) just enough to ensure the monotone causality \cite{sethian2001ordered}. 
In OUMs, the range of stencil extension is governed by an anisotropy coefficient\footnote{
The anisotropy coefficient can be used to bound from above the angle between $-\nabla u (\x)$ and the optimal direction of motion $\ba_*(\x).$  
This is why the stencil extension proportional to $\Upsilon(\x)$ is sufficient to ensure the monotone causality.
However, later OUMs were also extended to approximate the solutions of HJB PDEs in which both $F_2(\x)$ and $\Upsilon(\x)$ might be infinite \cite{cameron2012}.}     
 $\Upsilon(\x)  = F_2(\x)/F_1(\x),$ which is equal to one (no stencil extension) in the isotropic case.
The resulting algorithms are indifferent to the exact speed profile, using the same dynamic stencil extension for all $\Vf(\x)$ that share the same $F_1(\x)$ and $F_2(\x).$   
This yields the computational cost of $O(\Upsilon^{m-1} n \log n),$ where $n= |X|$ and $\Upsilon = \max_{\x} \Upsilon(\x).$
Subsequently, a number of papers have aimed to reduce this stencil extension and pre-build a static causal stencil for each $\x$ based on the finer properties of $\Vf(\x).$ 
The first general anisotropic $\delta$-MC criterion was derived in \cite{vladimirsky2008label} and then extended to non-smooth $\Vf(\x)$ in \cite{alton2012ordered}.
Starting with a uniform Cartesian grid, Mirebeau showed how to build MC-stencils for elliptic and ellipsoidal speed profiles for $m= 2,3$ with 
the uniform bound on stencil cardinality $|\N(\x)|$ independent of $\Upsilon$ \cite{mirebeau2014anisotropic}.  A related approach was also developed to build reduced-cardinality MC-stencils for general (smooth and convex) speed profiles in 2D \cite{mirebeau2014efficient} and 3D \cite{desquilbet2021single}.
But in all of these cases, the MC criteria were based on the properties of the Hamiltonian $H,$ on implicit conditions for the derivatives  of $f$,
or on analytic properties of a function embedding $\Vf$ as a level set.
Below we show that the criteria for OSSPs developed in \S \ref{ss:OSSP_MC} yield surprisingly simple $\delta$-MC conditions for 2D and 3D stencils, 
using only the basic geometric properties of $\Vf.$

\textbf{Simplified Notation:} Throughout the remainder of this section, we make use of the following notational conventions. Without loss of generality, we assume that the coordinate system is centered at $\x=\bm{0}$ and will mostly focus on a single simplex $s \in \Scal(\x)$ with vertices $\bm{0}, \z_1, ...., \z_m.$  
The speed, velocity, and cost for the direction $\ba_{\bxi}=\tilde{\x}_{\bxi} / |\tilde{\x}_{\bxi}|$ will be denoted by 
$f(\bxi) = f(\x, \ba_{\bxi}), \, \bv(\bxi) = f(\bxi) \ba_{\bxi},$ and 
$C(\bxi) = C^s(\bxi) =  |\tilde{\x}_{\bxi}| / f(\bxi)$ respectively.
A similar notation ($f_j, \bv_j,$ and $C_j = |\z_j| / f_j$) will also be used for the speed, velocity, and cost for each direction $\z_j / |\z_j|$
toward the corresponding vertex.
We will use $\Vf_s = \{ \bv \in \Vf \, \mid \, \bv = \sum_{j=1}^m \theta_j \bv_j, \; \text{all } \theta_j \geq 0 \}$ to denote  
the part of $\Vf$ falling into this simplex.

\begin{observation}
\label{thm:convex_profile}
If $\Vf(\x)$ is convex then the function $C(\bxi)$ is convex in every simplex $s \in \Scal(\x).$
\end{observation}
\begin{proof}{Proof:}
Suppose $\boldeta_1, \boldeta_2 \in \Xi_m$ and $\bphi = \xi_1 \boldeta_1 + \xi_2 \boldeta_2$ for some $\bxi = (\xi_1, \xi_2) \in \Xi_2.$
The corresponding waypoints on the simplex $(\z_1, ..., \z_m)$ are 
$\xtilde_{\boldeta_i}= \sum_{j=1}^m \eta_{i,j} \z_j$ for $i=1,2$ and $\xtilde_{\bphi} =  \xi_1 \xtilde_{\boldeta_1} + \xi_2 \xtilde_{\boldeta_2}.$
The respective directions are $\ba_{\boldeta_i} = \xtilde_{\boldeta_i} / |\xtilde_{\boldeta_i}|$ and $\ba_{\bphi} = \xtilde_{\bphi} / |\xtilde_{\bphi}|.$
The velocity corresponding to $\ba_{\bphi}$ is
$$
\bv(\bphi) \, =  \, f(\bphi)\frac{\tilde{\x}_{\bphi}}{|\tilde{\x}_{\bphi}|} \, = \, \frac{f(\bphi)}{|\tilde{\x}_{\bphi}|}\left(\xi_1 \xtilde_{\boldeta_1} + \xi_2 \xtilde_{\boldeta_2} \right)
\, = \, \theta_1 \bv(\boldeta_1) + \theta_2 \bv(\boldeta_2),
$$
where $\theta_1, \theta_2 \geq 0$ 
since $\bv(\bphi)$ belongs to the angular sector formed by $\bv(\boldeta_1)$ and $\bv(\boldeta_2).$ 

Since $\bv(\boldeta_i) = f(\boldeta_i) \frac{\xtilde_{\boldeta_i}}{|\xtilde_{\boldeta_i}|},$ we see that
$\frac{f(\bphi) \xi_i}{|\tilde{\x}_{\bphi}|} = \theta_i \frac{f(\boldeta_i)}{|\xtilde_{\boldeta_i}|}$ for $i=1,2.$
By definition of $C,$ this is equivalent to 
\begin{equation}\label{eqn:theta_mc_convexity}
\theta_1 = \frac{\xi_1C(\boldeta_1)}{C(\bphi)}  \hspace{0.5cm} \text{and} \hspace{0.5cm} \theta_2 = \frac{\xi_2C(\boldeta_2)}{C(\bphi)}.
\end{equation}
The convexity of $\Vf(\x)$ implies  $\theta_1 + \theta_2 \geq 1$ or, equivalently, 
$\xi_1 C(\boldeta_1) + \xi_2 C(\boldeta_2)  \geq C(\bphi).$
\Halmos
\end{proof} 
We also note that, if $f(\bxi)$ is written as a function of direction (i.e., homogeneous of degree zero)
then $C(\bxi) = |\xtilde_{\bxi}| / f(\bxi)$ is homogeneous of degree one, making MC conditions such as 
\eqref{eq:MSSP_MC} applicable.  But we will see that Theorems \ref{theorem:gen_caus_condition_improved} and \ref{theorem:gen_caus_condition} are more useful in relating MC to geometry.

\begin{remark}\label{rem:another_mirebeau_connection}
In the case of a convex $C(\bxi)$ homogeneous of degree one, the $\delta$-MC criterion \eqref{eq:improved_MC_cond} 
from Theorem \ref{theorem:gen_caus_condition_improved} simplifies to 
\begin{equation}
\label{eq:eq:improved_MC_cond_simplified}
C(\bxi) \geq C \left(\bxi - \xi_r \e_r \right)  + \xi_r \delta, \qquad \forall r \in \{1, ..., m \},
\end{equation}
where $\e_r$'s are the canonical basis vectors.  
While $\left(\bxi - \xi_r \e_r \right)$ is not a probability distribution on successor nodes,  this formula is still very suggestive.  
For the specific case of $m=2$ and $\delta=0,$ 
a similar relation is the principle behind the notion of ``F-acuteness'' introduced in \cite{mirebeau2014efficient}.  
\end{remark}

\subsection{Monotone Causality in $\R^2$}
\begin{theorem}[
Sufficient Causality Condition on $\Vf$ in $\R^2$]\label{thm:bounding_parallelograms_r2}
$C(\bxi)$ is monotone causal on the simplex formed by $(\x=\bm{0}, \z_1,\z_2)$ if 
$\Vf_s(\x)$
is fully contained in a parallelogram with vertices at $\bm{0}, \bv_1, \bv_2$, and $\bv_1 + \bv_2$.

If the above property holds for every $s \in \Scal(\x)$ (i.e., if the entire $\Vf(\x)$ is contained in the union of such parallelograms),
then the entire stencil is causal at $\x$.  If this is the case  $\forall \x \in X,$ then Dijkstra's method is applicable.
\end{theorem}
\proof{Proof: }
Using the same argument as in Observation \ref{thm:convex_profile} but with $m=2$ and 
deterministic $\boldeta_i = \e_i$ for $i=1,2$, we see that
$\bphi = \bxi$ and $\xtilde_{\boldeta_i} = \z_i.$  Thus, \eqref{eqn:theta_mc_convexity} reduces to
\begin{equation}\label{eqn:lambda_mc_restriction}
\theta_1 = \frac{\xi_1C_1}{C(\bxi)}  \hspace{0.5cm} \text{and} \hspace{0.5cm} \theta_2 = \frac{\xi_2C_2}{C(\bxi)}.
\end{equation}
If $\bv(\bxi) = \theta_1 \bv_1 + \theta_2 \bv_2$ lies in that parallelogram, this means that $\theta_i \leq 1$ for $i=1,2$.  
Combined with \eqref{eqn:lambda_mc_restriction}, this is equivalent to \eqref{eq:original_MC_cond} with 
$\delta = 0$ and Theorem \ref{theorem:gen_caus_condition} applies.
 \Halmos
\endproof

\begin{figure*}[h]
\centering
$
\arraycolsep=1pt\def\arraystretch{0.1}
\begin{array}{ccc}
\includegraphics[width = 0.3\textwidth]{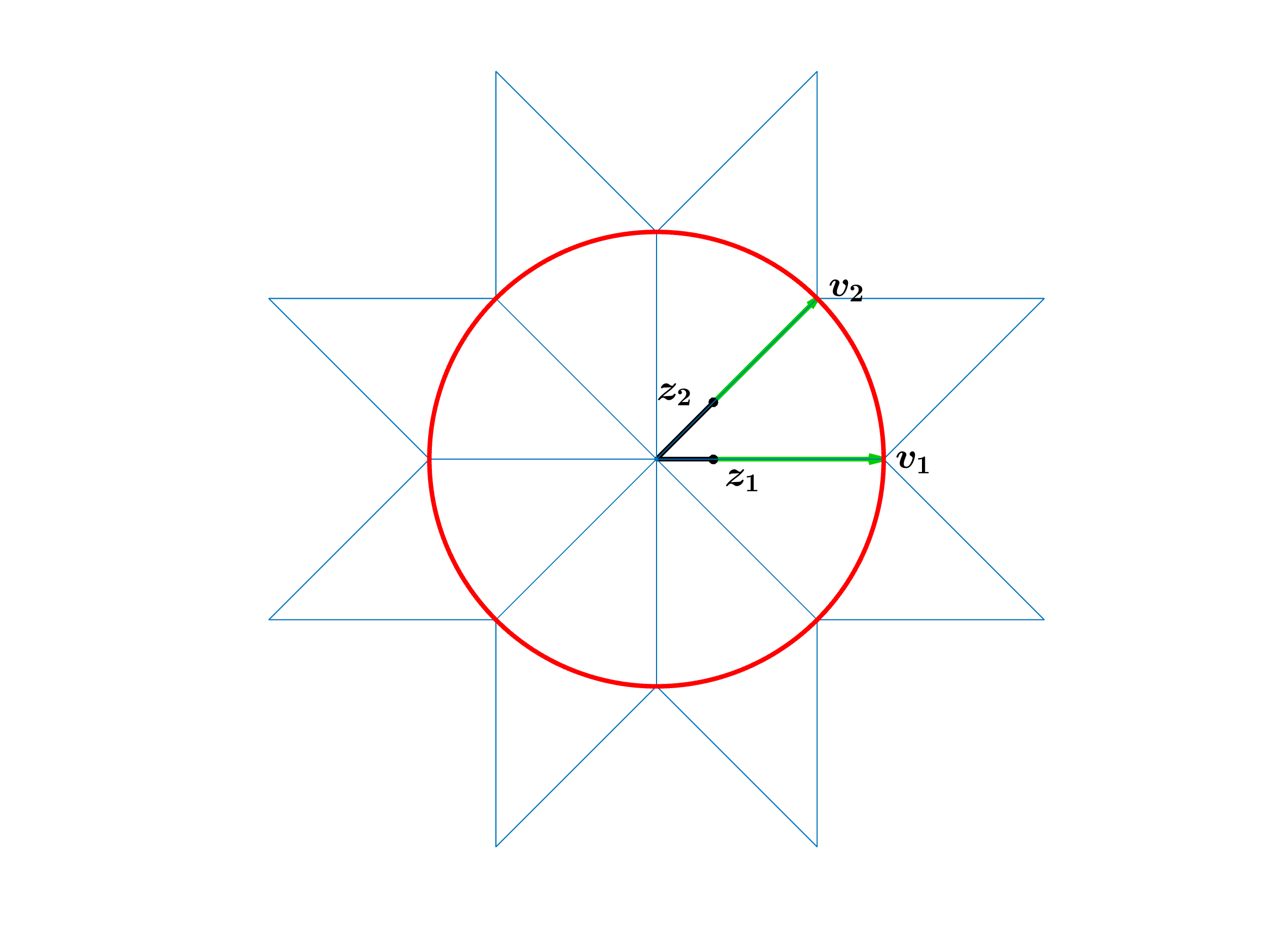} &
\includegraphics[width = 0.3\textwidth]{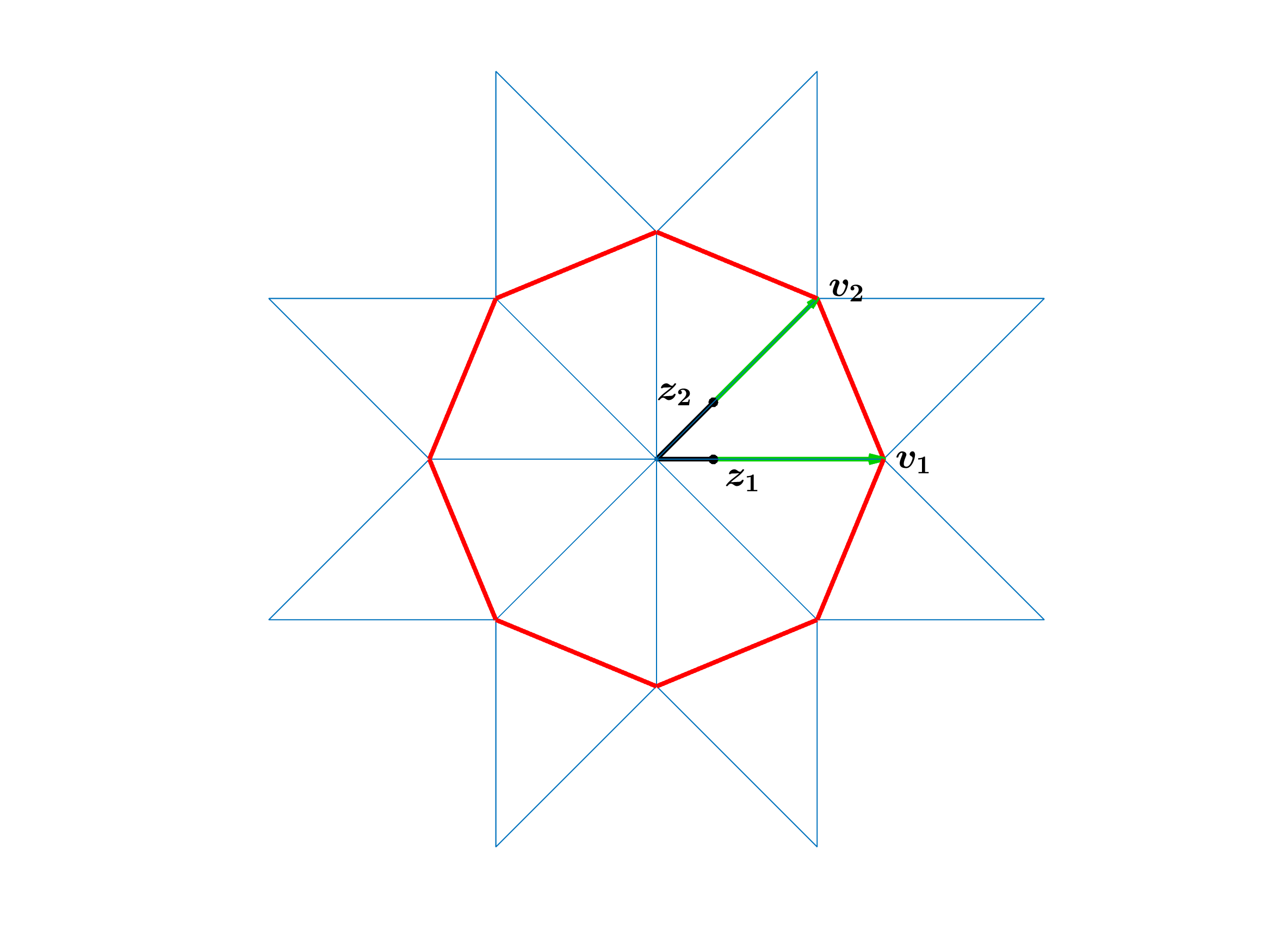} &
\includegraphics[width = 0.3\textwidth]{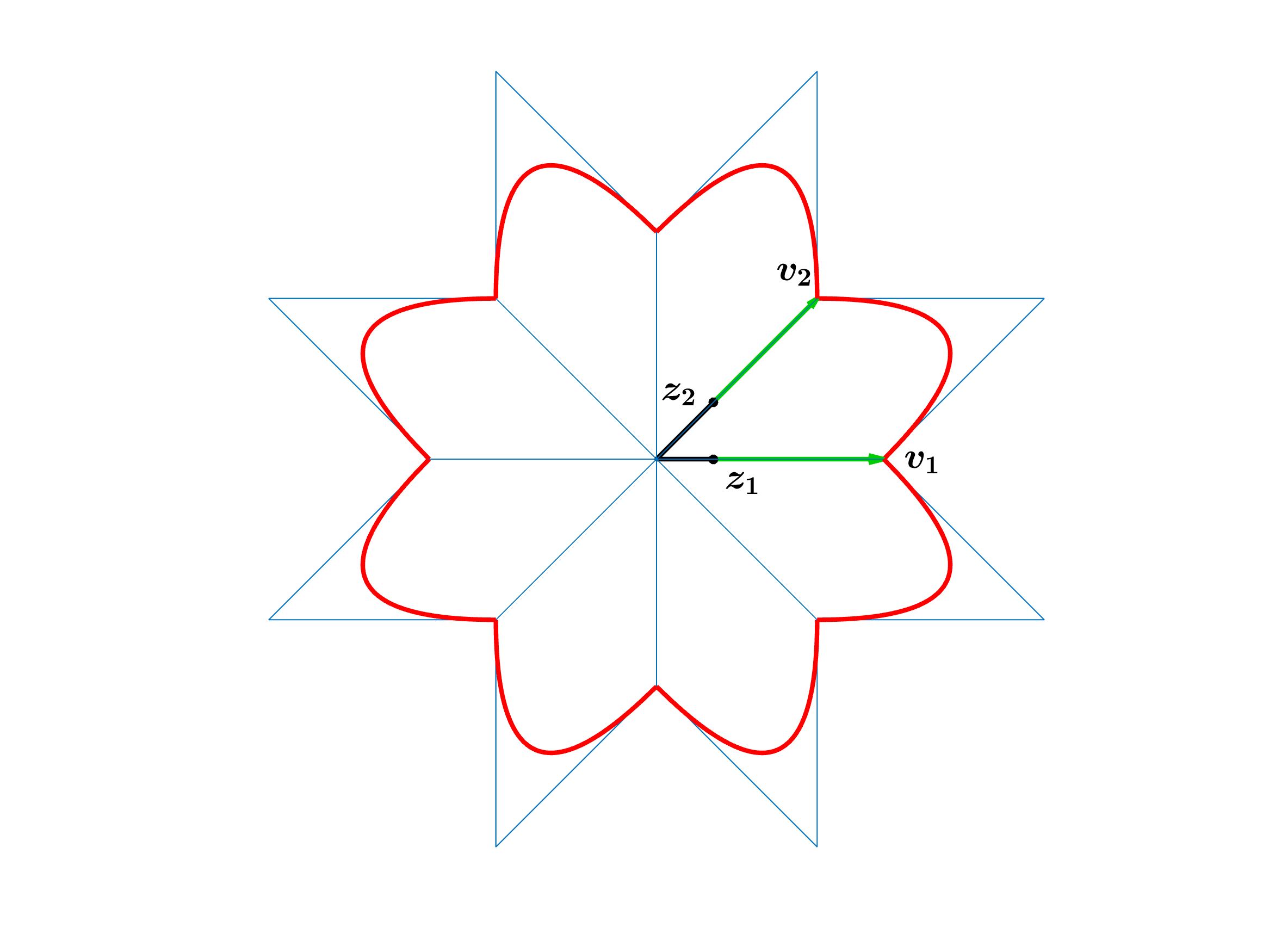} \\[-10pt]
\mbox{\footnotesize(a)} & \mbox{\footnotesize(b)}  & \mbox{\footnotesize(c)} \\
\includegraphics[width = 0.3\textwidth]{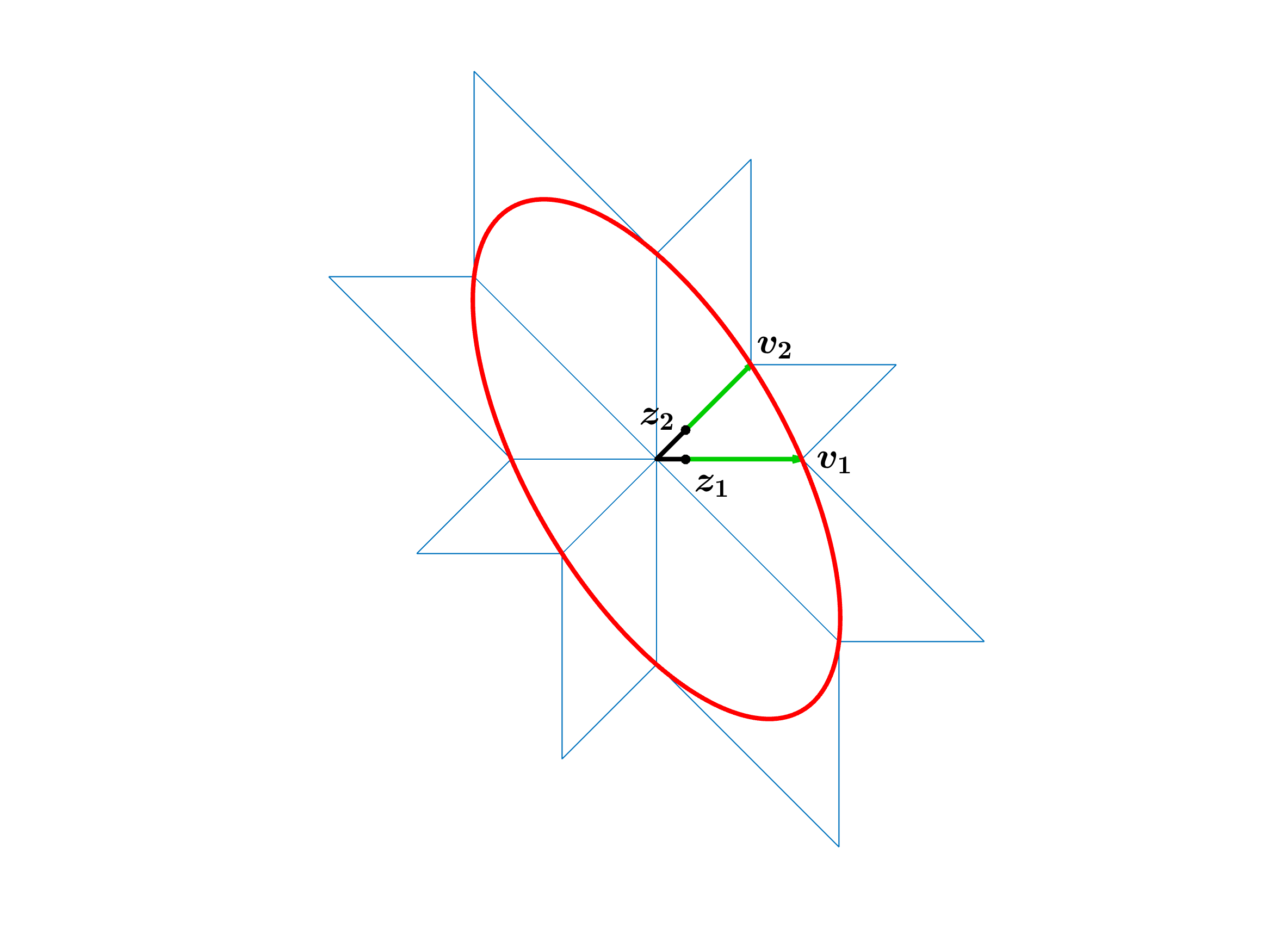} &
\includegraphics[width = 0.3\textwidth]{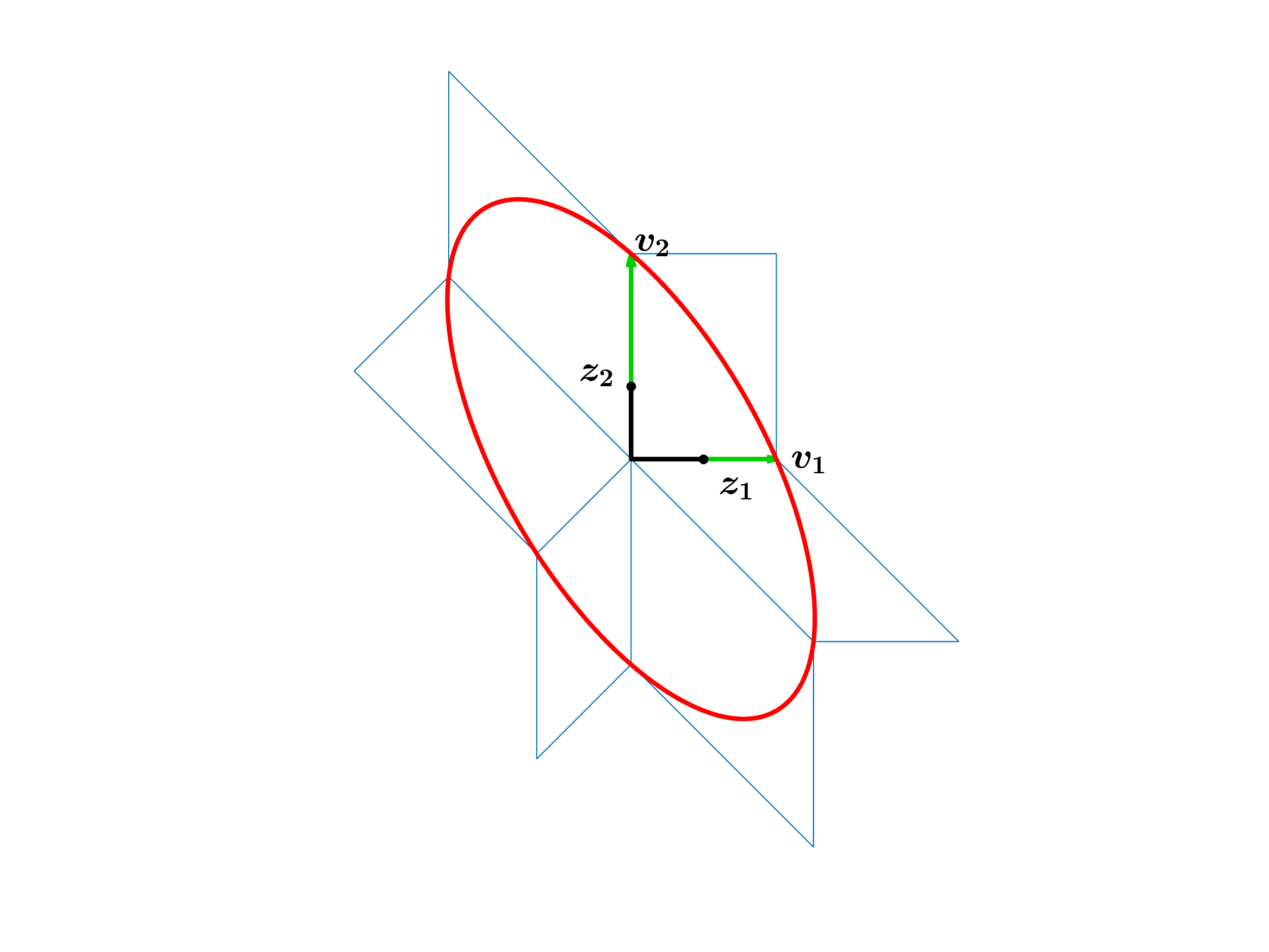} &
\includegraphics[width = 0.3\textwidth]{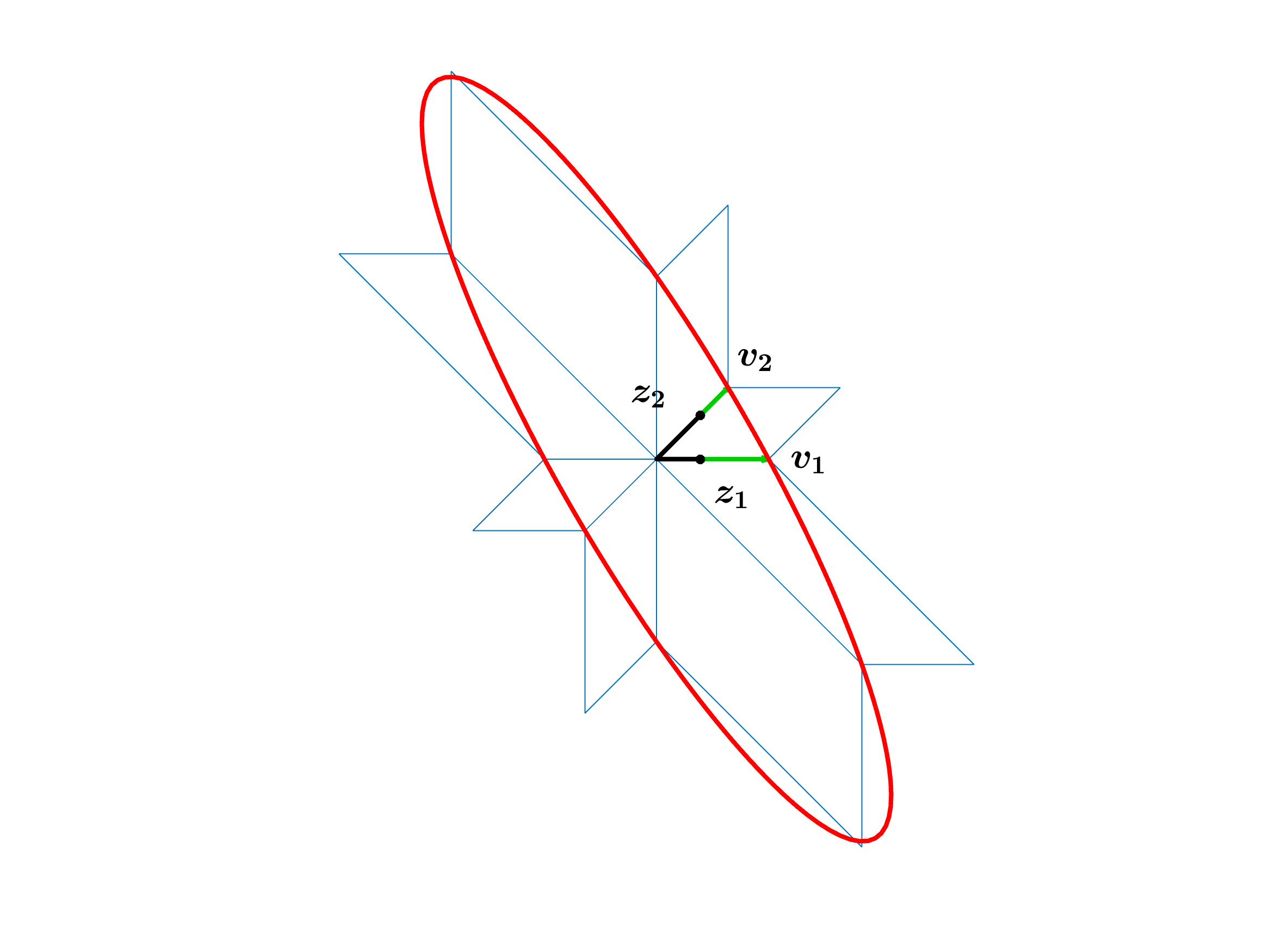} \\[-10pt]
\mbox{\footnotesize(d)} & \mbox{\footnotesize(e)}  & \mbox{\footnotesize(f)}
\end{array}
$
\caption{Bounding parallelograms for several speed profiles on the 8-point stencil shown in Figure \ref{fig:4pt_and_8pt_stencils}(b). The speed profile is depicted in red. The vectors along directions $\boldsymbol{z}_1$ and $\boldsymbol{z}_2$ are shown in black, and the velocities $\boldsymbol{v}_1$ and $\boldsymbol{v}_2$ along their respective directions $\boldsymbol{z}_1$ and $\boldsymbol{z}_2$  are shown in green. Panel (a): Isotropic circular speed profile. Panel (b): Octagonal speed profile. 
Panel (c):  A non-convex speed profile.
Panels (d) and (e): Tilted elliptical speed profile. Panel (e) illustrates the result of reducing the standard 8-point stencil used in the previous examples to a 6-point stencil by 
removing any superfluous
nodes $\z_j$
which are not crucial to the monotone-causality of the stencil. 
The resulting MC stencil is not symmetric relative to $\x$ despite the symmetry of the speed profile $\Vf(\x).$
Panel (f): Tilted elliptical speed profile with an aspect ratio twice that of the ellipse in panels (d) and (e). For this aspect ratio, monotone-causality no longer holds on the standard 8-point stencil. The condition in Theorem \ref{thm:bounding_parallelograms_r2} is violated on two simplexes in $\Scal(\x).$ 
\label{fig:r2_simplex_and_parallelogram}}
\end{figure*}
By Theorem \ref{thm:bounding_parallelograms_r2},
the standard 8-point stencil is guaranteed to be monotone causal at $\x$ for all examples of speed profiles in Figure \ref{fig:r2_simplex_and_parallelogram}(a-d).
The non-convex example in Figure \ref{fig:r2_simplex_and_parallelogram}(c) might seem paradoxical.  It is well-known that replacing $\Vf$ with its convex hull does not change the value function since the velocities which are not on the convex hull boundary are never optimal for any $\nabla u.$  
Thus, these eight ``deterministic directions'' (corresponding to transitions to eight closest gridpoints) will never be optimal in that continuous path-planning problem,
but might be still optimal in a monotone causal OSSP that discretizes it.  
This is an artifact of our simplex-by-simplex approach of enforcing MC -- the non-convexity (and the irrelevance of a large number of directions) is  not known if we can only consider the portion of $\Vf$ falling between $\bv_1$ and $\bv_2.$   Figure \ref{fig:r2_simplex_and_parallelogram}(e) illustrates MC-stencil reduction by merging two pairs of simplexes from panel (d). As shown in \cite{mirebeau2014anisotropic}, an MC-stencil based on only 6 Cartesian grid points can be found for any elliptic $\Vf,$ but for ellipses with larger aspect ratio and generic orientation such causal stencils become increasingly non-local.       Figure \ref{fig:r2_simplex_and_parallelogram}(f) shows that a local/standard 8-point stencil is not MC when the ellipse's aspect ratio is sufficiently large.

When $\Vf$ is known to be convex, checking the causality condition \eqref{eqn:lambda_mc_restriction} becomes a lot simpler,
since it can be replaced by a pair of inequalities verified at $\bv_1$ and $\bv_2$.  The following is
an alternative geometric proof of the already known criterion \cite[Lemma~2.2]{mirebeau2014efficient}.
\begin{proposition}[Causality and Tangent Information in $\R^2$]\label{thm:normal_vec_conditions} \mbox{}\\
Consider a simplex formed by $\x=\bm{0}, \z_1, \z_2.$
Suppose $\Vf$ is convex and smooth at $\bv_1$ and $\bv_2,$ with $\nhat_1$ and $\nhat_2$ denoting the respective outward normal vectors to $\Vf.$ 
Then $C(\bxi)$ is guaranteed to be monotone causal within this simplex if 
\begin{equation}
\label{eq:tang_MC}
\bv_1 \cdot \nhat_2 \geq 0 \hspace{0.3cm} \text{and} \hspace{0.3cm} \bv_2 \cdot \nhat_1 \geq 0.
\end{equation}
\end{proposition}
\proof{Proof: }
By convexity, the angular sector of $\Vf$ is contained in a quadrilateral formed by $\bv_1, \bv_2,$ and two respective tangent lines $T_1$ and $T_2$.  The inequalities \eqref{eq:tang_MC} ensure 
that this quadrilateral is fully contained in the parallelogram with vertices $(\bm{0}, \bv_1, \bv_2, \bv_1 + \bv_2);$  see Figure \ref{fig:convex_angle_diagram}(a).
\Halmos 
\endproof

In the Eikonal case, it is well known that the stencil is MC at $\x$ if and only if each angle between $(\z_i - \x)$ and $(\z_j - \x)$ is non-obtuse in each $s \in \Scal(\x)$ \cite{sethian2000fast}.
Since the isotropic $\Vf$ is just a ball, in 2D this is a simple consequence of \eqref{eq:tang_MC}.
For smooth convex anisotropic speed profiles in 2D, the above criterion also provides an easy to verify alternative to the 
``negative-gradient-acuteness'' condition developed in \cite{alton2012ordered}.

\subsection{Monotone Causality in $\R^{3}$}

It might seem intuitive that the 3D version of the MC condition in Theorem \ref{thm:bounding_parallelograms_r2} should require the portion of $\Vf$ falling into the simplex $(\x=\bm{0}, \z_1, \z_2, \z_3)$
to lie within a parallelepiped $\Pi = \{ \sum_{j=1}^3 \theta_j \bv_j \, \mid \, \theta_j \in [0,1] \text{ for } j = 1,2,3\}.$  However, this condition turns out to be insufficient.  
As a motivating (counter-)example,
consider a simplex $s$ corresponding to a positive orthant, based on unit vectors $\z_j = \e_j$ with $C_j = 1$ for $j=1,2,3.$  
We will assume that in this simplex there are only 4 principal directions of motion ($\e_1, \e_2, \e_3,$ and 
along $\hat{\bxi} = (\frac{1}{3}, \frac{1}{3}, \frac{1}{3})$), and for the other directions in $s$ the velocity is obtained by taking a convex combination of those four.
In other words, the relevant part of the speed profile $\Vf_s$ is a union of three triangles with respective vertices  
$\{\bv(\hat{\bxi}), \bv_1, \bv_2\}, \{\bv(\hat{\bxi}), \bv_2, \bv_3\},$ and $\{\bv(\hat{\bxi}), \bv_1, \bv_3\}.$
In this case, $\Pi$ is just a unit cube and selecting $f(\hat{\bxi}) = \sqrt{3}$ ensures  
$\{\bv_1, \bv_2, \bv_3, \bv(\hat{\bxi}) \} \subset \Vf_s \subset \Pi.$  
To show that this simplex is not MC for the described $\Vf,$ it is enough to find some $U(\z_j)$ values such that  $\hat{\bxi}$ is uniquely optimal but $U(\x) < \max_j U(\z_j).$
It is easy to check that this is the case with $U(\z_1) = 0,  \, U(\z_2) = \frac{1}{4}, \, U(\z_3) = \frac{3}{2}$. 
Since $|\tilde{\x}_{\hat{\bxi}} - \x|  = |\hat{\bxi}| = \frac{1}{\sqrt{3}},$ we have $C^s(\hat{\bxi}) = |\tilde{\x}_{\hat{\bxi}} - \x| / f(\hat{\bxi}) = \frac{1}{3}$. Thus, $U(\x) = \frac{1}{3} + \frac{1}{3}U(\z_1) + \frac{1}{3}U(\z_2) + \frac{1}{3}U(\z_3) = \frac{11}{12} < \min_j \{C_j + U_j\} = 1$, and so $\hat{\bxi}$ is uniquely optimal.  On the other hand,
$U(\x) = \frac{11}{12} < U(\z_3),$ with the latter not available in time if we tried to compute $U(\x)$ using Dijkstra's method.

The correct sufficient condition is more subtle and uses oblique projections.
Before stating it, we need to introduce some additional notation.  
If $(i,j,k)$ is any permutation of $(1,2,3),$ we define the parallelograms 
$\Pi^{k} = \{ \theta_1  \bv_i + \theta_2 \bv_j \, \mid \, \theta_1,\theta_2 \in [0,1] \}$
and the parts of $\Vf_s$ in those same planes
$ \Vf^{k} =  \{ (\theta_1  \bv_i + \theta_2 \bv_j) \in \Vf  \, \mid \, \theta_1,\theta_2 \geq 0 \}.$
Finally, we also define the corresponding generalized semi-infinite cylinder as
$ \Psi^{k} = 
\{ \theta_1  \bv_k + \theta_2 \bv  \, \mid \, \bv \in \Vf^{k},  \, \theta_1 \geq 0, \, \theta_2 \in [0,1] \}.$

\begin{theorem}[Causality Condition on $\Vf$ in $\R^3$]\label{thm:mc_speed_profile_r3}
A simplex $s$
is guaranteed to be monotone causal if
$\Vf$ is convex,
$\Vf_s \subset \Big(  \Psi^1 \cap \Psi^2 \cap \Psi^3 \Big),$ and $\Vf^r \subset \Pi^r$ for $r=1,2,3.$
\end{theorem}
\begin{proof}{Proof: }
We will prove that the resulting $C$ satisfies the conditions of Theorem \ref{theorem:gen_caus_condition_improved} with $\delta = 0.$
For deterministic actions $\e_i$ this follows directly from {\bf (A5')}. 

If $\bv(\bxi) \in \Vf^{1},$ then $\I(\bxi) = 2$ and both of its oblique projections are deterministic: $\bgamma_2 = \e_3$ and $\bgamma_3 = \e_2.$
So, \eqref{eq:improved_MC_cond} is equivalent to $C(\bxi) \geq \xi_2 C_2$ and $C(\bxi) \geq \xi_3 C_3.$
But both of these are guaranteed by $\Vf^1 \in \Pi^1$; see Theorem \ref{thm:bounding_parallelograms_r2}.
The same argument also covers the cases $\bv(\bxi) \in \Vf^{2}$ and  $\bv(\bxi) \in \Vf^{3}.$

If $\bxi \in \interior(\Xi_3),$ then $\bv(\bxi) = \sum_{j=1}^3 \theta_j \bv_j.$
A similar derivation to the one in Observation \ref{thm:convex_profile} and Theorem \ref{thm:bounding_parallelograms_r2} shows 
\begin{equation}
\label{eq:theta_j}
\theta_j \; = \; \frac{\xi_j C_j}{C(\bxi)} \; = \; \xi_j \frac{f(\bxi)}{|\xtilde_{\bxi}|} \frac{|\z_j|}{f_j} \; > \;0  \qquad  \forall j.
\end{equation}

So, for any $r \in \{1,2,3 \},$ we have the oblique projection $\bgamma_r,$ the corresponding waypoint 
\mbox{$\xtilde_{\bgamma_r} = \sum_{j \neq r} \gamma_{r,j} \z_j,$} and velocity 
$\bv(\bgamma_r) = \sum_{j \neq r} \frac{\gamma_{r,j} \z_j}{|\xtilde_{\bgamma_r}|}  f(\bgamma_r).$ 
Thus,
\begin{align*}
\bv(\bxi) & = \,  \theta_r \bv_r + \sum\limits_{j \neq r} \theta_j \bv_j \, = \,
\theta_r \bv_r  + \frac{f(\bxi)}{|\xtilde_{\bxi}|} \sum\limits_{j \neq r} \xi_j  \frac{|\z_j|}{f_j} \bv_j 
\, = \, 
\theta_r \bv_r  + \frac{f(\bxi)}{|\xtilde_{\bxi}|} \sum\limits_{j \neq r} \xi_j  \z_j\\ 
& = \, 
\theta_r \bv_r  + (1-\xi_r) \frac{f(\bxi)}{|\xtilde_{\bxi}|} \sum\limits_{j \neq r} \gamma_{r,j}  \z_j 
\, = \, 
\theta_r \bv_r  + \left[ (1-\xi_r) \frac{f(\bxi)}{|\xtilde_{\bxi}|}  \frac{|\xtilde_{\bgamma_r}|}{f(\bgamma_r)} \right]  \bv(\bgamma_r),\\
& = \,  
\theta_r \bv_r  + \left[ (1-\xi_r) \frac{C(\bgamma_r)}{C(\bxi)} \right]  \bv(\bgamma_r),
\end{align*}
where the sequence of equalities uses \eqref{eq:theta_j}, the definition of $\bv_j,$ the definition of $\bgamma_r,$ the representation
of $\bv(\bgamma_r),$ and the formula for $C$.  
Since $\Vf_s \subset \Psi^r,$ the coefficient in front of  $\bv(\bgamma_r)$ has to be less than or equal to 1.  By Observation \ref{thm:convex_profile}, the convexity of $\Vf_s$
guarantees $\check{C}(\bgamma_r) = C(\bgamma_r),$ and thus $C(\bxi) \geq  (1-\xi_r) \check{C}(\bgamma_r).$
\Halmos
\end{proof}

Another interesting question is the amount of anisotropic distortion compatible with the monotone causality of a standard Cartesian stencil in $\R^m.$
For $m=2$  and the stencil in Figure \ref{fig:4pt_and_8pt_stencils}(a), $\Vf(\x)$ simply needs to lie in a rectangle whose vertices are at 
$\left( -f_{\text{\tiny W}}, f_{\text{\tiny N}} \right), \,
\left( f_{\text{\tiny E}}, f_{\text{\tiny N}} \right), \,
\left( f_{\text{\tiny E}}, -f_{\text{\tiny S}} \right), \,
\left( -f_{\text{\tiny W}}, -f_{\text{\tiny S}} \right).$
Here the subscripts indicate directions of motion and the coordinate system is centered at $\x$.

But as we already saw, with $m=3$ and the standard 6 point stencil, the answer is more complicated and depends on 
$f(\ba)$ in all directions parallel to coordinate planes (comprising $\Vf^{k}$'s).
For the example at the start of this subsection, 
$ \Big(  \Psi^1 \cap \Psi^2 \cap \Psi^3 \Big) = \{ (v_1, v_2, v_3) \in \R^3_{+,0} \; \mid \; \max(v_1+v_2, v_2 + v_3, v_1 + v_3) \leq 1 \}.$ 
Thus, Theorem \ref{thm:mc_speed_profile_r3} guarantees that this simplex is MC only if $f(\hat{\bxi}) \leq \sqrt{3}/2.$
For another interesting 3D example, consider a speed profile that is known to be isotropic (i.e., $f(\ba) = F$) but only for $\ba$'s that are parallel to coordinate planes
$(x_1,x_2),  \, (x_2,x_3),$ and $(x_1,x_3)$.  Theorem \ref{thm:mc_speed_profile_r3} guarantees that the standard 6-point stencil will be monotone causal if
the full 
$\Vf(\x)$ lies within a {\em tricylinder} 
$ \Big(  \Psi^1 \cap \Psi^2 \cap \Psi^3 \Big) = \{ (v_1, v_2, v_3) \in \R^3 \; \mid \; \max(v_1^2+v_2^2, v_2^2 + v_3^2, v_1^2 + v_3^2) \leq F^2 \}.$

We also note that in principle the argument in Theorem \ref{thm:mc_speed_profile_r3} can be recursively extended to any $m>3$. (The second half of the proof can be adopted for any $m$, but the conditions on $\partial \Xi_m$ will be more complicated.  E.g., for $m=4$, the current theorem provides conditions 
for all 3-dimensional faces of  $\Xi_4$.)
We do not pursue this extension because the criteria would not be as geometrically suggestive and because grid discretizations of \eqref{eqn:stationary_hj} are most commonly used when $m=2$ or $3.$

Similarly to the 2D case, for a smooth convex speed profile, the conditions of Theorem \ref{thm:mc_speed_profile_r3} can be replaced by criteria verified on the facets of the simplex only. The following Proposition is a 3D analog of Proposition \ref{thm:normal_vec_conditions}.
It also provides an explicit and easier to verify version of the 3D causality condition derived in 
\cite[Section 2.3]{desquilbet2021single}.

Let $\nhat(\tilde{\bv})$ denote an outward pointing normal vector to $\Vf$ at any $\tilde{\bv} \in \Vf^r.$
The support hyperplane to $\Vf$ at $\tilde{\bv}$ is specified by $(\bv - \tilde{\bv}) \cdot \nhat(\tilde{\bv}) = 0.$
We define the $r$-th ``tangential envelope'' of $\Vf_s$ as 
$$
\Phi^r = \left \{ \bv = \theta_1 \bv_1 + \theta_2 \bv_2 + \theta_3 \bv_3 \, \mid \,
\theta_i \geq 0 \; \text{ for } i=1,2,3 \text{ and } (\bv - \tilde{\bv}) \cdot \nhat(\tilde{\bv}) \leq 0 \text{ for all } \tilde{\bv} \in \Vf^r 
\right\}.
$$

\begin{proposition}[Normal MC verification in 3D.]\label{prop:3D_tangent_cond}
Suppose that the speed profile $\Vf_s$ is convex and smooth.  
A simplex $s \in \M(\x)$ with vertices $\x=\bm{0}, \z_1, \z_2, \z_3$ and corresponding velocities $(\bv_1, \bv_2, \bv_3)$
is guaranteed to be monotone causal if\\
\begin{enumerate}
\item
$\bv_i \cdot \nhat(\bv_j) \geq 0,  \quad \forall i,j \in \{1,2,3\};$
\item
$\bv_r \cdot \nhat(\tilde{\bv}) \geq 0,  \quad \forall r \in \{1,2,3\}, \; \tilde{\bv} \in \Vf^r.$
\end{enumerate}
\end{proposition}
\proof{Proof: }
By Proposition \ref{thm:normal_vec_conditions}, the first condition implies that $\Vf^r \subset \Pi^r$ for all $r \in \{1,2,3\}.$\\
We note that $\Vf_s \subset \Phi^r$ due to convexity of $\Vf$ and that $\partial \Phi^r$ is tangential to $\Vf$ along the $\Vf^r \subset \partial \Psi^r.$
Thus, $\Vf_s \subset \Psi^r$ if and only if $\Phi^r \subset \Psi^r,$ which is ensured by the second condition above.
This shows that $s$ satisfies the criteria listed in Theorem \ref{thm:mc_speed_profile_r3}.
\Halmos 
\endproof

\subsection{Monotone $\delta$-Causality in $\R^2$} 
The MC conditions derived above are fully based on $\Vf$ and the set of directions represented in a particular simplex; i.e., the distances  $|\z_i - \x|$
were irrelevant in specifying the MC-constrained set for $\Vf_s.$  This is not the case for monotone $\delta$-causality as we show below for $m=2.$ 
\begin{theorem}[$\delta$-Causality Condition on $\Vf$ in $\R^2$]\label{thm:delta_caus_r2}
Suppose that a simplex $s \in \Scal(\x)$ formed by vertices $\x=\bm{0}, \,  \z_1,$ and $\z_2$ is monotone causal for a specific $\Vf(\x)$ and
that $0 < \delta \leq \min(C_1= \frac{|\z_1|}{f_1},  \, C_2=\frac{|\z_2|}{f_2}).$
This simplex is monotone $\delta$-causal if $\Vf_s$ is contained in a quadrilateral with vertices at $\x= \bm{0}$, $\bv_1$, $\bv_2$, and $\bw(\delta) = \theta^{\#}_1(\delta)\bv_1 + \theta^{\#}_2(\delta)\bv_2,$ where 
\begin{equation}
\label{eq:quasikite_vertex}
\theta^{\#}_1(\delta) = \frac{C_1(C_2 - \delta)}{C_1C_2 - \delta^2} \quad \text{ and } \quad \theta^{\#}_2(\delta) = \frac{C_2(C_1 - \delta)}{C_1C_2 - \delta^2}.
\end{equation}
\end{theorem}
\begin{proof}{Proof: }
For every $\bxi \in \Xi_2,$ the corresponding velocity  is $\bv(\bxi) = \theta_1 \bv_1 + \theta_2 \bv_2,$ where $\theta_1, \theta_2 \ge 0.$ 
By Theorem \ref{theorem:gen_caus_condition}, monotone $\delta$-causality of $C(\bxi)$ is guaranteed if 
$$
C(\bxi) \geq \xi_1 C_1 + \xi_2 \delta
\hspace{0.5cm} \text{and} \hspace{0.5cm} 
C(\bxi) \geq \xi_1 \delta + \xi_2 C_2.
$$
Dividing through by $C(\bxi)$ and using \eqref{eqn:lambda_mc_restriction}, we see that the equivalent conditions are 
$$
1\ge \theta_1+ \frac{\theta_2\delta}{C_2}  \hspace{0.5cm} \text{and} \hspace{0.5cm} 1\ge \theta_2+ \frac{\theta_1\delta}{C_1}.
$$
Each of these inequalities specifies a half-plane where $\bv(\bxi)$ is allowed to lie, with two restriction lines 
$L_1$ (passing through the points $\bv_1$ and $\frac{C_2}{\delta}\bv_2 = \frac{1}{\delta}\z_2$) and 
$L_2$ (passing through the points $\bv_2$ and $\frac{C_1}{\delta}\bv_1 = \frac{1}{\delta}\z_1$).
Solving for the intersection point $\bw(\delta),$ we obtain \eqref{eq:quasikite_vertex}.
\Halmos
\end{proof}
 
It is easy to see that this quadrilateral is fully contained in the MC parallelogram $\Pi$ and that $\bw(\delta) \rightarrow (\bv_1 + \bv_2)$
as $\delta \rightarrow 0.$    If $\delta = \min(C_1, C_2),$ one of the restriction lines coincides with the straight line connecting $\bv_1$ and $\bv_2,$
and the $\delta$-MC region becomes a triangle $(\bm{0}, \bv_1, \bv_2),$ ensuring that either $\frac{\z_1}{|\z_1|}$ or $\frac{\z_2}{|\z_2|}$ 
will be always optimal.  

Similarly to the MC setting, if $C^s(\bxi)$ is $\delta$-causal for all $\s \in \Scal(\x),$ then the entire stencil is $\delta$-MC at $\x$ and
the speed profile $\Vf(\x)$ is  inscribed within the union of bounding $\delta$-MC quadrilaterals which encompass all possible directions of motion. 
Figure \ref{fig:delta_causal_and_non} displays this result for two speed profiles on 
the 8-point stencil depicted in Figure \ref{fig:4pt_and_8pt_stencils}(b).  
\begin{figure*}[t]
\centering
$
\arraycolsep=1pt\def\arraystretch{0.1}
\begin{array}{cc}
\includegraphics[width = 0.4\textwidth]{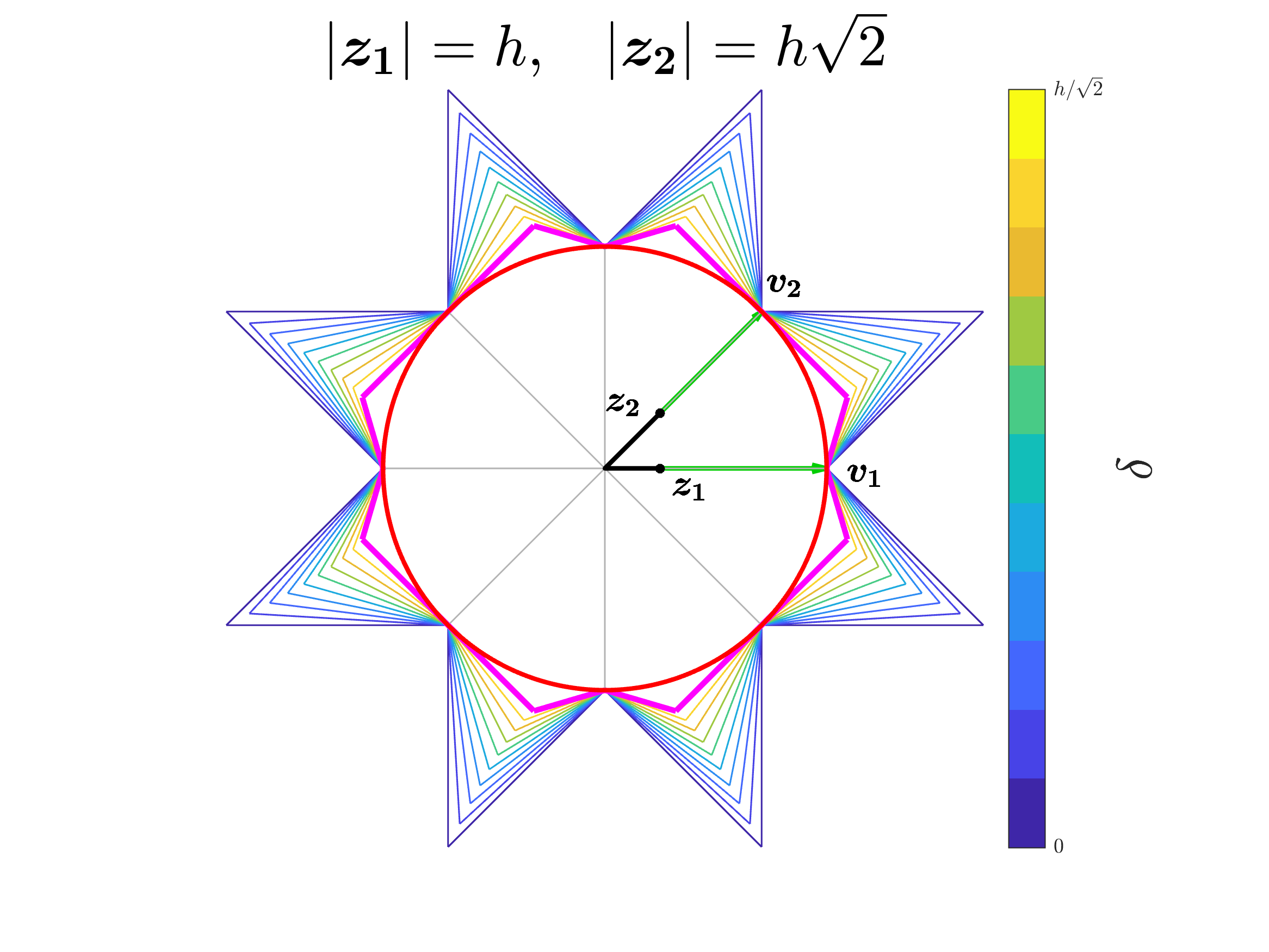} &
\includegraphics[width = 0.4\textwidth]{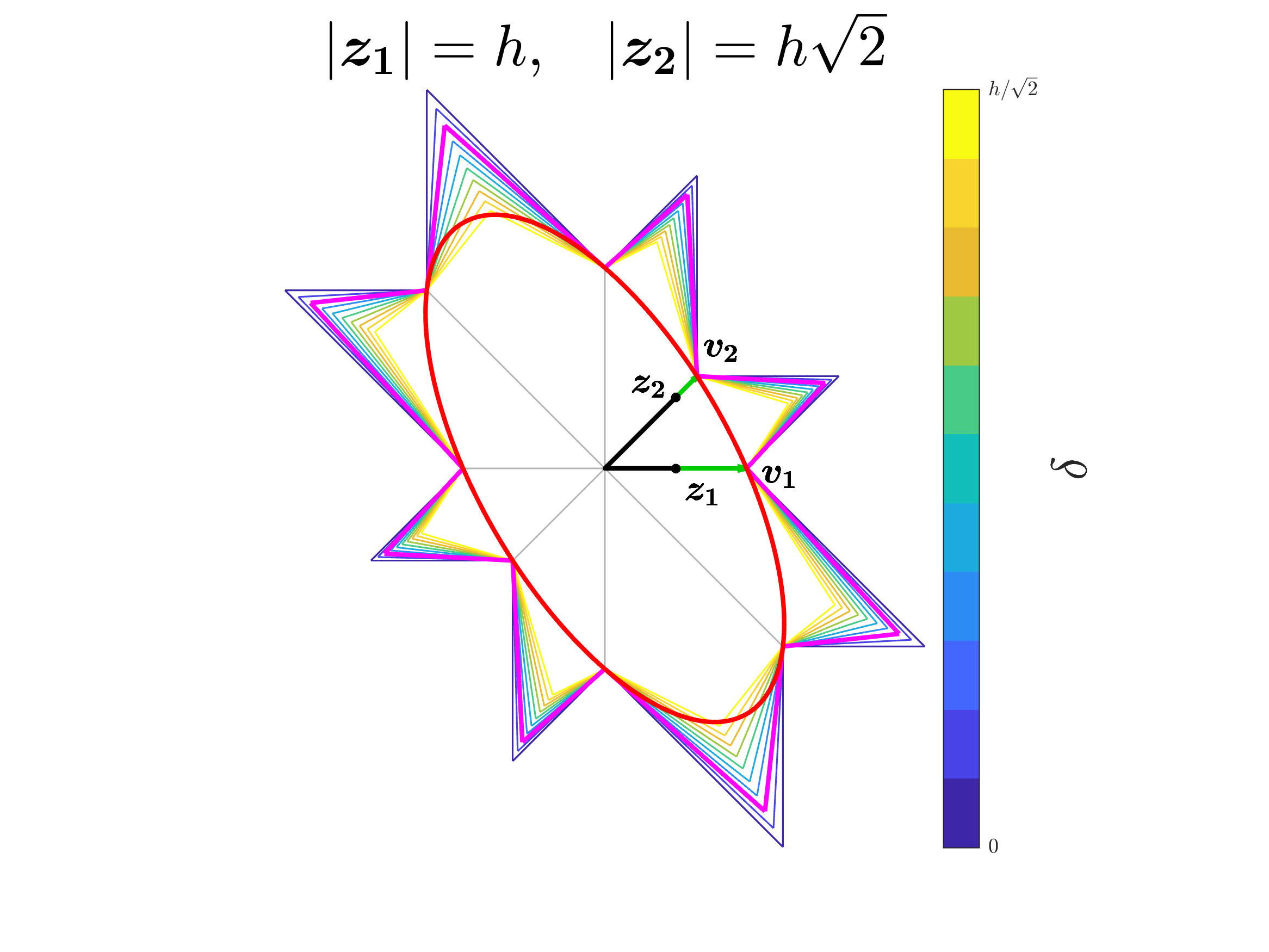} \\[-10pt]
\mbox{\footnotesize(a)} & \mbox{\footnotesize(b)}
\end{array}
$
\caption{Bounding $\delta$-MC ``sunflowers'' for the 8-point stencil (grid spacing $h = 1$ for $|\z_1| = h$, $|\z_2| = h\sqrt{2}$) found by fixing $f_i(\x) = f(\x, \ba_{\boldsymbol{e}_i})$ for 8 directions $(\x_i - \x)$ and varying $\delta.$ 
The $f_i$ are chosen to match an isotropic speed profile in panel (a) and an anisotropic elliptical speed profile in panel (b);  shown in red.  The direction vectors corresponding to $\z_1$ and $\z_2$ are shown in black, while the velocity vectors $\bv_1$ and $\bv_2$ are shown in green. The bounding sunflower corresponding to the maximum $\Delta(\x)$
applicable for this ($\Vf$, stencil) combination is shown
in magenta. \label{fig:delta_causal_and_non}}
\end{figure*}

The next obvious question is on the largest $\delta$ such that a specific simplex is $\delta$-MC.  The answer can be found from Proposition \ref{cor:max_delta_m2}, 
but the geometric derivation provided below is more natural.

\begin{proposition}[Max $\delta$ for Smooth, Strictly Convex $\Vf$ in $\R^2$]\label{thm:max_delta}
Suppose that $m=2$ and $\Vf(\x)$ is smooth and strictly convex.
Suppose also that a simplex  $s \in \Scal(\x)$ formed by $\x=\bm{0}, \,  \z_1,$ and $\z_2$ is MC for this $\Vf(\x).$
The maximal $\delta$ for which this simplex is $\delta$-MC can be computed as
 $\Delta(\x, s) = \min\{\delta_1, \delta_2\},$ where
\begin{equation}
\label{eq:max_simp_delta}
\delta_1 = C_2\frac{\bv_2 \cdot \nhat_1}{\bv_1 \cdot \nhat_1} 
\hspace{0.3cm} \text{ and } \hspace{0.3cm} 
\delta_2 = C_1\frac{\bv_1 \cdot \nhat_2}{\bv_2 \cdot \nhat_2},
\end{equation}
and $\nhat_i$'s denote normal vectors to $\Vf(\x)$ at respective $\bv_i$'s. 
\end{proposition}
\begin{proof}{Proof:}
Starting with the MC-parallelogram, gradually increasing $\delta$ moves the  lines $L_1, L_2$ (described in Theorem \ref{thm:delta_caus_r2}) and shrinks the $\delta$-MC quadrilateral.
For a convex and smooth speed profile, this process can be continued until one of the 
$L_i$'s becomes tangential to $\Vf(\x)$ at $\bv_i.$
Suppose this happens with $i=1$ first, when $\delta = \delta_1.$  
The line $L_1$ passes through the points $\bv_1$ and $\frac{C_2}{\delta_1}\bv_2.$  Tangentiality of $L_1$ means that 
$\left(\frac{C_2}{\delta_1}\bv_2 - \bv_1 \right) \cdot \nhat_1 = 0,$ which is equivalent to the expression in \eqref{eq:max_simp_delta}.
The case where $L_2$ becomes tangential first similarly yields the expression for $\delta_2.$
\Halmos
\end{proof}

If either of these lines is already tangential to $\Vf(\x)$ when $\delta=0,$ this means that the simplex is MC but not $\delta$-MC for any $\delta > 0;$
see Figure \ref{fig:convex_angle_diagram}(b).  
To find the largest $\delta$ such that the entire stencil is  $\delta$-MC at $\x,$ we define $\Delta(\x) = \min_{s \in \Scal(\x)} \Delta(\x, s).$
Similarly, let $\Delta(X) = \min_{\x \in X} \Delta(\x).$  For an MC discretization,  if $\Delta(X) > 0$ then Dial's method with the bucket width $\Delta(X)$ is also applicable.
We note that, since $C_i = \frac{|\z_i|}{f_i},$ for any fixed speed function $f(\x, \ba),$ the allowed bucket width $\Delta(X)$ will be shrinking under grid refinement.
\begin{figure*}[h]
\centering
\includegraphics[width = 1\textwidth]{figs/f_and_causality/normal_restrictions.tikz} 
\caption{Simplex $s$ and causality conditions on $\Vf_s$ for three different convex speed profiles. 
Tangent lines to $\Vf$ at are labeled $T_1$ and $T_2$ respectively and plotted in dashed blue. $\bv_1, \bv_2$, and $\alpha$ are equivalent in all three examples. $\Vf_s$ in panel (a) satisfies the monotone causality conditions posed in Theorem \ref{thm:bounding_parallelograms_r2} and the monotone $\delta$-causality conditions posed in Theorem \ref{thm:max_delta}. The corresponding sides of the resulting bounding quadrilateral are shown in solid purple. In this case, $L_2$ (shown in dotted purple) is tangential to $\Vf_s$, and $C(\bxi)$ is monotone $\delta$-causal in this simplex with $\delta = \delta_2$. In panel (b), $\Vf_s$ only satisfies the conditions of Theorem \ref{thm:bounding_parallelograms_r2} and $C(\bxi)$ cannot be monotone $\delta$-causal, as the side of the parallelogram opposite to $\bv_1$ is tangential to $\Vf$. Finally in panel (c), $\Vf_s$ violates the $\bv_1 \cdot \hat{\bn}_2 \ge 0$ condition in Proposition \ref{thm:normal_vec_conditions}, and $C(\bxi)$ cannot be monotone causal on this simplex.
\label{fig:convex_angle_diagram}}
\end{figure*} 

Proposition \ref{thm:max_delta} also provides a simpler derivation for some of the results previously found for specific 2D-stencils.
E.g.,  in the isotropic case $\Vf$ is just a ball; so, $\bv_i$ and $\nhat_i$ are parallel.  If $\beta$ is an angle between $\z_1$ and $\z_2,$
\eqref{eq:max_simp_delta} reduces to $\delta_1 = C_2 \cos \beta = \frac{|\z_2|}{f(\x)}\cos \beta$ and $\delta_2 = C_1 \cos \beta = \frac{|\z_1|}{f(\x)}\cos \beta.$
In the 8-point stencil of Figure \ref{fig:4pt_and_8pt_stencils}(b),  $|\z_1| = h, \, |\z_2| = h \sqrt{2},$ and $\beta = \pi/4.$  Thus, $\delta_1 = \frac{h}{f(\x)}, \, \delta_2 =  \frac{h}{\sqrt{2} f(\x) },$
and $\Delta(\x) = \delta_2.$  As a result, $\Delta(X) = \frac{h}{F_2 \sqrt{2}},$ which matches the bucket width derived in \cite{tsitsiklis1995efficient}. 
General 2D stencils and triangulated meshes can be treated similarly, using upper bounds on $|\z_i|$'s and $\beta.$

\section{OSSPs and Autonomous Vehicle Routing}\label{section:lane_change_formulation}
Routing modules in autonomous vehicles (AVs) have multiple interacting levels of  planning and  control which work together to bring the vehicle from its starting destination $\s$ to a predetermined destination $\bt$  \cite{michon1985critical}. 
Typically, the routing module first produces a Strategic Plan (SP), which is a set of deterministic, cost\footnote{E.g., time, fuel consumption, toll charges, passengers' comfort or a combination of these factors.}-minimizing, turn-by-turn directions from $\s$ to $\bt$.
SP is usually precomputed based on 
traffic and weather conditions using 
some label-setting method
to identify the deterministic cheapest path on a graph representing the road network.  
After that,
the system may determine a Tactical Plan (TP) outlining the planned timing and location of lane switch maneuvers (LSMs) necessary to execute the SP. Finally, the SP and TP are communicated to the Operational Control (OC) module which is responsible for planning the vehicle's continuous trajectory, steering, and dynamically accelerating / decelerating in order to follow the planned route. OC occurs in real time and is heavily constrained by safety overrides based on the actual observed vehicles.
If a planned LSM fails, this often makes it necessary to follow a suboptimal path and repeatedly recompute both SP and TP. 

The overall performance can be clearly improved by directly modeling the probabilistic nature of LSMs and minimizing the expected cost to target.
This is the essence of the new approach developed by Jones, Haas-Heger, and van den Berg \cite{jones2022lane}, 
which can be characterized as computing a combined stochastic Strategic / Tactical Plan (STP). 
They cast the routing problem as an SSP in which the states correspond to nodes within a lane-level road network such as the examples in Figure \ref{fig:lane_level_network}(a) and Figure  \ref{fig:roundabout_diagram}(b).
Using our terminology, their model is an example of an OSSP in which $|s| \leq 2$ for all $s \in \M(\x), \, \x \in X$. 
Deterministic modes (with $|s|=1$) are used to model the actions that normally don't fail (e.g., moving forward wherever lane changes are not allowed or available, or turning at an intersection without any lane changes).
On the other hand, a possible switch to each specific adjacent lane is encoded using a separate mode with $|s|=2.$
In \cite{jones2022lane}, each such mode is associated with exactly three LSM actions: the deterministic stay-in-lane, the deterministic ``forced'' LSMs, and one tentative LSM which succeeds with probability $\tilde{p} \in (0,1)$. Figure \ref{fig:lane_level_network}(b) shows an example of the LSMs available to the vehicle at $\x$ in mode $s_1 = \{\x_j, \x_k\}$. The vehicle may deterministically transition to $\x_j$, deterministically transition to $\x_k$, or use a stochastic transition leading to $\x_k$ or $\x_j$ with respective probabilities $\tilde{p}$ and $(1-\tilde{p})$. 
\begin{figure*}[h]
\centering
\includegraphics[width = 0.8\textwidth]{figs/lane_level_network.tikz} 
\caption{ 
Panel (a): Example lane-level road network representation of a three-lane highway. Lanes are discretized into cells of length $D$ meters, and each node marks the center of a cell. The vehicle travels from the starting point $\s$ to the destination $\bt$ via a series of planned LSMs. Panel (b): Actions available at node $\x$ in mode $s_1 = \{\x_j, \x_k\} \in \M(\x)$. The vehicle may continue driving in the current lane and directly transition to $\x_j$ (solid purple arrow), forcefully switch lanes and directly transition to $\x_k$ (solid red arrow), or attempt a tentative lane change (dashed blue arrow).  The other mode available at $\x$ is $s_2 = \{\x_j, \x_i\}$, encoding a possible switch to another lane.
\label{fig:lane_level_network}}
\end{figure*}

The vehicle's transitions are penalized as  
\begin{equation}\label{eqn:jurs_cost}
K(0) = g(\x), \hspace{0.5cm} K(\tilde{p}) = g(\x) + \tilde{p}g^{}_{1}, \hspace{0.5cm} \text{and} \hspace{0.5cm} K(1) = g(\x) + g^{}_{1} +(1-\tilde{p})g^{}_{2}
\end{equation}
where $g(\x)$ is the cost of traveling to the next node in the same lane, 
$g^{}_{1} > 0$ is the additional cost incurred if a tentative LSM is successful, and $g^{}_{2} > g^{}_{1}$ is a large penalty incurred for forcing the lane switch. For an unforced/tentative LSM, the transition probability is modeled as $\tilde{p} = 1 - e^{-\alpha D}$ where $D$ is the distance to the next node in the same lane, and $\alpha$ is the success rate determined from local traffic data. The cost $K(p)$ in \eqref{eqn:jurs_cost} is monotonically increasing in $p$ based on the assumption that the instantaneous cost of changing lanes will always be higher than staying in the current one.

Jones et al. have proved in \cite{jones2022lane} that a Dijkstra-like method is applicable to this SSP provided  
\begin{equation}
\label{eq:Jur's_condition}
g(\x) \; \ge \; \alpha D g^{}_{2}.
\end{equation}
The same result also follows as a direct application of our Theorem \ref{theorem:gen_caus_condition}, which for $m=2$ is equivalent to Theorem \ref{theorem:gen_caus_condition_improved}. Since $K(p)$ is monotonically increasing on $[0,1]$, we only need to check the inequality \eqref{eq:original_MC_cond} for $r = 1,$ which requires 
\begin{equation}
\label{eq:true_MC_cond_for_Jur}
K(\tilde{p}) = g(\x) + \tilde{p}g^{}_{1} \; \ge \; \tilde{p} \left[ g(\x) + g^{}_{1} + (1-\tilde{p})g^{}_{2}\right],
\end{equation} 
or, equivalently, $g(\x) \ge \tilde{p} g^{}_{2}.$
By the convexity of exponentials, $\alpha D \ge \tilde{p},$ and so \eqref{eq:Jur's_condition} implies \eqref{eq:true_MC_cond_for_Jur}.
We note that, outside of MSSPs \cite{vladimirsky2008label},
this is the first known example of a monotone causal OSSP. 
Moreover, if \eqref{eq:Jur's_condition} holds, it is easy to use our Proposition \ref{cor:max_delta_m2} to show that this OSSP is also monotone $\delta$-causal with $\Delta = \min_{\x} g(\x) - \tilde{p}g^{}_{2}.$

The theory developed in \S \ref{section:ssps_and_ossps} and \S \ref{section:label_setting_and_mc} significantly widens the scope of STP models that can be treated with label-setting methods\footnote{The following material is subject to Provisional US Patents 10471-01-US and 10471-02-US.}
.  E.g., we can now similarly account for a larger number of tentative LSM actions and different cost models with monotonically increasing and convex $K(p).$
The OSSP framework also allows for a useful reinterpretation of the lane change success probability in terms of the  \emph{lane change urgency}\footnote{
It might seem that the urgency of an LSM should be fully determined by the vehicle's distance to its next preplanned turn, which has to be executed from the lane we are trying to switch to.
But if similarly cheap alternate routes are available, missing a planned turn or highway exit may only have a very minor effect on the vehicle's total time (or other relevant expenditures) up to $\bt.$  Thus, an accurate assessment of urgency should take into account the global road network structure and traffic patterns.
},
which reflects the degree of controller's willingness to alter the vehicle's velocity to ensure a successful lane change \cite{gipps1986model}.
In reality, drivers \emph{gradually} increase or decrease LSM urgency in response to the local traffic conditions or nearby infrastructure \cite{ahmed1999modeling}. 
Thus, urgency exists on a continuous spectrum in which the stay-in-lane maneuver corresponds to no urgency ($p = 0$), the forced lane change maneuver corresponds to full urgency ($p = 1$), and all LSMs with stochastic outcomes correspond to intermediate urgency levels ($p \in (0,1)$). 

The simplest generalization of the model in \cite{jones2022lane} is to allow for more than one (but finitely many) intermediate LSMs, which might be executed 
\emph{progressively} at $\x$. The cost of a forced LSM in \eqref{eq:Jur's_condition} reveals that upon selecting the forced lane change maneuver, the vehicle attempts to switch lanes at the intermediate level first, and then only forces the LSM should that initial attempt fail. As the LSM urgency increases between the first attempt and second attempt, the LSMs themselves in this framework may also be described as \emph{escalating}. 

To be more precise, suppose there are $L+1$ available LSMs with associated success probabilities $p^{}_{\ell} \in \mathcal{P}$ such that
$0= p_0 < \dots <  p^{}_{\ell} <  \dots < p_{L} = 1.$ 
The stay-in-lane cost is $K(p_0) = g(\x)$, and maneuvers with $p^{}_{\ell} > 0$ are subject to additional penalties $Y_{\ell} > 0$ which are monotonically increasing in $\ell$ and incurred upon a successful lane change at the corresponding urgency level. 
If these actions are executed in escalating manner, their expected cost has the form 
\begin{equation}\label{eqn:gen_progressive_finite_cost}
K(p^{}_{\ell}) = p_{\ell-1}K(p_{\ell-1}) + (1-p_{\ell-1})\left[K(p_{\ell-1}) + Y_{\ell}\right]
 = K(p_{\ell-1}) + (1-p_{\ell-1}) Y_{\ell}, 
 \qquad \ell = 1, \dots, L.
\end{equation}
The convexity of $K$ is ensured when 
$\frac{K(p^{}_{\ell+1}) - K(p^{}_{\ell})}{p^{}_{\ell+1} - p^{}_{\ell}} \ge \frac{K(p^{}_{\ell}) - K(p_{\ell-1})}{p^{}_{\ell} - p_{\ell-1}}$ 
holds for all $\ell = 1, \dots L-1$, which implies that all of these actions are actually ``useful'' in the sense of Observation \ref{obs:ssp_useless_non_optimum}. 
If the above inequality does not hold, some of these LSM actions can be safely removed (in pre-processing) to reduce the computational cost of label-setting methods without affecting the value function.

By direct application of Theorem \ref{theorem:gen_caus_condition}, the OSSP will be monotone $(\delta$-)causal when
\begin{equation*}
K(p_{\ell-1}) + (1-p_{\ell-1})Y_{\ell}  \ge p^{}_{\ell}\left[K(p_{L-1}) + (1-p_{L-1})Y_L \right] + (1-p^{}_{\ell})\delta, 
 \qquad \ell = 1, \dots, L-1.
\end{equation*}
We note that the exact model used in \cite{jones2022lane} can be recovered by taking $L = 2$, $p_1 = \tilde{p}$, $Y_1 = \tilde{p}g^{}_{1}$, and $Y_2 = g^{}_{2} + g^{}_{1}.$ 

Whether or not the above progressive escalation framework is realistic is debatable:   
a lane change may take six or more seconds to complete \cite{fitch2009analysis} and attempting multiple consecutive LSMs before reaching the successor node in the same lane might be impossible.  But our framework can be similarly used for any increasing sequence of costs $K_{\ell} = K(p^{}_{\ell})$.  Due to the monotonicity of $K,$ the inequality 
\eqref{eq:original_MC_cond}
 has to be enforced for $r=1$ only and will hold automatically for $r=2.$
The lane switch mode will be monotone $(\delta$-)causal as long as 
\begin{equation}\label{eqn:no_escalating_finite_cost}
K(p) \; \geq \; p K(1)  + (1-p)\delta
\end{equation}
holds for every attainable $p < 1.$  This criterion works with models of finitely many urgency levels ($p \in \{p_0, ..., p_{L}\}$)
and also for models with a continuous spectrum of urgency levels ($p \in[0,1]$); see also Figure \ref{fig:mc_geometric_illustration}.
The maximum allowable $\delta$ can be similarly computed using Proposition \ref{cor:max_delta_m2}.

We also mention two other LSM-cost models illustrated by the numerical experiments covered in the next subsection.
For the continuous urgency spectrum, one of the simplest (monotone, convex) cost models is quadratic
\begin{equation}\label{eqn:smooth_cont_c}
K(p) = \beta_s(\x) p^2 + \gamma_s(\x)
\end{equation} 
where $\beta_s(\x), \gamma_s(\x) > 0 $ are constants which may reflect traffic conditions near $\x$ or learned driver behavior. 
Using this $K(p)$ in \eqref{eqn:no_escalating_finite_cost} and simplifying, we see that this lane switch mode will be $\delta$-monotone causal as long as $\beta_s(\x) + \delta \leq \gamma_s(\x).$
Thus, Dijkstra's method will be applicable if $\beta_s(\x) \leq \gamma_s(\x)$ for all $\x \in X, \, s \in \M(\x).$  If this inequality is strict,  we can also use Dial's method with buckets of width
$\Delta(X) = \min_{\x \in X} \min_{s \in \M(\x)} \left( \gamma_s(\x) - \beta_s(\x) \right).$

Alternatively, we can start with an (increasing, convex) sequence of urgency levels and associated costs $\left(p_{\ell}, K_{\ell} \right)$ learned from local traffic data or from the vehicle's performance analytics, and then extend it to a continuous urgency spectrum $p \in [0,1]$ by defining $K(p)$ through a suitable interpolation.

For example, when $L = 2$, a natural choice of interpolant is a quadratic Rational B\'ezier Curve (RBC) with suitable control points and weights chosen to ensure that the resulting smooth approximation to $K(p)$ is monotone, convex, and monotone ($\delta$-)causal by Theorem \ref{theorem:gen_caus_condition}. The RBC is described parametrically as $\boldsymbol{B}(t) = (p(t), K(t))$, and the entire curve sits within the convex hull of a given set of $L+1$ control points. The curve's shape is governed by those control points and a set of corresponding weights $\{\omega_0, \dots, \omega_L \}$ which determine how much the curve bends toward each control point \cite{piegl1996nurbs}. Assuming that the original three-point $\left(p_{\ell}, K_{\ell} \right)$ sequence is $\delta$-MC, we can take the first control point to be $(0, K(0))$ and the last control point to be $(1, K(1))$, with the remaining control point chosen to lie at the intersection of the lines $K = K_0$ and $K = p K_2  + (1-p)\delta$ to ensure that the resulting curve is both monotonically increasing and monotone ($\delta$-)causal by Theorem \ref{theorem:gen_caus_condition}. Since RBCs are guaranteed to pass through their first and last control points, we can set $\omega_0 = \omega_2 = 1$, and we construct a system of two linear equations (one for $K(t) = K_1$ and one for $p(t) = p_1$) to compute $t$ and the value of $\omega_1$ which ensures that the RBC also passes through $(p_1, K_1)$ as required. 
We note that, for this choice of control points, Theorem \ref{theorem:gen_caus_condition} guarantees that this OSSP will be monotone causal though not $\delta$-MC for any $\delta>0.$

Evaluating the usefulness (and the optimal urgency) of merging to an adjacent lane is the essential part of computing the value function.
E.g., starting from the point $\x$ in Figure \ref{fig:lane_level_network}(b),
we need to find $p$ that minimizes 
$\left( K(p) + (1-p)U_j + p U_k \right).$  For a continuous spectrum of $p$ values, the availability of $K'(p)$ can be used to perform this minimization
either analytically (e.g., with the quadratic $K$ model) or semi-analytically (e.g., with the RBC-based $K$ model described above)
to ensure the computational efficiency.

\subsection{Numerical Examples}\label{subsec:numerical_results}
We use a Dijkstra-like method to determine the optimal STP for a vehicle traveling through several different lane-level road networks.
In all examples, $D = 10$ meters, and the direction of traffic flow between successor nodes is indicated by arrows. 
When visualizing optimal STPs, we use solid arrows to indicate when a deterministic LSM is optimal and dashed arrows to indicate when the optimal LSM is deterministic. 
The color of such arrows indicates the optimal $p^*,$ whose value (rounded for the sake of readbility) is also shown as a label.

\vspace*{3mm}
\noindent
{\bf Example 1: a 3-lane highway; 4 urgency levels.}\\
We first determine the STP for an autonomous vehicle traveling on a $1500$ meter section of a three-lane highway.
The target offramp $\bt$ is located in the left lane, but the vehicle is incentivized to move in the right lane as much as possible.
This is a common preference in routing autonomous trucks; following \cite{jones2022lane}, we accomplish this by defining the cost
of moving in the current lane as
$g(\x) = D\sigma_i$
where $\sigma_i = 1 + i\varepsilon$, the lanes are enumerated from right to left ($i=0,1,2$), and $\varepsilon > 0$ is a fixed penalty. 
We also model an onramp, with additional traffic entering the right lane from a merge lane at node $\x_{\#}$, exactly $1km$ from the target.
The right-lane nodes within $10$ meters of $\x_{\#}$ experience an additional cost $\mu = 35$ in their $g(\x)$ to account for 
moderate congestion due to the traffic merging in from the onramp.
Nodes in the left and right lanes have only one mode, while nodes in the middle lane have two (switch to the left or to the right). In each mode, the system may choose between four available LSMs of increasing urgency
with associated success probabilities
$0 = p_0 < p_1 < p_2 < p_3 = 1$.
$K_{\ell} = K(p_{\ell})$ is given by \eqref{eqn:gen_progressive_finite_cost} with $L = 3$, $p_1 = \tilde{p}$, $Y_1 = \tilde{p}g^{}_{1}$, $p_2 = 0.2$, $Y_2 = 2$, and $Y_3 = 40$.
The expression for $\tilde{p}$ 
mirrors the definition
in \cite{jones2022lane}, and we set $\alpha = 0.01$, $\varepsilon = 0.1$, and $g^{}_{1} = 3$. A plot of $K(p_{\ell})$ for vehicles driving in the right lane is displayed in Figure \ref{fig:ex1}(a). The resulting $K(p_{\ell})$ is monotone $\delta$-causal by Theorem \ref{theorem:gen_caus_condition} with $\delta_{\star} \approx 4.095$.

We compute the STP on this stretch of a highway and show the optimal policy for a smaller segment (between $970$ meters and $1040$ meters away from $\bt$) in Figure \ref{fig:ex1}(b).
It illustrates how additional intermediate LSMs allow the vehicle to dynamically adjust the urgency of its attempts in response to anticipated locations of higher cost. 
E.g., vehicles in the right lane prefer to temporarily switch left as they approach the onramp
with an increasing LSM urgency as they get closer. 
At $\x_{\#}$ this urgency decreases again (we have already suffered through most of the delay), and beyond $\x_{\#}$ those traveling on the right will prefer 
to continue in their current lane until they are much closer to the target.
The vehicles traveling in the center lane will only start trying to switch right after passing the onramp.
However, even after  $\x_{\#}$ their urgency level to switch right remains low -- this reflects the fact that the target offramp is already not too far and on the left.

We also show the deterministic optimal SP on the same highway segment in Figure \ref{fig:ex1}(c).  
Following these instructions would result in much more aggressive lane switching from the central and left lanes and,
not surprisingly, would yield a higher overall cost to target.
Comparing SP and STP over all nodes in this network, STP results in the median, average, and maximum cost reduction
of  $5.23\%, \, 5.49\%,$ and $15.65 \%$ respectively.

\begin{figure}[htbh]
$
\arraycolsep=1pt\def\arraystretch{0.1}
\begin{array}{ccc}
\includegraphics[width=0.33\linewidth]{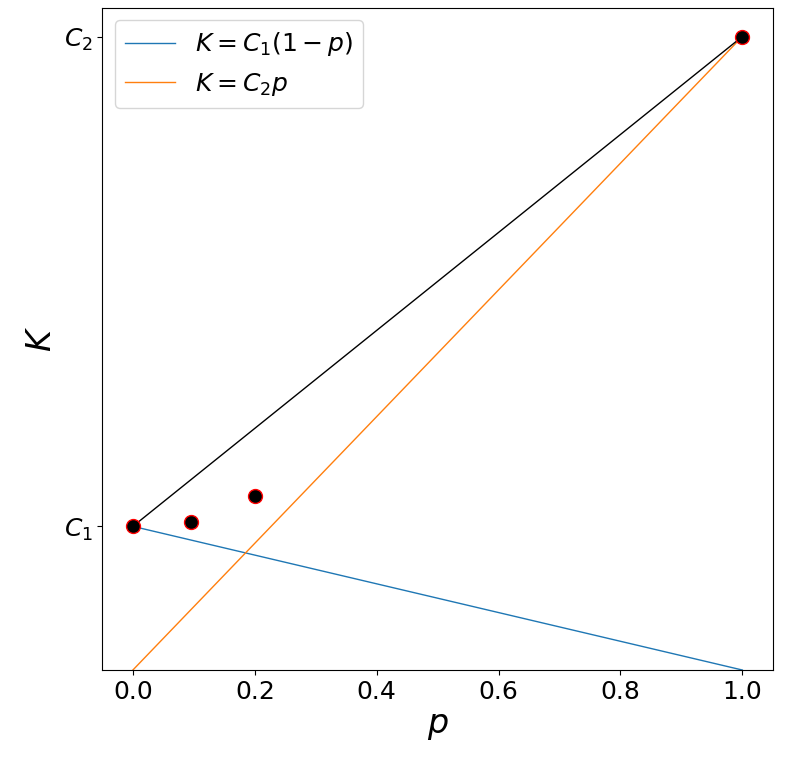} &
\includegraphics[width=0.33\linewidth]{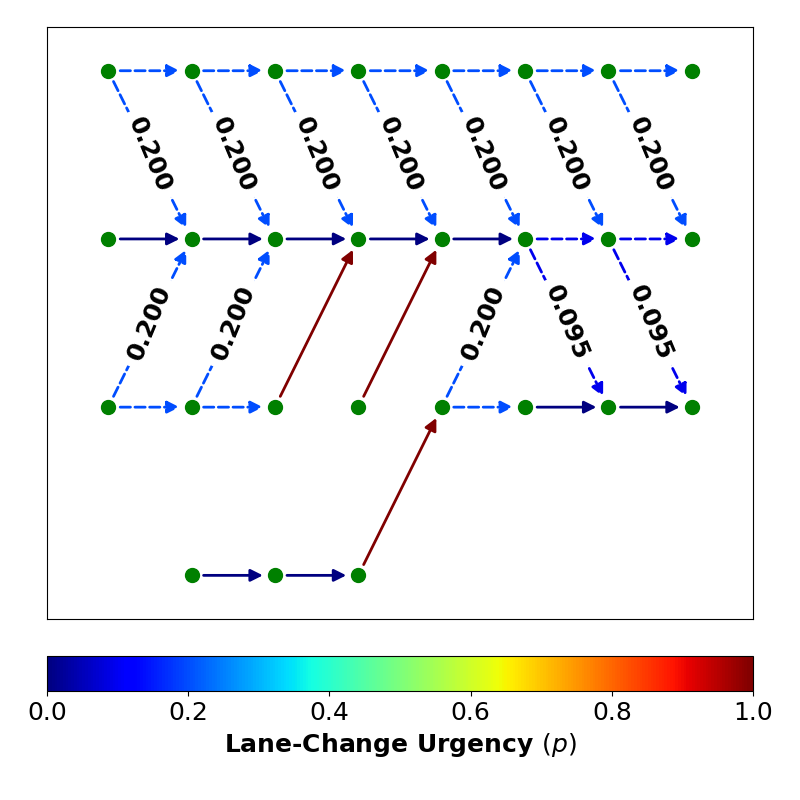} &
\includegraphics[width=0.33\linewidth]{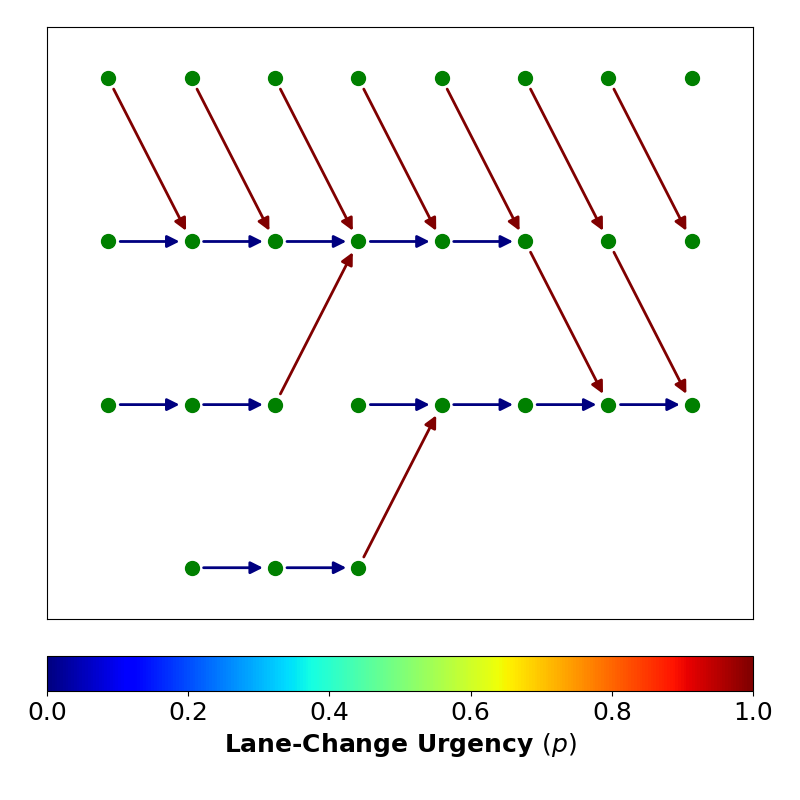}\\[-1pt]
\mbox{\footnotesize (a)} & \mbox{\footnotesize(b)} & \mbox{\footnotesize(c)}
\end{array}
$
\caption{Lane-level planning on a three-lane highway between $970$ meters and $1040$ meters away from $\bt$. Panel (a): Finite cost function $K(p)$ for vehicles traveling in the right lane (except within $10$ meters from $\x_{\#}$ on either side) for the $L=3$ escalating LSMs framework described above. $C_1$ is the stay-in-lane cost, $K(0)$, and $C_2$ is the forced LSM cost, $K(1)$. Panel (b): STP when there are four available LSMs per mode. The edges are colored corresponding to the success probability associated with the optimal LSM urgency level at $\x$. Panel (c): The deterministic SP when the stay-in-lane cost is $g(\x)$ and the lane change cost is equal to $K(1)$. 
\label{fig:ex1}}
\end{figure}

\vspace*{3mm}
\noindent
{\bf Example 2: a 3-lane highway; continuous urgency spectrum.}\\
Extending the previous example,
we now use
Rational B\'ezier Curves to construct a continuous cost function through the points $(p_0=0, K_0), (p_2, K_2)$, and $(p_3=1, K_3).$
These three points are defined for each lane in Example 1, and we know this action set is monotone causal by Theorem \ref{theorem:gen_caus_condition}; see Figure \ref{fig:ex1}(a). 

For this stretch of a three-lane highway,
there are \emph{four} RBCs to consider -- one RBC for each lane (since $g(\x)$ is lane-dependent) and 
one RBC for the slow-down (onramp merge) zone, which includes the right-lane nodes within $10$ meters of $\x_{\#}$. 
Each RBC has control points $\left\{(p_0, K_0), \left(\frac{K_0}{K_1}, K_0\right), (p_3, K_3 )\right\}$ and weights $\{\omega_0 = 1, \omega_1, \omega_2 = 1\},$
with $\omega_1$ chosen to make the RBC satisfy $K(p_2) = K_2.$
The resulting cost function for vehicles traveling in the right lane (except at the locations within $10$ meters of $\x_{\#}$) is displayed in Figure \ref{fig:ex2}(a). 
 \begin{figure}[h]
 \centering
$
\arraycolsep=24pt\def\arraystretch{0.1}
\begin{array}{cc}
\includegraphics[width=0.4\linewidth]{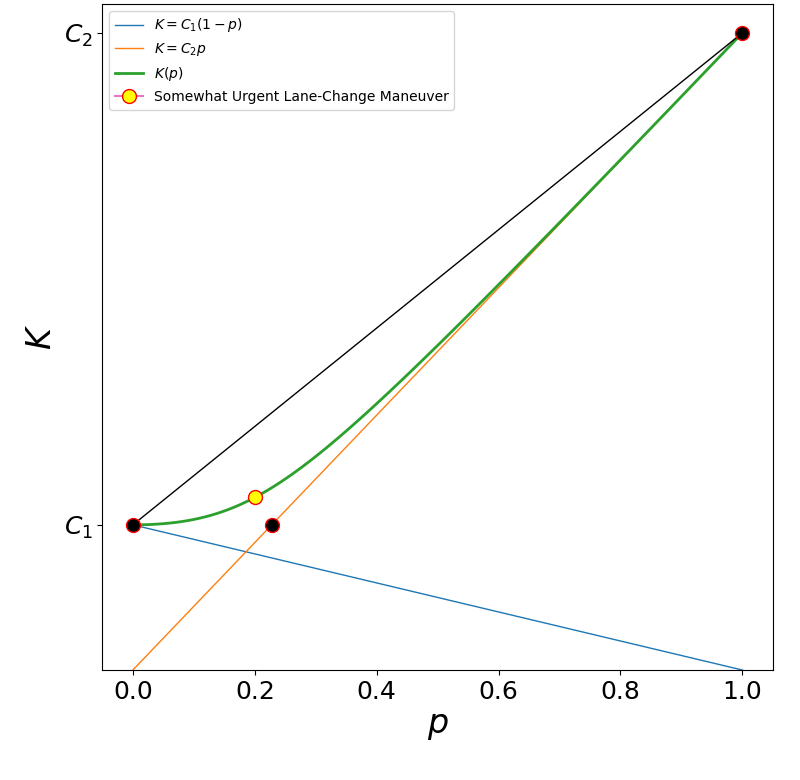} &
\includegraphics[width=0.4\linewidth]{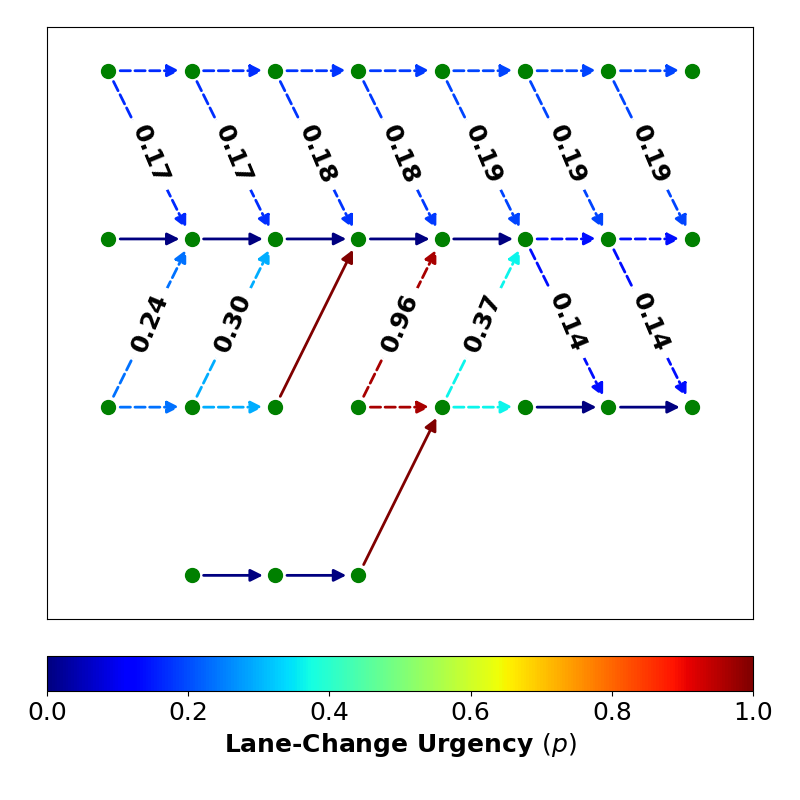} \\[-1pt]
\mbox{\footnotesize (a)} & \mbox{\footnotesize(b)} 
\end{array}
$
\caption{Lane-level planning on a three-lane highway between $970$ meters and $1040$ meters away from $\bt$ in the left lane. Panel (a): Rational B\'ezier cost function through three cost estimate datapoints for vehicles traveling in the right lane (except $10$ meters from the merge location, $\x_{\#}$). $C_1 = K(0), C_2 = K(1)$, and the point corresponding to the intermediate LSM from the dataset is plotted in yellow. Tangency of $K(p) = C_2p$ to the cost curve prohibits monotone $\delta$-causality by Theorem \ref{theorem:gen_caus_condition}. Panel (b): Resulting STP under the RBC cost structure between $970$ meters and $1040$ meters away from $\bt$. \label{fig:ex2}}
\end{figure}   

Since $K(p_1)$ under the RBC model is fairly close to $K_1$ defined in the previous subsection, it is instructive to compare the RBC-based optimal STP shown in Figure  \ref{fig:ex2}(b) with the STP based on just 4 discrete LSMs in Figure  \ref{fig:ex1}(b).   
While the schematic representation of these policies is similar,
the optimal behavior is significantly more nuanced in the case of RBCs, with gradual urgency build-ups highlighting the advantages of continuous spectrum models.

\vspace*{3mm}
\noindent
{\bf Example 3: multiple roundabouts; continuous urgency spectrum.}\\
We note that in both previous examples the dependency digraph $G_{\mu}$ is acyclic for every stationary policy.  This is due to the simplicity of the underlying road network (all arrows pointing rightward), which represented a stretch of a single highway.  As a result, the OSSPs are explicitly causal and \eqref{eqn:value_fn_vi} could be solved efficiently in a single Gauss-Seidel iteration (sweeping through the nodes from right to left) regardless of the properties of $K(p)$.  
But most road networks contain cycles and thus are not explicitly causal, making the MC properties (and applicability of label-setting methods) far more important.
Non-trivial cycles are ubiquitous due to numerous road intersections on a large map, but can also be found in smaller examples with ring roads and roundabouts.
The ``magic'' (or interconnected) roundabouts used in several cities of the United Kingdom present drivers with particularly interesting strategic choices when
parts of the roundabout are clogged with traffic.


 \begin{figure}[t]
$
\arraycolsep=1pt\def\arraystretch{0.1}
\begin{array}{cc}

\tikzset{every picture/.style={line width=0.75pt}} 

\begin{tikzpicture}[x=0.3pt,y=0.3pt,yscale=-1,xscale=1]

\draw [line width=2.25]  [dash pattern={on 2.53pt off 3.02pt}]  (189.73,365.93) .. controls (190.47,353.86) and (182.73,322.93) .. (212.73,268.93) ;
\draw [color={rgb, 255:red, 245; green, 166; blue, 35 }  ,draw opacity=1 ][line width=1.5]    (648.74,453.57) -- (539.01,366.48) ;
\draw [color={rgb, 255:red, 245; green, 166; blue, 35 }  ,draw opacity=1 ][line width=1.5]    (270.01,3.02) -- (270.38,91.92) ;
\draw [color={rgb, 255:red, 245; green, 166; blue, 35 }  ,draw opacity=1 ][line width=1.5]    (625.97,572.89) -- (480.97,458.43) ;
\draw [color={rgb, 255:red, 245; green, 166; blue, 35 }  ,draw opacity=1 ][line width=1.5]    (120,369.46) -- (0.74,462.84) ;
\draw [color={rgb, 255:red, 245; green, 166; blue, 35 }  ,draw opacity=1 ][line width=1.5]    (442.01,83.45) .. controls (443.01,99.2) and (441.01,108.2) .. (458.01,126.2) ;
\draw [color={rgb, 255:red, 245; green, 166; blue, 35 }  ,draw opacity=1 ][line width=1.5]    (196.01,130.2) .. controls (213.01,117.2) and (211.01,97.2) .. (211.01,83.45) ;
\draw [color={rgb, 255:red, 245; green, 166; blue, 35 }  ,draw opacity=1 ][line width=1.5]    (87,328.46) .. controls (108.01,314.2) and (112.01,290.2) .. (110.01,279.2) ;
\draw [color={rgb, 255:red, 245; green, 166; blue, 35 }  ,draw opacity=1 ][line width=1.5]    (192,511.46) .. controls (212.01,501.2) and (234.01,497.2) .. (246.01,508.2) ;
\draw [color={rgb, 255:red, 245; green, 166; blue, 35 }  ,draw opacity=1 ][line width=1.5]    (192,511.46) -- (65.01,610.48) ;

\draw [color={rgb, 255:red, 245; green, 166; blue, 35 }  ,draw opacity=1 ][line width=1.5]    (405.01,510.2) .. controls (425.01,501.2) and (453.01,503.2) .. (458.01,510.48) ;
\draw [color={rgb, 255:red, 245; green, 166; blue, 35 }  ,draw opacity=1 ][line width=1.5]    (603.01,624.93) -- (458.01,510.48) ;
\draw [color={rgb, 255:red, 245; green, 166; blue, 35 }  ,draw opacity=1 ][line width=1.5]    (567.99,320.48) .. controls (548.01,303.2) and (544.01,275.2) .. (549.01,277.2) ;
\draw [color={rgb, 255:red, 245; green, 166; blue, 35 }  ,draw opacity=1 ][line width=1.5]    (657.74,392.57) -- (567.99,320.48) ;

\draw  [color={rgb, 255:red, 43; green, 85; blue, 234 }  ,draw opacity=1 ][line width=1.5]  (110.36,304.24) .. controls (110.36,183.22) and (208.47,85.11) .. (329.49,85.11) .. controls (450.52,85.11) and (548.63,183.22) .. (548.63,304.24) .. controls (548.63,425.27) and (450.52,523.38) .. (329.49,523.38) .. controls (208.47,523.38) and (110.36,425.27) .. (110.36,304.24) -- cycle ;
\draw [color={rgb, 255:red, 245; green, 166; blue, 35 }  ,draw opacity=1 ][line width=1.5]    (176,459.46) -- (49.01,558.48) ;
\draw [color={rgb, 255:red, 245; green, 166; blue, 35 }  ,draw opacity=1 ][line width=1.5]    (87,328.46) -- (-6.26,398.84) ;
\draw [color={rgb, 255:red, 245; green, 166; blue, 35 }  ,draw opacity=1 ][line width=1.5]    (386.01,6.02) -- (386.01,90.45) ;
\draw [color={rgb, 255:red, 245; green, 166; blue, 35 }  ,draw opacity=1 ][line width=1.5]    (442.01,6.02) -- (442.01,83.45) ;
\draw [color={rgb, 255:red, 245; green, 166; blue, 35 }  ,draw opacity=1 ][line width=1.5]    (211.01,3.02) -- (211.01,83.45) ;
\draw  [draw opacity=0][dash pattern={on 2.53pt off 3.02pt}][line width=2.25]  (209.44,406.66) .. controls (207.54,421.59) and (197.11,440.01) .. (180.3,455.65) .. controls (178.88,456.97) and (177.45,458.24) .. (176,459.46) -- (156.59,430.15) -- cycle ; \draw  [dash pattern={on 2.53pt off 3.02pt}][line width=2.25]  (209.44,406.66) .. controls (207.54,421.59) and (197.11,440.01) .. (180.3,455.65) .. controls (178.88,456.97) and (177.45,458.24) .. (176,459.46) ;  
\draw  [draw opacity=0][dash pattern={on 2.53pt off 3.02pt}][line width=2.25]  (120,369.46) .. controls (132.55,361.16) and (153.31,357.04) .. (176.14,359.57) .. controls (178.06,359.79) and (179.96,360.04) .. (181.83,360.34) -- (172.29,394.18) -- cycle ; \draw  [dash pattern={on 2.53pt off 3.02pt}][line width=2.25]  (120,369.46) .. controls (132.55,361.16) and (153.31,357.04) .. (176.14,359.57) .. controls (178.06,359.79) and (179.96,360.04) .. (181.83,360.34) ;  
\draw  [draw opacity=0][dash pattern={on 2.53pt off 3.02pt}][line width=2.25]  (477.4,355.99) .. controls (491.91,351.98) and (512.92,354.5) .. (533.84,363.98) .. controls (535.6,364.78) and (537.33,365.62) .. (539.01,366.48) -- (519.46,395.69) -- cycle ; \draw  [dash pattern={on 2.53pt off 3.02pt}][line width=2.25]  (477.4,355.99) .. controls (491.91,351.98) and (512.92,354.5) .. (533.84,363.98) .. controls (535.6,364.78) and (537.33,365.62) .. (539.01,366.48) ;  
\draw  [draw opacity=0][dash pattern={on 2.53pt off 3.02pt}][line width=2.25]  (480.97,458.43) .. controls (468.52,449.98) and (456.83,432.33) .. (450.51,410.25) .. controls (449.98,408.39) and (449.49,406.54) .. (449.06,404.69) -- (483.98,400.67) -- cycle ; \draw  [dash pattern={on 2.53pt off 3.02pt}][line width=2.25]  (480.97,458.43) .. controls (468.52,449.98) and (456.83,432.33) .. (450.51,410.25) .. controls (449.98,408.39) and (449.49,406.54) .. (449.06,404.69) ;  
\draw  [draw opacity=0][dash pattern={on 2.53pt off 3.02pt}][line width=2.25]  (386.01,90.45) .. controls (386.19,105.5) and (378.39,125.18) .. (363.89,142.98) .. controls (362.67,144.48) and (361.42,145.94) .. (360.16,147.35) -- (336.9,120.99) -- cycle ; \draw  [dash pattern={on 2.53pt off 3.02pt}][line width=2.25]  (386.01,90.45) .. controls (386.19,105.5) and (378.39,125.18) .. (363.89,142.98) .. controls (362.67,144.48) and (361.42,145.94) .. (360.16,147.35) ;  
\draw  [draw opacity=0][dash pattern={on 2.53pt off 3.02pt}][line width=2.25]  (299.49,147.22) .. controls (287.49,138.14) and (276.72,119.92) .. (271.54,97.54) .. controls (271.11,95.66) and (270.72,93.78) .. (270.38,91.92) -- (305.46,89.69) -- cycle ; \draw  [dash pattern={on 2.53pt off 3.02pt}][line width=2.25]  (299.49,147.22) .. controls (287.49,138.14) and (276.72,119.92) .. (271.54,97.54) .. controls (271.11,95.66) and (270.72,93.78) .. (270.38,91.92) ;  
\draw [line width=2.25]  [dash pattern={on 2.53pt off 3.02pt}]  (210,397.62) .. controls (234.01,422.2) and (262.28,430.2) .. (312.01,424.2) ;
\draw [line width=2.25]  [dash pattern={on 2.53pt off 3.02pt}]  (421.01,224.93) .. controls (423.01,214.93) and (400.01,167.93) .. (340.49,158.13) ;
\draw [line width=2.25]  [dash pattern={on 2.53pt off 3.02pt}]  (317.01,158.93) .. controls (277.01,160.93) and (239.01,201.93) .. (239.01,224.93) ;
\draw [line width=2.25]  [dash pattern={on 2.53pt off 3.02pt}]  (359,423.62) .. controls (391.01,429.2) and (441.01,406.2) .. (453.01,379.2) ;
\draw [line width=2.25]  [dash pattern={on 2.53pt off 3.02pt}]  (460,366.62) .. controls (471.01,326.2) and (460.01,304.2) .. (449.01,280.2) ;
\draw  [color={rgb, 255:red, 43; green, 85; blue, 234 }  ,draw opacity=1 ][fill={rgb, 255:red, 43; green, 85; blue, 234 }  ,fill opacity=1 ] (132.33,178.97) -- (160.94,161.46) -- (157.79,194.85) -- (153,174.19) -- cycle ;
\draw  [color={rgb, 255:red, 43; green, 85; blue, 234 }  ,draw opacity=1 ][fill={rgb, 255:red, 43; green, 85; blue, 234 }  ,fill opacity=1 ] (503.46,148.12) -- (510.3,180.95) -- (479.93,166.73) -- (501,169.19) -- cycle ;
\draw  [color={rgb, 255:red, 43; green, 85; blue, 234 }  ,draw opacity=1 ][fill={rgb, 255:red, 43; green, 85; blue, 234 }  ,fill opacity=1 ] (343.64,538.54) -- (314,522.83) -- (344.35,508.55) -- (329,523.19) -- cycle ;
\draw  [color={rgb, 255:red, 66; green, 188; blue, 31 }  ,draw opacity=1 ][fill={rgb, 255:red, 66; green, 188; blue, 31 }  ,fill opacity=1 ] (337.01,198.2) -- (316,184.43) -- (335.8,168.98) -- (326.2,184.01) -- cycle ;
\draw  [color={rgb, 255:red, 66; green, 188; blue, 31 }  ,draw opacity=1 ][line width=1.5]  (208.98,304.24) .. controls (208.98,237.69) and (262.94,183.73) .. (329.49,183.73) .. controls (396.05,183.73) and (450,237.69) .. (450,304.24) .. controls (450,370.8) and (396.05,424.75) .. (329.49,424.75) .. controls (262.94,424.75) and (208.98,370.8) .. (208.98,304.24) -- cycle ;
\draw  [color={rgb, 255:red, 245; green, 166; blue, 35 }  ,draw opacity=1 ][fill={rgb, 255:red, 245; green, 166; blue, 35 }  ,fill opacity=1 ] (48.86,343.04) -- (70.14,341.22) -- (62.87,361.3) -- (63,346.7) -- cycle ;
\draw  [color={rgb, 255:red, 245; green, 166; blue, 35 }  ,draw opacity=1 ][fill={rgb, 255:red, 245; green, 166; blue, 35 }  ,fill opacity=1 ] (83.86,384.04) -- (105.14,382.22) -- (97.87,402.3) -- (98,387.7) -- cycle ;
\draw  [color={rgb, 255:red, 245; green, 166; blue, 35 }  ,draw opacity=1 ][fill={rgb, 255:red, 245; green, 166; blue, 35 }  ,fill opacity=1 ] (199.6,62.83) -- (210.9,44.7) -- (222.62,62.56) -- (211,53.7) -- cycle ;
\draw  [color={rgb, 255:red, 245; green, 166; blue, 35 }  ,draw opacity=1 ][fill={rgb, 255:red, 245; green, 166; blue, 35 }  ,fill opacity=1 ] (258.6,61.83) -- (269.9,43.7) -- (281.62,61.56) -- (270,52.7) -- cycle ;
\draw  [color={rgb, 255:red, 245; green, 166; blue, 35 }  ,draw opacity=1 ][fill={rgb, 255:red, 245; green, 166; blue, 35 }  ,fill opacity=1 ] (397.42,40.58) -- (386.1,58.69) -- (374.4,40.82) -- (386,49.7) -- cycle ;
\draw  [color={rgb, 255:red, 245; green, 166; blue, 35 }  ,draw opacity=1 ][fill={rgb, 255:red, 245; green, 166; blue, 35 }  ,fill opacity=1 ] (452.42,40.58) -- (441.1,58.69) -- (429.4,40.82) -- (441,49.7) -- cycle ;
\draw  [color={rgb, 255:red, 245; green, 166; blue, 35 }  ,draw opacity=1 ][fill={rgb, 255:red, 245; green, 166; blue, 35 }  ,fill opacity=1 ] (590.94,324.09) -- (598.12,344.21) -- (576.84,342.28) -- (591,338.7) -- cycle ;
\draw  [color={rgb, 255:red, 245; green, 166; blue, 35 }  ,draw opacity=1 ][fill={rgb, 255:red, 245; green, 166; blue, 35 }  ,fill opacity=1 ] (567.94,374.09) -- (575.12,394.21) -- (553.84,392.28) -- (568,388.7) -- cycle ;
\draw  [color={rgb, 255:red, 245; green, 166; blue, 35 }  ,draw opacity=1 ][fill={rgb, 255:red, 245; green, 166; blue, 35 }  ,fill opacity=1 ] (520.97,505.3) -- (513.93,485.14) -- (535.19,487.21) -- (521,490.7) -- cycle ;
\draw  [color={rgb, 255:red, 245; green, 166; blue, 35 }  ,draw opacity=1 ][fill={rgb, 255:red, 245; green, 166; blue, 35 }  ,fill opacity=1 ] (482.97,545.3) -- (475.93,525.14) -- (497.19,527.21) -- (483,530.7) -- cycle ;
\draw  [color={rgb, 255:red, 245; green, 166; blue, 35 }  ,draw opacity=1 ][fill={rgb, 255:red, 245; green, 166; blue, 35 }  ,fill opacity=1 ] (179.22,536.04) -- (157.99,538.33) -- (164.82,518.09) -- (165,532.7) -- cycle ;
\draw  [color={rgb, 255:red, 245; green, 166; blue, 35 }  ,draw opacity=1 ][fill={rgb, 255:red, 245; green, 166; blue, 35 }  ,fill opacity=1 ] (145.22,497.04) -- (123.99,499.33) -- (130.82,479.09) -- (131,493.7) -- cycle ;
\draw  [color={rgb, 255:red, 189; green, 16; blue, 224 }  ,draw opacity=1 ][line width=1.5]  (134.98,397.62) .. controls (134.98,376.9) and (151.78,360.11) .. (172.49,360.11) .. controls (193.21,360.11) and (210,376.9) .. (210,397.62) .. controls (210,418.33) and (193.21,435.13) .. (172.49,435.13) .. controls (151.78,435.13) and (134.98,418.33) .. (134.98,397.62) -- cycle ;
\draw  [color={rgb, 255:red, 189; green, 16; blue, 224 }  ,draw opacity=1 ][line width=1.5]  (448.98,395.62) .. controls (448.98,374.9) and (465.78,358.11) .. (486.49,358.11) .. controls (507.21,358.11) and (524,374.9) .. (524,395.62) .. controls (524,416.33) and (507.21,433.13) .. (486.49,433.13) .. controls (465.78,433.13) and (448.98,416.33) .. (448.98,395.62) -- cycle ;
\draw  [color={rgb, 255:red, 189; green, 16; blue, 224 }  ,draw opacity=1 ][line width=1.5]  (290.98,122.62) .. controls (290.98,101.9) and (307.78,85.11) .. (328.49,85.11) .. controls (349.21,85.11) and (366,101.9) .. (366,122.62) .. controls (366,143.33) and (349.21,160.13) .. (328.49,160.13) .. controls (307.78,160.13) and (290.98,143.33) .. (290.98,122.62) -- cycle ;
\draw  [color={rgb, 255:red, 189; green, 16; blue, 224 }  ,draw opacity=1 ][fill={rgb, 255:red, 189; green, 16; blue, 224 }  ,fill opacity=1 ] (204.83,361.7) -- (206.79,381.86) -- (187.6,375.39) -- (201.5,375.2) -- cycle ;
\draw  [color={rgb, 255:red, 189; green, 16; blue, 224 }  ,draw opacity=1 ][fill={rgb, 255:red, 189; green, 16; blue, 224 }  ,fill opacity=1 ] (139.83,433.61) -- (138.39,413.41) -- (157.41,420.38) -- (143.5,420.2) -- cycle ;
\draw  [color={rgb, 255:red, 189; green, 16; blue, 224 }  ,draw opacity=1 ][fill={rgb, 255:red, 189; green, 16; blue, 224 }  ,fill opacity=1 ] (317.57,74.55) -- (335,84.87) -- (318.44,96.53) -- (326.5,85.2) -- cycle ;
\draw  [color={rgb, 255:red, 189; green, 16; blue, 224 }  ,draw opacity=1 ][fill={rgb, 255:red, 189; green, 16; blue, 224 }  ,fill opacity=1 ] (339.75,169.59) -- (322.02,159.79) -- (338.22,147.64) -- (330.5,159.2) -- cycle ;
\draw  [color={rgb, 255:red, 66; green, 188; blue, 31 }  ,draw opacity=1 ][fill={rgb, 255:red, 66; green, 188; blue, 31 }  ,fill opacity=1 ] (319.4,410.82) -- (340.4,424.59) -- (320.6,440.04) -- (330.2,425.01) -- cycle ;
\draw  [color={rgb, 255:red, 189; green, 16; blue, 224 }  ,draw opacity=1 ][fill={rgb, 255:red, 189; green, 16; blue, 224 }  ,fill opacity=1 ] (439.71,379.96) -- (457.82,370.88) -- (458.67,391.11) -- (453.5,378.2) -- cycle ;
\draw  [color={rgb, 255:red, 189; green, 16; blue, 224 }  ,draw opacity=1 ][fill={rgb, 255:red, 189; green, 16; blue, 224 }  ,fill opacity=1 ] (530.21,414.86) -- (512.51,424.71) -- (510.78,404.53) -- (516.5,417.2) -- cycle ;
\draw  [fill={rgb, 255:red, 0; green, 0; blue, 0 }  ,fill opacity=1 ] (129.3,352.57) -- (145.91,360.97) -- (132.91,374.29) -- (138.5,362.2) -- cycle ;
\draw  [fill={rgb, 255:red, 0; green, 0; blue, 0 }  ,fill opacity=1 ] (204.8,446.03) -- (186.9,451.12) -- (187.42,432.51) -- (191.5,445.2) -- cycle ;
\draw  [fill={rgb, 255:red, 0; green, 0; blue, 0 }  ,fill opacity=1 ] (203.39,319.17) -- (190.54,332.64) -- (181.55,316.34) -- (191.5,325.2) -- cycle ;
\draw  [fill={rgb, 255:red, 0; green, 0; blue, 0 }  ,fill opacity=1 ] (235.07,404.32) -- (245.58,419.68) -- (227.78,425.1) -- (238.5,417.2) -- cycle ;
\draw  [fill={rgb, 255:red, 0; green, 0; blue, 0 }  ,fill opacity=1 ] (270.4,125.04) -- (272.93,106.6) -- (289.76,114.55) -- (276.5,113.2) -- cycle ;
\draw  [fill={rgb, 255:red, 0; green, 0; blue, 0 }  ,fill opacity=1 ] (392.68,116.22) -- (376.25,124.96) -- (372.84,106.66) -- (379.5,118.2) -- cycle ;
\draw  [fill={rgb, 255:red, 0; green, 0; blue, 0 }  ,fill opacity=1 ] (376.66,185.49) -- (371.58,167.58) -- (390.19,168.12) -- (377.5,172.2) -- cycle ;
\draw  [fill={rgb, 255:red, 0; green, 0; blue, 0 }  ,fill opacity=1 ] (287.55,176.88) -- (269.12,179.42) -- (272.23,161.07) -- (274.5,174.2) -- cycle ;
\draw  [fill={rgb, 255:red, 0; green, 0; blue, 0 }  ,fill opacity=1 ] (407.9,403.65) -- (426.23,406.87) -- (417.65,423.39) -- (419.5,410.2) -- cycle ;
\draw  [fill={rgb, 255:red, 0; green, 0; blue, 0 }  ,fill opacity=1 ] (455.39,335.91) -- (462.87,318.87) -- (476.88,331.13) -- (464.5,326.2) -- cycle ;
\draw  [fill={rgb, 255:red, 0; green, 0; blue, 0 }  ,fill opacity=1 ] (519,344.96) -- (527.14,361.7) -- (508.73,364.43) -- (520.5,358.2) -- cycle ;
\draw  [fill={rgb, 255:red, 0; green, 0; blue, 0 }  ,fill opacity=1 ] (459.91,455.29) -- (461.65,436.76) -- (478.8,443.98) -- (465.5,443.2) -- cycle ;

\end{tikzpicture} & \hspace{0.75cm}

\includegraphics[width=0.55\linewidth]{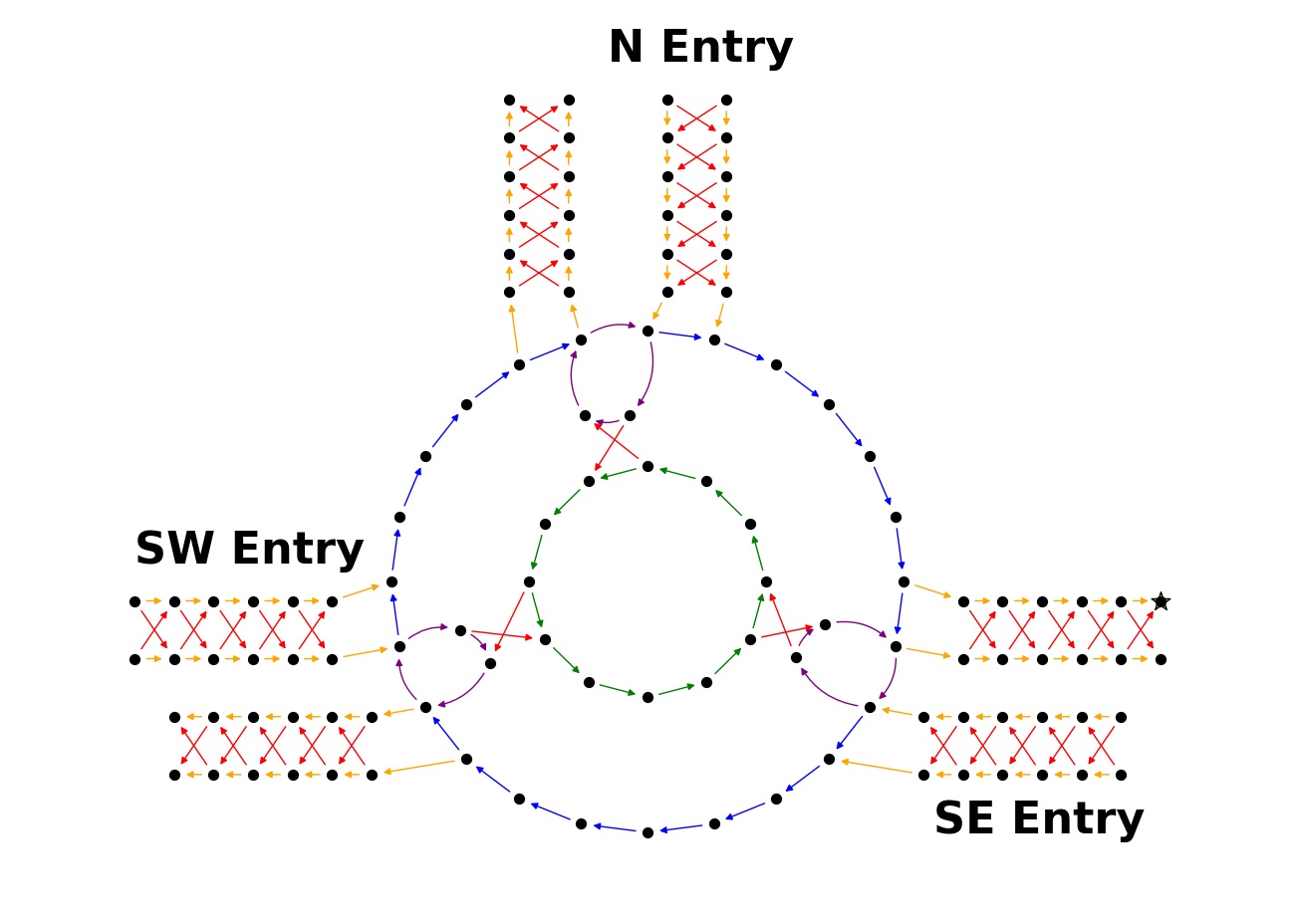}\\[3pt]
\mbox{\footnotesize (a)} & \mbox{\footnotesize(b)}
\end{array}
$
\caption{A magic roundabout example.  Panel (a): Schematic of roundabout network with five connected traffic circles. Arrows indicate the direction of traffic flow and dashed edges indicate available transitions between roundabouts. We assume that the vehicle drives on the left-hand side of the road as is customary in the United Kingdom. Panel (b): Lane-level road network representation of the roundabout network with all possible transitions from each node shown. 
The arrows show available node-to-node transitions, with red arrows indicating stochastic lane changes while all other transitions are deterministic.
The destination $\bt$ is along the southeastern exit and identified by the black star.\label{fig:roundabout_diagram}}
\end{figure}

We consider one such example based on a network (shown in Figure  \ref{fig:roundabout_diagram}) of five interconnected roundabouts 
 -- one outer (blue road segments, moving clockwise), one inner (green, moving counterclockwise), and three miniature roundabouts (purple) which allow vehicles to switch between the outer and the inner roundabouts. To avoid terminological confusion, we will further refer to all of them as ``rings,'' reserving the term ``Roundabout'' for their interconnected combination.
We assume there are three roads running into/out of this Roundabout and that the target is in the left lane of the southeastern exit, identified with a black star in
Figure \ref{fig:roundabout_diagram}(b).
For a driver approaching the Roundabout, the main strategic decision is whether it is worth changing lanes before the entry point and, if the answer is yes, how the urgency of LSMs should vary as we get closer.  (Entering from the left lane leads to a clockwise trajectory; entering on the right is worthwhile if you want to transition to the inner roundabout and travel counterclockwise at least at first.)  The basic tradeoff is usually between following the most direct route to $\bt$ and limiting exposure to heavy congestion.

As in the previous example, wherever a lane change is possible (as indicated by red arrows in Figure \ref{fig:roundabout_diagram}(b)), 
the vehicle has access to a complete urgency spectrum of LSMs, but here our transition cost is quadratic in $p$:
\begin{equation}\label{eqn:cost_fn_roundabout}
	K(p) = \beta(\x)p^2 + \gamma(\x) + [f]
\end{equation}
where $\beta(\x), \gamma(\x) > 0 $ reflect the congestion present within the road segment (e.g., entryways, outer ring, inner ring, mini-ring) at $\x.$ The positive constant $f$ is
only added at the entrance-nodes from each of the roads (into the mini and outer rings respectively). 
These constants are chosen to reflect any waiting that the vehicle must endure as a result of current traffic conditions before being able to enter the roundabout. 
The causality condition established in Theorem \ref{theorem:gen_caus_condition} is satisfied when 
$\beta(\x) \leq \gamma(\x) +  [f]$ at all $\x \in X.$ 
Our goal is to highlight the effect of traffic congestion on the outer and inner rings; so, for the sake of simplicity,
we assume that $\beta(\x) = \gamma(\x) = 1$ on all two-lane roads leading to/from the Roundabout and on the three purple mini-rings.  
The resulting $K(p)$ is shown in Figure \ref{fig:ex3}(a).
For the Roundabout entry nodes, we also set $f = 5$ in all experiments.

For all deterministic transitions (shown by orange, blue, purple, and green arcs), we assume the stay-in-lane cost $K = \gamma(\x)$ based on the local level of congestion.
We focus on its effect on the optimal policy for the cars approaching along the northern and southwestern roads.
In the first experiment, the inner ring is much more congested ($K = \gamma(\x) = 6.8$ on green arcs) than the outer ring ($K = 3$ on blue arcs).
As shown in Figure \ref{fig:ex3}(b), it is optimal for the approaching cars to enter the Roundabout from the left lane and travel along the outer ring. 
Those cars that approach the Roundabout from the Southwest in the right lane will attempt LSMs to the left with an increasing urgency as they get closer.
In the second experiment, the situation is essentially reversed with $K = \gamma(\x) = 3$ on the inner ring and
$K =
5.2$ on the outer. 
As shown in Figure \ref{fig:ex3}(c), the approaching cars now have a clear preference to enter through the right lane;
this is largely true even for cars coming from the North, for which the counter-clockwise path (through the inner ring) is preferred despite being longer.

\begin{figure*}[t]
\centering
$
\arraycolsep=1pt\def\arraystretch{0.1}
\begin{array}{ccc}
\includegraphics[width = 0.33\textwidth]{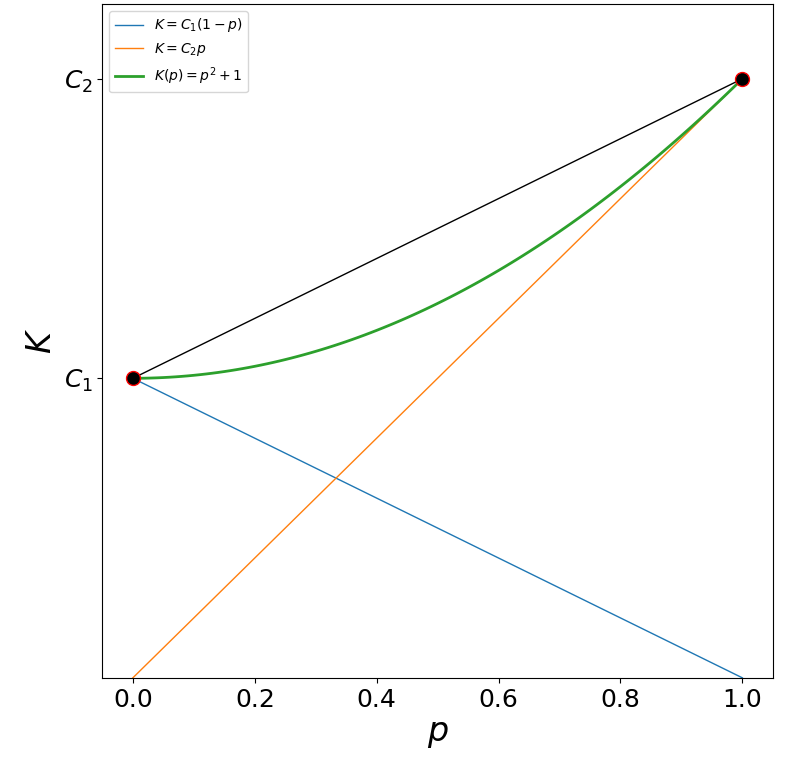} &
\includegraphics[width = 0.33\textwidth]{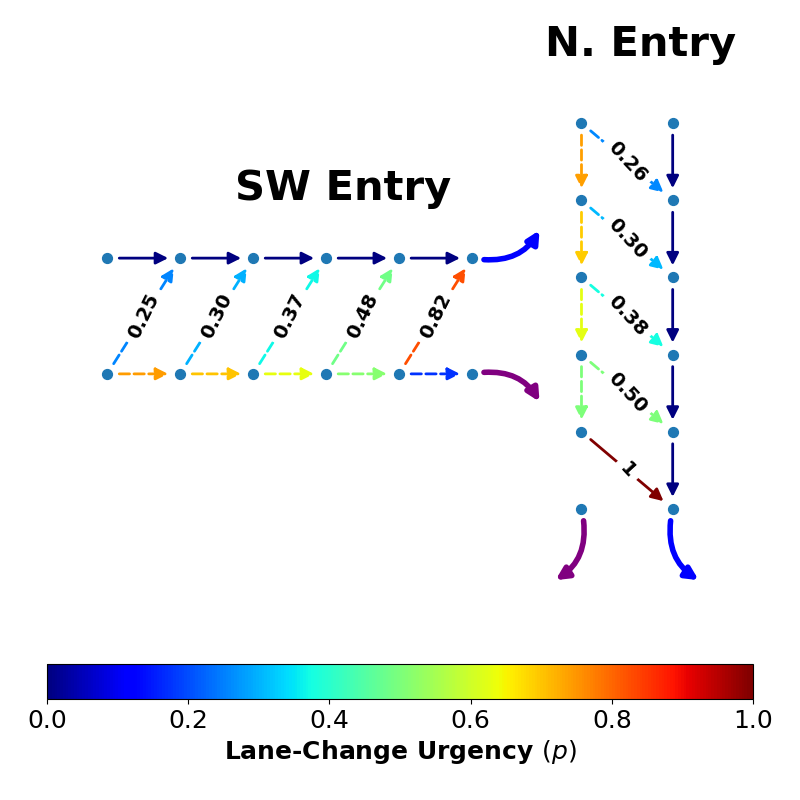} &
\includegraphics[width = 0.33\textwidth]{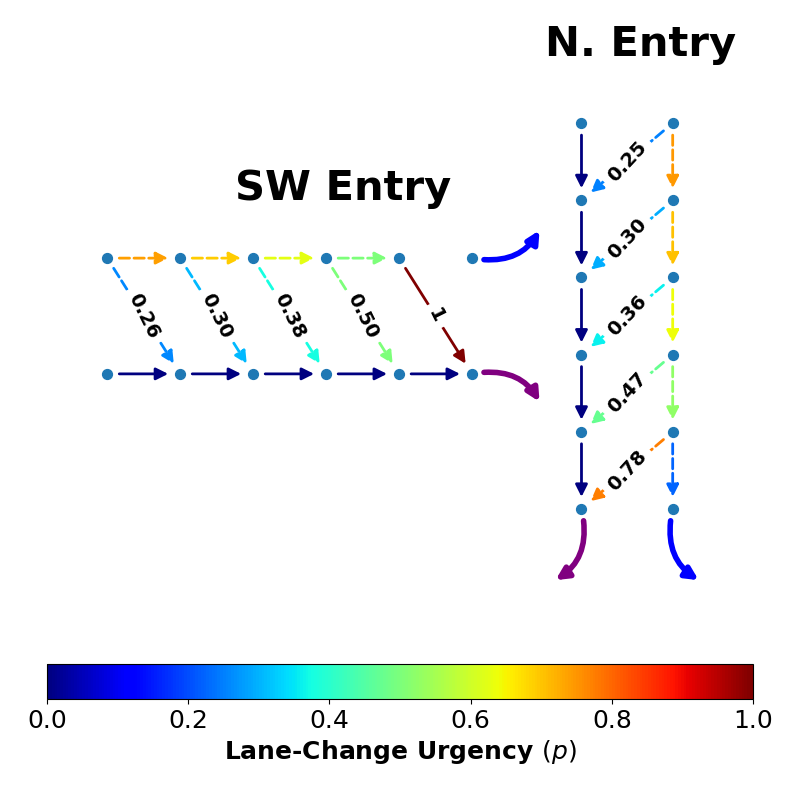} \\[-2pt]
\mbox{\footnotesize(a)} & \mbox{\footnotesize(b)} & \mbox{\footnotesize(c)}
\end{array}
$
\caption{Optimal lane-change policy when approaching a magic roundabout.  Panel (a): $K(p)$ along the two-lane entryways to the roundabout network. Here, $\beta(\x) = \gamma(\x) = 1$, and $K(p)$ is only monotone causal by Theorem \ref{theorem:gen_caus_condition}. Panels (b) and (c): optimal STPs at the southwestern and northern entryways when the inner ring has the highest congestion (panel (b)), and the outer ring has the highest congestion (panel (c)). Thick blue arrows indicate a direct entry to the outer ring while thick purple arrows indicate a direct entry into a mini ring. In both cases, vehicles tackle the tradeoff between taking a more direct route to $\bt$ and the amount of congestion they are willing to encounter. \label{fig:ex3}}
\end{figure*}
%
%
 \section{CONCLUSIONS}
We have introduced a class of Opportunistically Stochastic Shortest Path (OSSP) problems and proved several sufficient ``monotone causality'' conditions 
to guarantee the applicability of efficient label-setting methods. The approach has important applications both in discrete and continuous optimization.
For example, given an anisotropic inhomogeneous speed function $f$ for motion in a continuous domain, the deterministic time-to-target minimized over all feasible paths can be found as a viscosity solution of the corresponding stationary Hamilton-Jacobi-Bellman PDE.  A first-order accurate semi-Lagrangian discretization of that PDE can be re-interpreted as an OSSP, and our MC criteria can be then used to check which discretization stencils are compatible with Dijkstra's method (in 2D and 3D) and with Dial's method (in 2D only).
Importantly, the conditions we developed are expressed in terms of simple geometric properties of the anisotropic speed profile $\Vf.$
However, our current analysis does not provide any guidance for finding the best MC-causal stencil for each  $\Vf$ (since causal properties need to be balanced against a possible increase in local truncation errors).  Restricted versions of this problem are considered in  
\cite{mirebeau2014anisotropic, mirebeau2014efficient, desquilbet2021single}, but many aspects still remain open for general speed profiles.  This balancing act is even more delicate with Dial's method, where increasing the angular resolution of a stencil decreases that stencil's spatial locality but allows using a larger bin width $\Delta.$

In the discrete setting, we have demonstrated the usefulness of OSSPs in optimizing the lane change level routing of autonomous vehicles.  Extending the prior work \cite{jones2022lane}, we showed that Dijkstra's and Dial's methods are applicable to a much broader class of vehicle routing models that include multiple intermediate urgency levels (or even a continuous urgency spectrum) of lane change maneuvers and a variety of cost functions.   We note that the same approach is also useful in a semi-autonomous context; 
our Strategic/Tactical Plan could be also used by human drivers and 
the recommended lane change urgency levels could be communicated (e.g., as indicator bars of varying length or color) through assistive navigation hardware or software.

OSSP models capture the inherent uncertainty of lane change maneuvers, but if the traffic conditions significantly change from those used to formulate the problem, the entire optimal policy has to be recomputed on the fly. 
The usability of label-setting methods makes such occasional online replanning possible. 
But both Dijkstra's and Dial's method compute the optimal policy for reaching the specific target $\bm{t}$ from {\em each} starting node of the road network.
However, most of that network is probably irrelevant for a vehicle starting at a specific location $\bm{s}$.
In deterministic path planning, such single source / single target problems are often solved by an even more efficient A* method \cite{hart1968formal}, 
which uses a ``consistent heuristic'' to restrict the computations to a smaller (implicitly defined) neighborhood of the optimal path from $\bm{s}$ to $\bm{t}.$
While such a consistent heuristic is unavailable for general SSPs, 
it can be constructed in OSSPs with positive $\underline{C} = \min\limits_{\x_i \in X} \min\limits_{\ba \in A(\x_i)} C(\x_i, \ba).$
This is essentially the approach explored in \cite{jones2023lanelevel}.  
But since realistic road networks include short lane segments, this $\underline{C}$ is usually quite small, making the consistent heuristic very conservative
and yielding little computational savings.
We believe that a more promising approach is to explore the use of inconsistent heuristics (ideally, with a bound on suboptimality of resulting routing policies).
It might be possible to use the ``asymptotically causal'' domain restriction techniques developed for discretizations of HJB equations \cite{clawson2014causal},
with the true optimality of trajectories recovered only in the limit (under the grid refinement).

Another interesting direction for future work is to find criteria for applicability of label-setting to ``stochastic shortest path games,'' in which 
the probability distribution over the successor nodes at each stage depends on the actions chosen by two antagonistic players \cite{patek1999stochastic}.
We note that some results on the applicability of Dijkstra's method to {\em deterministic} games on graphs have been developed in the last ten years 
\cite{bardi2016dijkstra, bertsekas2019robust}, but the case of general stochastic games on graphs remains
open.\label{section:conclusions}

\section*{Acknowledgments}
The authors are grateful to Jur van den Berg for
sparking their interest in lane-change-level vehicle routing problems.
The authors also thank Jean-Marie Mirebeau for a helpful discussion of references \cite{mirebeau2014anisotropic, mirebeau2014efficient, desquilbet2021single}
and for pointing out the connections which led to
our 
Remark \ref{rem:another_mirebeau_connection}
and
Proposition \ref{prop:3D_tangent_cond}.
This work was supported in part by the first author's NDSEG
Fellowship,
by the NSF DMS (awards 1645643 and 2111522),
and by the AFOSR (award FA9550-22-1-0528).

\bibliographystyle{informs2014} 
\bibliography{bibl} 

\begin{thebibliography}{58}
\providecommand{\natexlab}[1]{#1}
\providecommand{\url}[1]{\texttt{#1}}
\providecommand{\urlprefix}{URL }

\bibitem[{Ahmed(1999)}]{ahmed1999modeling}
Ahmed KI (1999) \emph{Modeling {D}rivers' {A}cceleration and {L}ane {C}hanging
  {B}ehavior}. Ph.D. thesis, Massachusetts Institute of Technology.

\bibitem[{Ahuja et~al.(1993)Ahuja, Magnanti, \protect\BIBand{}
  Orlin}]{ahuja1993}
Ahuja R, Magnanti T, Orlin J (1993) \emph{Network {F}lows} (Upper Saddle River,
  NJ: Prentice Hall, Inc.).

\bibitem[{Alton \protect\BIBand{} Mitchell(2009)}]{alton2009fast}
Alton K, Mitchell IM (2009) Fast {M}arching {M}ethods for {S}tationary
  {H}amilton-{J}acobi {E}quations with {A}xis-{A}ligned {A}nisotropy.
  \emph{SIAM Journal on Numerical Analysis} 47(1):363--385.

\bibitem[{Alton \protect\BIBand{} Mitchell(2012)}]{alton2012ordered}
Alton K, Mitchell IM (2012) An {O}rdered {U}pwind {M}ethod with {P}recomputed
  {S}tencil and {M}onotone {N}ode {A}cceptance for {S}olving {S}tatic {C}onvex
  {H}amilton-{J}acobi {E}quations. \emph{Journal of Scientific Computing}
  51(2):313--348.

\bibitem[{Bak et~al.(2010)Bak, McLaughlin, \protect\BIBand{} Renzi}]{Renzi}
Bak S, McLaughlin J, Renzi D (2010) Some {I}mprovements for the {F}ast
  {S}weeping {M}ethod. \emph{SIAM Journal on Scientific Computing}
  32(5):2853--2874.

\bibitem[{Bardi et~al.(1997)Bardi, Dolcetta et~al.}]{bardi1997optimal}
Bardi M, Dolcetta IC, et~al. (1997) \emph{Optimal {C}ontrol and {V}iscosity
  {S}olutions of {H}amilton-{J}acobi-{B}ellman {E}quations}, volume~12
  (Springer).

\bibitem[{Bardi \protect\BIBand{}
  Maldonado~L{\'o}pez(2016)}]{bardi2016dijkstra}
Bardi M, Maldonado~L{\'o}pez JP (2016) A {D}ijkstra-type algorithm for dynamic
  games. \emph{Dynamic Games and Applications} 6(3):263--276.

\bibitem[{Bellman(1957)}]{Bellman_DP_book}
Bellman R (1957) \emph{Dynamic Programming} (Princeton University Press).

\bibitem[{Bertsekas(1993)}]{Bertsekas_SLF}
Bertsekas DP (1993) A simple and fast label correcting algorithm for shortest
  paths. \emph{Networks} 23(8):703--709.

\bibitem[{Bertsekas(2001)}]{Bertsekas_DPbook}
Bertsekas DP (2001) \emph{Dynamic {P}rogramming and {O}ptimal {C}ontrol},
  volume I and II (Boston, MA: Athena Scientific), 3rd edition.

\bibitem[{Bertsekas(2019)}]{bertsekas2019robust}
Bertsekas DP (2019) Robust shortest path planning and semicontractive dynamic
  programming. \emph{Naval Research Logistics (NRL)} 66(1):15--37.

\bibitem[{Bertsekas et~al.(1996)Bertsekas, Guerriero, \protect\BIBand{}
  Musmanno}]{Bertsekas_LLL}
Bertsekas DP, Guerriero F, Musmanno R (1996) Parallel {A}synchronous
  {L}abel-{C}orrecting {M}ethods for {S}hortest {P}aths. \emph{Journal of
  Optimization Theory and Applications} 88(2):297--320.

\bibitem[{Bertsekas \protect\BIBand{} Tsitsiklis(1991)}]{bertsekas1991analysis}
Bertsekas DP, Tsitsiklis JN (1991) An {A}nalysis of {S}tochastic {S}hortest
  {P}ath {P}roblems. \emph{Mathematics of Operations Research} 16(3):580--595.

\bibitem[{Bornemann \protect\BIBand{} Rasch(2006)}]{BorRasch}
Bornemann F, Rasch C (2006) Finite-element discretization of static
  {Hamilton-Jacobi} equations based on a local variational principle.
  \emph{Computing and Visualization in Science} 9:57--69.

\bibitem[{Bou\'e \protect\BIBand{} Dupuis(1999)}]{BoueDupuis}
Bou\'e M, Dupuis P (1999) {M}arkov {C}hain {A}pproximations for {D}eterministic
  {C}ontrol {P}roblems with {A}ffine {D}ynamics and {Q}uadratic {C}ost in the
  {C}ontrol. \emph{SIAM J. Numer. Anal} 36(3):667--695.

\bibitem[{Cameron(2012)}]{cameron2012}
Cameron M (2012) Finding the quasipotential for nongradient {SDEs}.
  \emph{Physica D: Nonlinear Phenomena} 241(18):1532--1550.

\bibitem[{Chacon \protect\BIBand{} Vladimirsky(2012)}]{ChacVlad1}
Chacon A, Vladimirsky A (2012) Fast two-scale methods for eikonal equations.
  \emph{SIAM Journal on Scientific Computing} 34(2):A547--A578.

\bibitem[{Chacon \protect\BIBand{} Vladimirsky(2015)}]{ChacVlad2}
Chacon A, Vladimirsky A (2015) A parallel two-scale method for eikonal
  equations. \emph{SIAM Journal on Scientific Computing} 37(1):A156--A180.

\bibitem[{Chen et~al.(2007)Chen, Zhao, Krishnamurthy, \protect\BIBand{}
  Djonin}]{chen2007transmission}
Chen Y, Zhao Q, Krishnamurthy V, Djonin D (2007) Transmission {S}cheduling for
  {O}ptimizing {S}ensor {N}etwork {L}ifetime: A {S}tochastic {S}hortest {P}ath
  {A}pproach. \emph{IEEE Transactions on Signal Processing} 55(5):2294--2309.

\bibitem[{Clawson et~al.(2014)Clawson, Chacon, \protect\BIBand{}
  Vladimirsky}]{clawson2014causal}
Clawson Z, Chacon A, Vladimirsky A (2014) Causal {D}omain {R}estriction for
  {E}ikonal {E}quations. \emph{SIAM Journal on Scientific Computing}
  36(5):A2478--A2505.

\bibitem[{Dai \protect\BIBand{} Goldsmith(2007)}]{DaiGoldsmith2007}
Dai P, Goldsmith J (2007) Topological {V}alue {I}teration {A}lgorithm for
  {Markov Decision Processes}. \emph{Proceedings of the 20th International
  Joint Conference on Artifical Intelligence}, 1860–1865, IJCAI'07 (San
  Francisco, CA, USA: Morgan Kaufmann Publishers Inc.).

\bibitem[{Dai et~al.(2009)Dai, Weld et~al.}]{DaiGoldsmith2009}
Dai P, Weld D, et~al. (2009) Focused {T}opological {V}alue {I}teration.
  \emph{Proceedings of the International Conference on Automated Planning and
  Scheduling}, volume~19, 82--89.

\bibitem[{Desquilbet et~al.(2021)Desquilbet, Cao, Cupillard, M{\'e}tivier,
  \protect\BIBand{} Mirebeau}]{desquilbet2021single}
Desquilbet F, Cao J, Cupillard P, M{\'e}tivier L, Mirebeau JM (2021) Single
  {P}ass {C}omputation of {F}irst {S}eismic {W}ave {T}ravel {T}ime in {T}hree
  {D}imensional {H}eterogeneous {M}edia with {G}eneral {A}nisotropy.
  \emph{Journal of Scientific Computing} 89:1--37.

\bibitem[{Dial(1969)}]{dial1969algorithm}
Dial RB (1969) {A}lgorithm 360: Shortest-path forest with topological ordering
  [h]. \emph{Communications of the {A}{C}{M}} 12(11):632--633.

\bibitem[{Dijkstra et~al.(1959)}]{dijkstra1959note}
Dijkstra EW, et~al. (1959) A {N}ote on {T}wo {P}roblems in {C}onnexion with
  {G}raphs. \emph{Numerische {M}athematik} 1(1):269--271.

\bibitem[{Fainberg(1976)}]{fainberg1976}
Fainberg EA (1976) On controlled finite state {Markov} processes with compact
  control sets. \emph{Theory of Probability \& Its Applications}
  20(4):856--862.

\bibitem[{Falcone \protect\BIBand{} Ferretti(2014)}]{Falcone_book}
Falcone M, Ferretti R (2014) \emph{Semi-{L}agrangian {A}pproximation {S}chemes
  for {L}inear and {H}amilton-{J}acobi {E}quations}, volume 133 (SIAM).

\bibitem[{Feinberg(1992)}]{feinberg1992markov}
Feinberg EA (1992) A markov decision model of a search process.
  \emph{Contemporary Mathematics} 125:87--96.

\bibitem[{Fitch et~al.(2009)}]{fitch2009analysis}
Fitch G, et~al. (2009) Analysis of {L}ane-{C}hange {C}rashes and
  {N}ear-{C}rashes. \emph{US Department of Transportation, National Highway
  Traffic Safety Administration} .

\bibitem[{Gipps(1986)}]{gipps1986model}
Gipps PG (1986) A {M}odel for the {S}tructure of {L}ane-{C}hanging {D}ecisions.
  \emph{Transportation Research Part B: Methodological} 20(5):403--414.

\bibitem[{Glover et~al.(1986)Glover, Glover, \protect\BIBand{}
  Klingman}]{glover1986threshold}
Glover F, Glover R, Klingman D (1986) The threshold shortest path algorithm.
  \emph{Networks} 14(1):12--37.

\bibitem[{Gonzales \protect\BIBand{} Rofman(1985)}]{GonzalezRofman}
Gonzales R, Rofman E (1985) On deterministic {C}ontrol {P}roblems: an
  {A}pproximate {P}rocedure for the {O}ptimal {C}ost, {I}, the {S}tationary
  {P}roblem. \emph{SIAM J. Control Optim.} 23(2):242--266.

\bibitem[{Hart et~al.(1968)Hart, Nilsson, \protect\BIBand{}
  Raphael}]{hart1968formal}
Hart PE, Nilsson NJ, Raphael B (1968) A {F}ormal {B}asis for the {H}euristic
  {D}etermination of {M}inimum {C}ost {P}aths. \emph{IEEE transactions on
  Systems Science and Cybernetics} 4(2):100--107.

\bibitem[{Hong et~al.(2021)Hong, Lee, Huang, Khonji, Alyassi, \protect\BIBand{}
  Williams}]{hong2021anytime}
Hong S, Lee SU, Huang X, Khonji M, Alyassi R, Williams BC (2021) An {A}nytime
  {A}lgorithm for {C}hance {C}onstrained {S}tochastic {S}hortest {P}ath
  {P}roblems and its {A}pplication to {A}ircraft {R}outing. \emph{2021 IEEE
  International Conference on Robotics and Automation (ICRA)}, 475--481 (IEEE).

\bibitem[{Howard(1960)}]{howard1960dynamic}
Howard RA (1960) \emph{Dynamic programming and {M}arkov processes.} (John
  Wiley).

\bibitem[{Jones et~al.(2022)Jones, Haas-Heger, \protect\BIBand{} van~den
  Berg}]{jones2022lane}
Jones M, Haas-Heger M, van~den Berg J (2022) Lane-{L}evel {R}oute {P}lanning
  for {A}utonomous {V}ehicles. \emph{Algorithmic Foundations of Robotics XV:
  Proceedings of the Fifteenth Workshop on the Algorithmic Foundations of
  Robotics}, 312--327 (Springer).

\bibitem[{Jones et~al.(2024)Jones, Haas-Heger, \protect\BIBand{} van~den
  Berg}]{jones2023lanelevel}
Jones M, Haas-Heger M, van~den Berg J (2024) Lane-level route planning for
  autonomous vehicles. \emph{The International Journal of Robotics Research}
  43(9):1425--1440.

\bibitem[{Kim(2001)}]{KimGMM}
Kim S (2001) An {O(N)} {L}evel {S}et {M}ethod for {E}ikonal {E}quations.
  \emph{SIAM J. Sci. Comp.} 22:2178--2193.

\bibitem[{Kimmel(1998)}]{kimmel1998fast}
Kimmel R (1998) Fast {M}arching {M}ethods on {T}riangulated {D}omains.
  \emph{Proc. Nat. Acad. Sci.} 95:8431--8435.

\bibitem[{Kushner \protect\BIBand{} Dupuis(1992)}]{KushnerDupuis}
Kushner HJ, Dupuis PG (1992) \emph{Numerical {M}ethods for {S}tochastic
  {C}ontrol {P}roblems in {C}ontinuous {T}ime} (New York: Academic Press).

\bibitem[{Michon(1985)}]{michon1985critical}
Michon JA (1985) A {C}ritical {V}iew of {D}river {B}ehavior {M}odels: {W}hat
  {D}o {W}e {K}now, {W}hat {S}hould {W}e {D}o? \emph{Human Behavior and Traffic
  Safety}, 485--524 (Springer).

\bibitem[{Mirebeau(2014{\natexlab{a}})}]{mirebeau2014anisotropic}
Mirebeau JM (2014{\natexlab{a}}) Anisotropic {F}ast-{M}arching on {C}artesian
  {G}rids {U}sing {L}attice {B}asis {R}eduction. \emph{SIAM Journal on
  Numerical Analysis} 52(4):1573--1599.

\bibitem[{Mirebeau(2014{\natexlab{b}})}]{mirebeau2014efficient}
Mirebeau JM (2014{\natexlab{b}}) Efficient fast marching with {Finsler}
  metrics. \emph{Numer. math.} 126(3):515--557.

\bibitem[{Osher \protect\BIBand{} Fedkiw(2001)}]{osher2001level}
Osher S, Fedkiw RP (2001) Level {S}et {M}ethods: {A}n {O}verview and {S}ome
  {R}ecent {R}esults. \emph{Journal of Computational physics} 169(2):463--502.

\bibitem[{Pape(1974)}]{pape1974implementation}
Pape U (1974) Implementation and efficiency of {Moore}-algorithms for the
  shortest route problem. \emph{Mathematical programming} 7:212--222.

\bibitem[{Patek \protect\BIBand{} Bertsekas(1999)}]{patek1999stochastic}
Patek SD, Bertsekas DP (1999) Stochastic shortest path games. \emph{SIAM
  Journal on Control and Optimization} 37(3):804--824.

\bibitem[{Piegl \protect\BIBand{} Tiller(1996)}]{piegl1996nurbs}
Piegl L, Tiller W (1996) \emph{The {NURBS} {B}ook} (Springer Science \&
  Business Media).

\bibitem[{Polymenakos et~al.(1998)Polymenakos, Bertsekas, \protect\BIBand{}
  Tsitsiklis}]{PolyBerTsi}
Polymenakos LC, Bertsekas DP, Tsitsiklis JN (1998) Implementation of efficient
  algorithms for globally optimal trajectories. \emph{IEEE Transactions on
  Automatic Control} 43(2):278--283.

\bibitem[{Sethian(1996{\natexlab{a}})}]{sethian1996fast}
Sethian JA (1996{\natexlab{a}}) A {F}ast {M}arching {L}evel {S}et {M}ethod for
  {M}onotonically {A}dvancing {F}ronts. \emph{Proceedings of the National
  Academy of Sciences} 93(4):1591--1595.

\bibitem[{Sethian(1996{\natexlab{b}})}]{SethBook2}
Sethian JA (1996{\natexlab{b}}) \emph{Level {S}et {M}ethods and {F}ast
  {M}arching {M}ethods: {E}volving {I}nterfaces in {C}omputational {G}eometry,
  {F}luid {M}echanics, {C}omputer {V}ision and {M}aterials Sciences} (Cambridge
  University Press).

\bibitem[{Sethian(1999)}]{sethian1999fast}
Sethian JA (1999) Fast {M}arching {M}ethods. \emph{SIAM review} 41(2):199--235.

\bibitem[{Sethian \protect\BIBand{} Vladimirsky(2000)}]{sethian2000fast}
Sethian JA, Vladimirsky A (2000) Fast methods for the eikonal and related
  {H}amilton--{J}acobi equations on unstructured meshes. \emph{Proceedings of
  the National Academy of Sciences} 97(11):5699--5703.

\bibitem[{Sethian \protect\BIBand{} Vladimirsky(2001)}]{sethian2001ordered}
Sethian JA, Vladimirsky A (2001) Ordered upwind methods for static
  {H}amilton-{J}acobi equations. \emph{Proceedings of the National Academy of
  Sciences} 98(20):11069--11074.

\bibitem[{Sethian \protect\BIBand{} Vladimirsky(2003)}]{SethVlad3}
Sethian JA, Vladimirsky A (2003) Ordered {U}pwind {M}ethods for {S}tatic
  {H}amilton--{J}acobi {E}quations: {T}heory and {A}lgorithms. \emph{SIAM
  Journal on Numerical Analysis} 41(1):325--363.

\bibitem[{Tsitsiklis(1994)}]{tsitsiklis1994efficient}
Tsitsiklis J (1994) Efficient {A}lgorithms for {G}lobally {O}ptimal
  {T}rajectories. \emph{Proceedings of 1994 33rd IEEE Conference on Decision
  and Control}, volume~2, 1368--1373 (IEEE).

\bibitem[{Tsitsiklis(1995)}]{tsitsiklis1995efficient}
Tsitsiklis JN (1995) Efficient {A}lgorithms for {G}lobally {O}ptimal
  {T}rajectories. \emph{IEEE transactions on Automatic Control}
  40(9):1528--1538.

\bibitem[{Van~Nunen(1976)}]{van1976set}
Van~Nunen J (1976) A {S}et of {S}uccessive {A}pproximation {M}ethods for
  {D}iscounted {Markovian} {D}ecision {P}roblems. \emph{Zeitschrift fuer
  operations research} 20:203--208.

\bibitem[{Vladimirsky(2008)}]{vladimirsky2008label}
Vladimirsky A (2008) Label-{S}etting {M}ethods for {M}ultimode {S}tochastic
  {S}hortest {P}ath {P}roblems on {G}raphs. \emph{Mathematics of Operations
  Research} 33(4):821--838.

\end{thebibliography}


\end{document}